\documentclass[11pt]{article}

\usepackage{verbatim,latexsym,amsfonts,amsmath,amssymb,graphicx,fancyhdr,hyperref,
}
\usepackage{appendix,latexsym,amsfonts,amsmath,amssymb,graphicx,hyperref,amsthm,soul,verbatim,authblk,xcolor}
\usepackage[framemethod=tikz]{mdframed}

\newcommand{\EE}{\mathbb{ E}}
\newcommand{\PP}{\mathbb{P}}

\newcommand{\R}{\mathbb{R}}
\newcommand{\C}{\mathbb{C}}
\newcommand{\Q}{\mathbb{Q}}

\newcommand{\HH}{\mathbb{H}}
\newcommand{\N}{\mathbb{N}}
\newcommand{\D}{\mathbb{D}}
\newcommand{\Z}{\mathbb{Z}}

\newcommand{\TT}{\mathbb{T}}

\newcommand{\pa}{\partial}

\newcommand{\F}{{\cal F}}

\def\til{\widetilde}
\def\ha{\widehat}
\def\sem{\setminus}
\def\lin{\overline}
\def\ulin{\underline}

\def\luto{\stackrel{\rm l.u.}{\longrightarrow}}

\def\L{{\cal L}}
\def\conf{\stackrel{\rm Conf}{\twoheadrightarrow}}

\DeclareMathOperator{\rad}{rad}
 
\DeclareMathOperator{\dist}{dist} \DeclareMathOperator{\dcap}{dcap}
\DeclareMathOperator{\hcap}{hcap} \DeclareMathOperator{\id}{id}
\DeclareMathOperator{\Imm}{Im }

\DeclareMathOperator{\mA}{m} 
 
 \DeclareMathOperator{\cc}{c}
\DeclareMathOperator{\bb}{b} \DeclareMathOperator{\doub}{doub}
\DeclareMathOperator{\Hull}{Hull}  \DeclareMathOperator{\ii}{i}

\theoremstyle{plain}
\newtheorem{Theorem}{Theorem}[section]
\newtheorem{Lemma}[Theorem]{Lemma}
\newtheorem{Corollary}[Theorem]{Corollary}
\newtheorem{Proposition}[Theorem]{Proposition}
\theoremstyle{definition}
\newtheorem{Definition}[Theorem]{Definition}
\newtheorem{Remark}[Theorem]{Remark}
\numberwithin{equation}{section}
\newcommand{\BGE}{\begin{equation}}
\newcommand{\BGEN}{\begin{equation*}}
\newcommand{\EDE}{\end{equation}}
\newcommand{\EDEN}{\end{equation*}}
\def\dto{\stackrel{\rm Cara}{\longrightarrow}}

\begin{document}
\title{Green's function for cut points of  chordal SLE\\ attached with boundary arcs}
\author{Dapeng Zhan\thanks{Michigan State University, zhan@math.msu.edu}
}
\maketitle


\begin{abstract}
{A technique of two-curve Green's function is used to study the Green's function of cut points of chordal SLE$_\kappa$ for $\kappa\in(4,8)$. In order to apply the technique, we take the union of the SLE curve with two open boundary arcs, which share two boundary points other than the endpoints of the SLE curve. The Green's function of interest is, for any $z_0$ in the domain, the limit as $r\downarrow 0$ of the $r^{-\alpha}$ times the probability that $\{|z-z_0|<r\}$ contains a cut point of the above union, where $\alpha=\frac 38\kappa-1$. We prove the limit exists, obtain an exact formula  up to a multiplicative constant, and derive a rate of convergence.
 }
\end{abstract}


\section{Introduction and Main Result}
The Schramm-Loewner evolution (SLE), first introduced by Oded Schramm in 1999 (\cite{S-SLE}), is a one-parameter ($\kappa\in(0,\infty)$) family of measures on plane curves. 
The SLE with different parameters have been proved to be scaling limits of various lattice models. We refer the reader to Lawler's textbook \cite{Law-SLE} for basic properties of SLE.

An SLE$_\kappa$ curve is simple for $\kappa\in(0,4]$, space-filling for $\kappa\in[8,\infty)$, and non-simple and non-space-filling for $\kappa\in(4,8)$ (cf.\ \cite{RS}). The Hausdorff dimension of an SLE$_\kappa$ curve is $\min\{1+\frac\kappa 8,2\}$ (cf.\ \cite{RS,Bf}).
Suppose $\kappa\in(0,8)$. Then the probability that a chordal SLE$_\kappa$ curve $\eta$ exactly passes through any fixed interior point $z_0$ in the domain is zero. People were interested in the decay rate of the probability $\PP[\dist(z_0,\eta)<r]$ as $r\downarrow 0$. Specifically, it is about the limit $\lim_{r\downarrow 0} r^{-\alpha} \PP[\dist(z_0,\eta)<r]$, where $\alpha=1-\frac \kappa 8>0$ equals $2$ minus the Hausdorff dimension. The limit is called the Green's function for SLE.

A Green's function problem usually includes two parts: proving the existence of the limit  and finding the exact formula of the Green's function. The exact formula of the Green's function of chordal SLE was predicted in \cite{RS}, and the convergence and the formula were proved later in \cite{LR}. The paper \cite{LR} also proves the convergence of two-point Green's function, i.e., the limit of the renormalized probability that an SLE$_\kappa$ curve gets close to two distinct points, and uses the one- and two- point Green's functions to prove the existence of Minkowski content of SLE.

Miller and Wu studied (cf.\ \cite{MW}) the Hausdorff dimensions of the sets of double points and cut points of SLE. Among other results, they proved that the Hausdorff dimension of the set of cut points of SLE$_\kappa$ is $3-\frac{3}8\kappa$ for $\kappa\in(4,8)$. They derived a Green's function type estimate: for a chordal SLE$_\kappa$ curve $\eta$ and a point $z_0$ in the same domain,
\BGE \PP[\eta \mbox{ has a cut point contained in }\{|z-z_0|<r\}]=r^{\alpha_0+o(1)},\quad \mbox{as }r\downarrow 0,\label{MW-est}\EDE
where   the exponent $\alpha_0$ given by
\BGE \alpha_0=\frac 38\kappa-1\label{alpha0}\EDE
equals $2$ minus the Hausdorff dimension of the cut-set, and $o(1)\to 0$ as $r\downarrow 0$.

 The intention of the current paper is to improve the estimate (\ref{MW-est}). Let $\eta$ be a chordal SLE$_\kappa$ curve in a simply connected domain from $a_1$ to $a_2$. We expect  that the limit
 \BGE G_{D;a_1,a_2}(z_0):= \lim_{r\downarrow 0} r^{-\alpha_0} \PP[\eta \mbox{ has a cut point contained in }\{|z-z_0|<r\}]\label{Green-original}\EDE should converge to a finite positive number.
This is inspired by \cite{Brownian-cut}, which studies the one-point and two-point Green's functions for cut points of Brownian motion, and uses them to prove the existence of Minkowski content for the set of cut points.
 

A new technique was introduced in \cite{Two-Green-interior}, which studies the Green's function of $2$-SLE$_\kappa$ for $\kappa\in(0,8)$. A $2$-SLE$_\kappa$ configuration is a pair of random curves $(\eta_1,\eta_2)$ in a simply connected domain $D$ connecting four marked boundary points such that when any curve is given, the other curve is a chordal SLE$_\kappa$ in a complement domain of the given curve. The Green's function for such $2$-SLE$_\kappa$ at $z_0\in D$ is the limit $\lim_{r\downarrow 0} r^{-\alpha}\PP[\eta_j\cap \{|z-z_0|<r\}\ne \emptyset,j=1,2]$, where $\alpha>0$ is given by $\frac{(12-\kappa)(\kappa+4)}{8\kappa}$. Such a limit turns out to converge to a positive number, whose value can be described using a hypergeometric function.

The main idea in \cite{Two-Green-interior} is the following. Assume that $D=\D:=\{|z|<1\}$ and $z_0=0$ by conformal invariance. Suppose $\eta_j$ connects $a_j$ with $b_j$, $j=1,2$, and $a_1,b_1,a_2,b_2$ are oriented clockwise. After a reduction we may assume that $b_1=1$ and $b_2=-1$. Then we simultaneously grow $\eta_1$ and $\eta_2$ respectively from $a_1$ and $a_2$ with some random speeds such that at any time $t$ before the process stops,
\begin{enumerate}
  \item [(C1)] the conformal radius of the remaining domain viewed from $0$ is $e^{-t}$, and
  \item [(C2)] $1$ and $-1$ equally split the harmonic measure of the remaining domain  viewed from $0$.
\end{enumerate}
 The process stops at some time $T$ when the two curves together disconnect $0$ from any of $b_1,b_2$. By Koebe's $1/4$ theorem and Beurling's estimate, we then know that for any $t\in[0,T)$, $\max\{\dist(0,\eta_j[0,t]),j=1,2\}\asymp e^{-t}$. After the time $T$, the two curves will not both get closer to $0$, and so $\max\{\dist(0,\eta_j),j=1,2\}\asymp e^{-T}$.

Thus, the original limit problem is transformed into finding the decay rate of $\PP[T>s]$ as $s=-\log(r)\to \infty$. For any $t\in[0,T)$, one may find a conformal map, which maps the remaining domain at $t$ conformally onto $\D$, and fixes $0,1,-1$. Then $\eta_1(t)$ and $\eta_2(t)$ are mapped to $e^{iZ_1(t)}$ and $-e^{iZ_2(t)}$ for some $Z_1(t),Z_2(t)\in(0,\pi)$. The $(Z_1,Z_2)$ is a two-dimensional diffusion process with lifetime $T$. The explicit transition density of this process was derived in \cite{Two-Green-interior} using It\^o's calculus and orthogonal polynomials of two variables, which implies the key estimate that there is a function $G$ on $(0,\pi)^2$ such that if $(Z_1,Z_2)$ starts from $(z_1,z_2)$, then $P[T>s]=G(z_1,z_2) e^{-\alpha_1 s} (1+O(e^{-\beta s}))$ for some constant $\beta>0$. Such an estimate was then used to prove the convergence of the Green's function.

The two-curve technique in \cite{Two-Green-interior} was later used in \cite{Two-Green-boundary} to study the existence of Green's function for $2$-SLE$_\kappa$ at a boundary point. The results of \cite{Two-Green-interior,Two-Green-boundary} may serve as the first few steps in proving the existence of Minkowski content of double points of SLE$_\kappa$.

In the current paper, we are going to apply the two-curve technique to study the cut point Green's function. Let $D$ be a Jordan domain, and $a_1\ne a_2\in\pa D$. Let $\kappa\in(4,8)$, and $\gamma_1$ be a chordal SLE$_\kappa$ curve  in $D$ from $a_1$ to $a_2$. In order to apply the technique, we introduce two other points $b_1,b_2\in\pa D$ such that $b_j$ lies on the open boundary arc of $D$ from $a_j$ to $a_{3-j}$ in the clockwise direction. Let $A_j$ be the open boundary arc of $D$ with end points $b_1,b_2$ that contains $a_j$, $j=1,2$. Instead of considering the Green's function for cut points of $\gamma_1$, we study the Green's function for the cut points of the union $\gamma_1\cup A_1\cup A_2$. This is a compromise we have to make now in order to apply the two-curve technique in \cite{Two-Green-interior}. Our main theorem is the following.


\begin{Theorem}
Let $z_0\in D$ and $R=\dist(z_0,D^c)$. Let $f$ be a conformal map from $D$ onto $\D$ such that $f(z_0)=0$. Let $w_j,v_j\in\R$ be such that $e^{iw_j}=f(a_j)$ and $e^{iv_j}=f(b_j)$, $j=1,2$. Let $\alpha_0$ be as in (\ref{alpha0}). Define  $ {G_{D;a_1,b_1,a_2,b_2}(z_0)}=|f'(z_0)|^{\alpha_0} \til G(w_1,v_1,w_2,v_2)$,
where
\BGE \til G(w_1,v_1,w_2,v_2):=\Big|\sin\Big(\frac{w_1-w_2}2\Big)\Big|^{\frac 8\kappa -1}\Big|\sin\Big(\frac{v_1-v_2}2\Big)\Big|^{\frac{(\kappa-4)^2}{2\kappa}} \prod_{j=1}^2\prod_{s=1}^2 \Big|\sin\Big(\frac{w_j-v_s}2\Big)\Big|^{1-\frac 4\kappa}.\label{G}\EDE
 Let $\beta_0=1-\frac 8{5\kappa}$.
Then there is a positive constant $C_0$ depending only on $\kappa$ such that
$$\PP[\gamma_1\cup A_1\cup A_2\mbox{ has a cut point contained in }\{|z-z_0|<r\}]$$
$$=C_0 {G_{D;a_1,b_1,a_2,b_2}(z_0)} r^{\alpha_0}(1+O(r/R)^{\beta_0}),\quad \mbox{as } r/R\downarrow 0.$$
Here the implicit constants in the $O(\cdot)$ symbol depend only on $\kappa$.
\label{Main-thm}
\end{Theorem}

We will basically follow  the approach in \cite{Two-Green-interior}.
First suppose $D=\D$, $z_0=0$, $b_1=1$, and $b_2=-1$. Let $\gamma_2$ be a time-reversal of $\gamma_1$, which by reversibility of SLE$_\kappa$ (cf.\ \cite{MS3}) is a chordal SLE$_\kappa$ in $\D$ from $a_2$ to $a_1$. We then grow $\gamma_1$ and $\gamma_2$ simultaneously respectively from $a_1$ and $a_2$ with random speeds such that at any time $t$ before the first time $T$ that $A_1\cup \eta_1[0,T]$ intersects $A_2\cup \eta_2[0,T]$, Conditions (C1) and (C2) are both satisfied. By mapping the remaining domain conformally onto $\D$ with $0,1,-1$ fixed, we obtain a two-dimensional diffusion process $(Z_1(t),Z_2(t))$ taking values in $(0,\pi)^2$. We then use orthogonal polynomials to derive the transition density of $(Z_1,Z_2)$, and use it to prove Theorem \ref{Main-thm}.

This paper can be viewed as the first step of proving the existence of Minkowski content of cut points of SLE. For that purpose, following the approach of \cite{LR}, we will also need   {the ordered two-point Green's function, i.e., for $z_1\ne z_2\in D$,  the limit
$$\lim_{r_1,r_2\downarrow 0} r_1^{-\alpha_0}r_2^{-\alpha_0} \PP[\gamma_1\mbox{ first visits a cut point of } \gamma_1\cup A_1\cup A_2\mbox{ in }\{|z-z_1|<r_1\}$$
$$\mbox{ and then visits a cut point of }\gamma_1\cup A_1\cup A_2\mbox{ in }\{|z-z_2|<r_2\}].$$  }

{The first result on two-point Green's functions was derived in \cite{LW}, where the ordered two-point Green's function for a chordal SLE$_\kappa$ curve $\gamma_1$ in $D$ from $a_1$ to $a_2$ is expressed as $$G_{D;a_1,a_2}(z_1)\cdot \EE[G_{D_{\gamma_*};z_1,a_2}(z_2)],$$ where $G_{D;a_1,a_2}$ is the one-point Green's function for $\gamma_1$, $D_{\gamma_*}$ is the complement domain of a random curve $\gamma_*$ in $D$ whose boundary contains $a_2$, where $\gamma_*$ is the first arm of a two-sided radial SLE$_\kappa$ curve in $D$ from $a_1$ to $a_2$ via $z_1$, which is also a radial SLE$_\kappa(2)$ curve in $D$ from $a_1$ to $z_1$ with the force point at $a_2$, and the expectation is with respect to the randomness of $\gamma_*$. Here the SLE$_\kappa(2)$ curve  $\gamma_*$ can be intuitively viewed as the chordal SLE$_\kappa$ curve $\gamma_1$ conditioned on the singular event that it passes through $z_1$, and $z_1$ is a boundary point of $D_{\gamma_*}$.}

{Thus, we expect that the ordered two-point Green's function for cut points of $\gamma_1\cup A_1\cup A_2$ can be expressed as the product
$$G_{D;a_1,b_1,a_2,b_2}(z_1)\cdot \EE[G_{D_{\gamma_*};z_1,b_1,a_2,b_2}(z_2)],$$
where  $G_{D;a_1,b_1,a_2,b_2}$ is the (one-point) Green's function as in Theorem \ref{Main-thm}, $D_{\gamma_*}$ is the complement domain of a random curve $\gamma_*$ in $D$ whose boundary contains $a_2$, where $\gamma_*$ is an SLE$_\kappa(2,\kappa-4,\kappa-4)$ curve in $D$ from $a_1$ to $z_1$ with force points at $a_2,b_1,b_2$. We will see later in Section \ref{subsubsection-construction} that such an SLE$_\kappa(2,\kappa-4,\kappa-4)$ curve $\gamma_*$ can be intuitively viewed as the chordal SLE$_\kappa$ curve $\gamma_1$  conditioned on the singular event that $\gamma_1\cup A_1\cup A_2$ has a cut point at $z_1$.
}

 {One may also consider the boundary Green's function for cut points of SLE. This is the limit of (\ref{Green-original}) in the case that $z_0$ is a boundary point. Here $\alpha_0$ should be replaced by another exponent $\alpha_1$, and $\pa D$ is assumed to be analytic near $z_0$. Similarly as in \cite{LW}, the boundary Green's function for cut points may be used to study the two-point cut point Green's function. While the existence of the boundary Green's function for cut points is beyond reach now, the technique of this paper may be used to prove the convergence of the limit:
$$\lim_{r\downarrow 0} r^{-1} \PP[\gamma_1\cup A_1\cup A_2\mbox{ has a cut point contained in }\{|z-b_1|<r\}],$$
where $b_1,b_2,A_1,A_2$ are as before. This means that we set $b_1=z_0$, and take $\alpha_1=1$.
}

Theorem \ref{Main-thm} does not immediately imply the convergence of the original limit (\ref{Green-original}), which motivated this paper. We made the problem easier by adding two marked boundary points. One approach to attack the original problem is to add two more  marked boundary points, which means that the attached boundary arcs do not share end points. In that case the formula of the Green's function is expected to be {a multivariable special function}.

We believe that the limit (\ref{Green-original}) does converge.  {We may consider the simple case that} $D=\HH=\{\Imm z>0\}$, $a_1=0$ and $a_2=\infty$. By the scaling property and reflection symmetry of SLE, there is a positive function $S(\theta)$ on $(0,\pi)$, which is symmetric about $\pi/2$, such that  $G_{\HH;0,\infty}(z)= {|z|^{-\alpha_0} S(\arg z)}$. {To get the explicit formula of $G$, we need to find out the formula of $S$.} The boundary exponent $\alpha_1=1$ suggests that $\lim_{\theta\downarrow 0}S(\theta)/\theta$ is finite and positive. 

Below is a sketch of the rest of the paper. In Section \ref{preliminary}, we review Loewner equations and two-time-parameter stochastic processes. In Section \ref{section-Ensemble-radial}, we recall the setup and results from \cite[Sections 3,4,5]{Two-Green-interior} with suitable modifications to fit the  {setting} here. We then prove Theorem \ref{Main-thm}  in Section \ref{section-cutpoint}. In the appendix, we give a detailed proof of a two directional  domain Markov property for chordal SLE$_\kappa$ which satisfies reversibility.

\section{Preliminary} \label{preliminary}
Let $\HH=\{z\in\C :\Imm z>0\}$ and $\D=\{z\in\C:|z|<1\}$. Let $J(z)=-1/z$.
By $f:D\conf E$ we mean that $f$ maps $D$ conformally onto $E$. Let $\cot_2,\sin_2,\cos_2,\tan_2$ denote the functions $\cot(\cdot/2),\sin(\cdot/2),\cos(\cdot/2),\tan(\cdot/2)$, respectively.

\subsection{Prime ends}
For the sake of rigor, we will use the notion of prime ends (cf.\ \cite{Ahl}). We prefer the description used in \cite{LERW}. Let $D$ be a simply connected domain. Consider the set of pairs $(f,\zeta)$, where $f:D\conf \D$, and $\zeta\in \lin \D$. Such a set is nonempty by Riemann mapping theorem. Define an equivalence relation ``$\sim$'' on this set such that $(f_1,\zeta_1)\sim(f_2,\zeta_2)$  if (the continuation of) $f_2\circ f_1^{-1}$ maps $\zeta_1$ to $\zeta_2$. The set of $(f,\zeta)$ modulo the equivalence relation is called the prime end closure of $D$, which is equipped with a unique metrizable topology such that for any $f:D\conf \D$, the map $[(f,\zeta)]\mapsto \zeta$ is a homeomorphism from $\ha D$ onto $\lin\D$. Note that $z\mapsto [(f,f(z))]$ is an embedding from $D$ into $\ha D$. We identify $D$ as an open subset of $\ha D$ by identifying $z\in D$ with $[(f,f(z))]$, which does not depend on the choice of $f$.  We call $\ha\pa D:=\ha D\sem D$ the prime end boundary of $D$, and every  $ p\in \ha\pa D$  a prime end of $D$. If $g:D_1\conf D_2$, then $g$ induces a homeomorphism $g:[(f,\zeta)]\mapsto [(f\circ g^{-1},\zeta)]$ from $\ha D_1$ onto $\ha D_2$, which maps $\ha\pa D_1$ onto $\ha\pa D_2$.

Let $\lin D^\#$ be the closure of $D$ in the Riemann sphere $\ha\C:=\C\cup\{\infty\}$, and $\pa^{\#} D=\lin D^\#\sem D$. Let $\zeta\in \pa^{\#} D$ (a point in $\ha\C$) and $p\in \ha\pa D$ (a prime end of $D$). If $z\in D$ and $z\to p$ (w.r.t.\ the topology of $\ha D$) implies that $z\to \zeta$ (w.r.t.\ the metric of $\ha\C$), then we say that  $\zeta$ is determined by $p$. On the other hand, if $z\in D$ and $z\to\zeta$ implies that $z\to p$, then we say that $\zeta$ determines $p$. If $p$ and $\zeta$ determine each other, then we do not distinguish them. For example, if $D$ is a Jordan domain in $\ha\D$ such as $\D$ or $\HH$, then every $p\in\ha D$ is identified with a point $\zeta\in \pa^\# D$. 

Suppose $\til D\subset D$ are two simply connected domains. Let $p $ be a prime end of $D$. If $\til D$ is a neighborhood of $p$ in $D$, i.e., the intersection of $D$ and a neighborhood of $p$ in $\ha D$, then by Schewarz reflection principle, there is a unique prime end $\til p$ of $\til D$ such that for $z\in \til D$, $z\to p$ in $D$ iff $z\to \til p$ in $\til D$. When this happens, we do not distinguish $\til p$ from $p$, and say that $D$ and $\til D$ share the prime end $p$.

\subsection{Chordal  Loewner equations}\label{chordal}
A relatively closed subset $K$ of $\HH$ is called an $\HH$-hull if $K$ is bounded and $\HH\sem K$ is a simply connected domain, and $\HH\sem K$ is called an $\HH$-domain.  For an $\HH$-hull $K$, there is a unique  $g_K:\HH\sem K\conf \HH$ such that $g_K(z)=z+\frac cz+O(1/z^2)$ as $z\to \infty$ for some $c\ge 0$. The constant $c$, denoted by $\hcap(K)$, is called the $\HH$-capacity of $K$, which is zero iff $K=\emptyset$.  We write $\hcap_2$ for $\hcap(\cdot/2)$.
For a set $S\subset\C$, if there is an $\HH$-hull $K$ such that $\HH\sem K$ is
the unbounded connected component of $\HH\sem \lin S$, then we say that $K$ is  the $\HH$-hull generated by $S$, and write $K=\Hull(S)$. In this case we write $\hcap(S)$ and $\hcap_2(S)$ respectively for $\hcap(K)$ and $\hcap_2(K)$.


Let $\ha w \in C([0,T),\R)$ for some $T\in(0,\infty]$. The chordal Loewner equation driven by $\ha w$  is
$$\pa_t g_t(z)=  \frac{2}{g_t(z)-\ha w(t)},\quad 0\le t<T;\quad g_0(z)=z.$$
 { For a fixed $z\in\C$, this is an ODE in $t$ with initial value at $0$. Let $\tau_z\in(0,\infty]$ be such that $[0,\tau_z)$ is the maximal  interval of the solution.}
For  $t\in[0,T)$, let $K_t=\{z\in\HH:\tau_z\le t\}$. It turns out that each $K_t$ is an $\HH$-hull with $\hcap_2(K_t)=t$, 
and each $g_t$ agrees with the $g_{K_t}$ associated with $K_t$. We call $g_t$ and $K_t$ the chordal Loewner maps and hulls, respectively, driven by $\ha w$.






If for every $t\in[0,T)$,  $g_{t}^{-1}$ extends continuously to $\lin\HH$, and $\eta(t):=g_{t}^{-1}({\ha w(t)})$, $0\le t<T$, is a continuous curve, then we say that $\eta$ is the chordal Loewner curve driven by $\ha w$. Such $\eta$ may not exist in general. When it exists, we have $\eta(0)=\ha w(0)\in\R$, $\eta(t)\in\lin{K_t}\subset \lin\HH$  and $K_t=\Hull(\eta[0,t])$ for all $0\le t<T$. On the other hand, if there is a continuous curve $\eta$ such that  $K_t=\Hull(\eta[0,t])$ for each $0\le t<T$, then $\eta$ is the chordal Loewner curve driven by $\ha w$. For every $t_0\in[0,T)$, the set $S_{t_0}:=(t_0,T)\cap \eta^{-1}(\HH\sem K_{t_0})$ is dense in $[t_0,T)$, and as $t\to t_0$ along $S_{t_0}$, $g_{t_0}(\eta(t))\to \ha w(t_0)$, and so $\eta(t)$ tends to the prime end  $g_{t_0}^{-1}(\ha w(t_0))$ of $\HH\sem K_{t_0}$. When there is no ambiguity, we also use $\eta(t)$ to denote the prime end $g_t^{-1}(\ha w(t))$ of $\HH\sem K_t$. 

For $\kappa\in(0,\infty)$, chordal SLE$_\kappa$ is defined by solving the chordal Loewner equation with the driving function being $\ha w(t)=\sqrt\kappa B(t)$, $0\le t<\infty$, where $B$ is a standard linear Brownian motion. It is known (cf.\ \cite{RS,LSW}) that a.s.\ $\ha w$ generates a chordal Loewner curve $\eta$, which satisfies $\eta(0)=\ha w(0)=0$ and $\lim_{t\to\infty}\eta(t)=\infty$. Such $\eta$ is called a chordal SLE$_\kappa$ curve in $\HH$ from $0$ to $\infty$, or a standard chordal SLE$_\kappa$. If $D$ is a simply connected domain with two distinct prime ends $a,b$, then there is a conformal map $f$ from $\HH$ onto $D$ such that $f(0)=a$ and $f(\infty)=b$. Then the $f$-image of a standard chordal SLE$_\kappa$, which is a continuous curve in the prime end closure of $D$, is called an SLE$_\kappa$ curve in $D$ from $a$ to $b$. Whenever $\pa^\# D$ is locally connected,   the chordal SLE$_\kappa$ is also a curve in the spherical closure of $ D$.

Chordal SLE$_\kappa$ satisfies the following DMP (domain Markov property, cf.\ \cite{Law-SLE}). Let $\eta$ be a  chordal SLE$_\kappa$ curve in $D$ from $a$ to $b$. Let $\cal A$ be a $\sigma$-algebra independent of $\eta$, and $\F^{\cal A}$ be the usual (complete and right continuous) augmentation of the filtration $(\sigma(\eta_1(s):s\le t)\vee {\cal A})_{t\ge 0}$. Then for any $\F^{\cal A}$-stopping time $T$, conditionally on $\F^{\cal A}_T$ and the event $\{T<\infty\}$, $\eta(T+\cdot)$ is a chordal SLE$_\kappa$ curve from (the prime end) $\eta(T)$ to $b$ in the connected component of $D\sem \eta[0,T]$ which shares the prime end $b$ with $D$. This follows easily from the Markov property of the (rescaled) Brownian motion $\sqrt\kappa B_t$.

For $\kappa\in(0,8]$, chordal SLE$_\kappa$ satisfies reversibility (cf.\ \cite{LSW} for $\kappa=8$, \cite{reversibility} for $\kappa\in(0,4]$, and \cite{MS3} for $\kappa\in(4,8)$). This means that, if $\eta$ is a chordal SLE$_\kappa$ curve in $D$ from $a$ to $b$, then any time-reversal of $\eta$ is a time-change of a chordal SLE$_\kappa$ curve in $D$ from $b$ to $a$.



\subsection{Radial Loewner equations}
A relatively closed subset $K$ of $\D$ is called a $\D$-hull if  $\D\sem K$ is a simply connected domain that contains $0$. For an $\D$-hull $K$, there is a unique conformal map $g_K$ from $\D\sem K$ onto $\D$ such that $g_K(0)=0$ and $g_K'(0)\ge 1$. The constant $\log |g_K'(0)|\ge 0$ is called the $\D$-capacity of $K$ and is denoted by $\dcap(K)$. 
For a set $S\subset\C$, if there is a $\D$-hull $K$ such that $D\sem K$ is the connected component of $\D\sem \lin S$ that contains $0$, then we say that $K$ is  the $\D$-hull generated by $S$, and write $K=\Hull(S)$, $\dcap(S)=\dcap(K)$, and $g_S=g_K$.

Let $\ha w \in C([0,T),\R)$ for some $T\in(0,\infty]$. The radial Loewner equation driven by $\ha w$  is
$$\pa_t g_t(z)= g_t(z) \frac{e^{i\ha w(t)}+g_t(z)}{e^{i\ha w(t)}-g_t(z)},\quad 0\le t<T;\quad g_0(z)=z.$$
 { For $z\in\C$, let $\tau_z\in (0,\infty]$ be such that $[0,\tau_z)$ is the maximal interval of the solution.}
For  $t\in[0,T)$, let $K_t=\{z\in\D:\tau_z\le t\}$. It turns out that each $K_t$ is an $\D$-hull with $\dcap(K_t)=t$, 
and each $g_t$ agrees with the $g_{K_t}$ associated with $K_t$. We call $g_t$ and $K_t$ the radial Loewner maps and hulls, respectively, driven by $\ha w$. By the Schwarz lemma and Koebe's $1/4$ theorem, we have $e^{-t}/4\le \dist(0,K_t)\le e^{-t}$ for $0<t<T$.

If for every $t\in[0,T)$,  $g_{t}^{-1}$ extends continuously to $\lin\D$, and $\eta(t):=g_{t}^{-1}(e^{i\ha w(t)})$, $0\le t<T$, is a continuous curve, then we say that $\eta$ is the radial Loewner curve driven by $\ha w$. Such $\eta$ may not exist in general. When it exists, we have $\eta(0)=e^{i\ha w(0)}\in\pa \D$, $\eta(t)\in\lin{K_t}\subset \lin\D$, and $K_t=\Hull(\eta[0,t])$ for all $0\le t<T$. On the other hand, if there is a continuous curve $\eta$ such that  $K_t=\Hull(\eta[0,t])$ for each $t\in[0,T)$, then $\eta$ is the radial Loewner curve driven by $\ha w$. When there is no ambiguity, we also use $\eta(t)$ to denote the prime end $g_t^{-1}(e^{i\ha w(t)})$ of $\D\sem K_t$.

We will use covering radial Loewner equations. Let $e^i$ denote the map $z\mapsto e^{iz}$, which maps $\HH$ onto $\D\sem \{0\}$. Let $\ha w\in C([0,T),\R)$. The covering radial Loewner equation driven by $\ha w$ is
$$\pa_t \til g_t(z)=\cot_2(\til g_t(z)-\ha w(t)),\quad \til g_0(z)=z.$$
For each $0\le t<T$, let $\til K_t=\{z\in\HH:\til \tau_z\le t\}$, where $\til \tau_z$
 {is such that $[0,\til\tau_z)$ is the maximal interval of the solution.} We call $\til g_t$ and $\til K_t$ respectively the covering radial Loewner maps and hulls driven by $\ha w$. If $g_t$ and $K_t$ are respectively the radial Loewner  maps and hulls driven by the same $\ha w$, then $\til K_t=(e^i)^{-1}(K_t)$ and $e^i\circ \til g_t=g_t\circ e^i$ for all $0\le t<T$.

If there is a continuous and strictly increasing function $u$ on $[0,T)$ with $u(0)=0$, such that $K_{u^{-1}(t)}$ and $g_{u^{-1}(t)}$, $0\le t<u(T)$, are respectively the radial Loewner hulls and maps driven by $\ha w\circ u^{-1}$, then we say that $K_t$ and $g_t$ are respectively radial Loewner hulls and maps with speed $du$ (understood as a measure) driven by $\ha w$. If $u$ is absolutely continuous and $u'=q$, we also say that the speed is $q$. We similarly define radial Loewner curve and covering radial Loewner hulls and maps with some speed. If $\eta$ is a radial Loewner curve with some speed, then we also understand  $\eta(t)$ as the prime end $g_{\eta[0,t]}^{-1}(\ha w(t))$ of $\HH\sem \Hull(\eta[0,t])$. 

We now recall radial SLE$_\kappa(\ulin\rho)$ processes, where $\kappa>0$ and $\ulin\rho=(\rho_1,\dots,\rho_m)\in\R^m$ for some $m\in\N\cup\{0\}$. Let $e^{iw},e^{iv_1},\dots,e^{iv_m}$ be distinct points on $\pa\D$. A radial SLE$_\kappa(\ulin\rho)$ curve in $\D$ started from $e^{iw}$ aimed at $0$ with force points $e^{iv_1},\dots,e^{iv_m}$ is the radial Loewner curve driven by $\ha w(t)$, $0\le t<\ha T$, which solves the SDE:
$$d\ha w(t)=\sqrt\kappa dB_t+\sum_{j=1}^m \frac{\rho_j}2 \cot_2(\ha w(t)-\til g^{\ha w}_t(v_j))\,dt,\quad \ha w(0)=w,$$
where $B$ is a standard Brownian motion, and $\til g^{\ha w}_t$ are covering radial Loewner maps driven by $\ha w$. The process stops whenever the solution $t\mapsto \til g^{\ha w}_t(v_j)$ blows up for any $j$. The  radial SLE$_\kappa(\ulin\rho)$ curve, whose existence follows from the Girsanov Theorem (cf.\ \cite{RY}), starts from $e^{iw}$, and may or may not end at $0$ (depending on $\kappa$ and $\ulin\rho$). The following proposition is \cite[Lemma 3.4]{Two-Green-interior}.

\begin{Proposition}
Let $\kappa>0$, $n\in\N$. Suppose $\ulin\rho=(\rho_1,\dots,\rho_n)\in\R^n$ satisfies $\rho_1,\rho_n\ge \frac\kappa 2-2$ and $\rho_k\ge 0$, $1\le k\le n$.  Let $e^{iw},e^{iv_1},\dots,e^{iv_n}$ be distinct points on $\TT$ such that $w>v_1>\cdots>v_n>w-2\pi$.  Let $\eta(t)$, $0\le t<T$, be a radial SLE$_\kappa(\ulin\rho)$ curve in $\D$ started from $e^{iw}$ aimed at $0$ with force points $e^{iv_1},\dots,e^{iv_n}$. Then a.s.\ $T=\infty$, $0$ is a subsequential limit of $\eta(t)$ as $t\to \infty$, and $\eta$ does not hit the arc $J:=\{e^{i\theta}:v_1\ge \theta\ge v_n\}$. \label{transience}
\end{Proposition}

\subsection{Two-parameter stochastic processes}\label{section-two-parameter}
Now we recall the results in  \cite[Section 2.3]{Two-Green-interior}. We assign a partial order $\le$ to $\R_+^2=[0,\infty)^2$  such that $\ulin t=(t_1,t_2)\le(s_1,s_2)= \ulin s$ iff $t_j\le s_j$, $j=1,2$. The minimal element of $\R_+^2$ is $\ulin 0=(0,0)$. We write $\ulin t<\ulin s$ if $t_j<s_j$, $j=1,2$. We define $\ulin t\wedge \ulin s=(t_1\wedge s_1,t_2\wedge s_2)$. Given $\ulin t,\ulin s\in\R_+^2$, we define $[\ulin t,\ulin s]=\{\ulin r\in{\R_+^2}:\ulin t\le \ulin r\le \ulin s\}$. Let $\ulin e_1=(1,0)$ and $\ulin e_2=(0,1)$. So $(t_1,t_2)=t_1\ulin e_1+ t_2\ulin e_2$.
For a function $X$ defined on a subset $U$ of $\R_+^2$, $j\in\{1,2\}$, and $t_j\in \R_+$, we use $X|^j_{t_j}(t)$ to denote the function $t\mapsto X(t_j\ulin e_j+t \ulin e_{3-j})$, whose domain is those $t\in\R_+$ such that $t_j \ulin e_j+t \ulin e_{3-j}\in U$.

\begin{Definition}
Let $\F_{\ulin t}$, $\ulin t\in\R_+^2$, be a family of  $\sigma$-algebras on a measurable space $\Omega$ such that $\F_{\ulin t}\subset \F_{\ulin s}$ whenever $\ulin t\le \ulin s$. Then we call  $(\F_{\ulin t})_{\ulin t\in\R_+^2}$ an ${\R_+^2}$-indexed filtration on $\Omega$, and simply denote it by $\F$. 
\end{Definition}

From now on, we let $\F$ be an $\R_+^2$-indexed filtration on $\Omega$, and let $\F_{\ulin\infty}=\vee_{\ulin t\in\R_+^2}\F_{\ulin t}$.

\begin{Definition}
 A random map $\ulin T:\Omega\to[0,\infty]^2$ is called an $\F$-stopping time if for any deterministic $\ulin t\in{\R_+^2}$, $\{\ulin T\le \ulin t\}\in \F_{\ulin t}$.
We call $\ulin T$ finite if it takes values in $\R_+^2$, and bounded if there  is a deterministic $\ulin t\in \R_+^2$ such that $\ulin T\le \ulin t$. For an   $\F$-stopping time $\ulin T$, we define $\F_{\ulin T}=\{A\in \F_{\ulin \infty}:A\cap \{\ulin T\leq \ulin t\}\in \F_{\ulin t}, \forall \ulin t\in{\R_+^2}\}$. \label{stoppingtime}
\end{Definition}


\begin{Definition}
		A relatively open subset $\cal R$ of $\R_+^2$ is called a history complete region, or simply an HC region, if for any $\ulin t\in \cal R$, we have $[\ulin 0, \ulin t]\subset\cal R$. Given an HC region $\cal R$, for $j\in\{1,2\}$, define   $T^{\cal R}_j:\R_+\to\R_+\cup\{\infty\}$ by $T^{\cal R}_j(t)=\sup\{s\ge 0:s \ulin e_j+t\ulin e_{3-j}\in{\cal R}\}$, where we set $\sup\emptyset =0$.

 {The space of HC regions is naturally equipped with a $\sigma$-algebra, which is generated by the family  $\{{\cal R}:\ulin t\in {\cal R}\}$, $\ulin t\in\R_+^2$.
A map $\cal D$ from $\Omega$ into the space of HC regions} is called an $\F$-stopping region if for any $\ulin t\in\R_+^2$, $\{\omega\in\Omega:\ulin t\in {\cal D}(\omega)\}\in\F_{ \ulin t}$. {Such a map is clearly measurable.} A random function $X({\ulin t})$  {defined on $\Omega$} with a random domain $\cal D$ is called an $\F$-adapted HC process if $\cal D$ is an $\F$-stopping region, and for every  $\ulin t\in\R_+^2$, $X_{\ulin t}$ restricted to $\{\ulin t\in\cal D\}$ is $\F_{ \ulin t}$-measurable. \label{Def-HC}
	\end{Definition}

The following proposition is 
\cite[Lemma 2.7]{Two-Green-interior}.
	
	\begin{Proposition}
		Let $\ulin T$ and $\ulin S$ be two $\F$-stopping times. Then (i)  $\{\ulin T\le \ulin S\} \in\F_{\ulin S}$; (ii) if $\ulin S$ is a constant $\ulin s\in\R_+^2$, then $\{\ulin T\le \ulin S\} \in\F_{ \ulin T}$; and (iii) if $f$ is an $\F_{ \ulin T}$-measurable function, then ${\bf 1}_{\{\ulin T\le \ulin S\}}f$ is $\F_{ \ulin S}$-measurable. In particular, if $\ulin T\le \ulin S$, then $\F_{ \ulin T}\subset \F_{ \ulin S}$. \label{T<S}
	\end{Proposition}

	\begin{Definition}
Suppose that there are two $\R_+$-indexed  filtrations $\F^1$ and $\F^2$ such that $\F_{(t_1,t_2)}=\F^1_{t_1}\vee \F^2_{t_2}$, $(t_1,t_2)\in\R_+^2$. Then we say that $\F$ is a separable filtration generated by $\F^1$ and $\F^2$. For such a separable filtration, if $T_j$ is an $\F^j$-stopping time, $j=1,2$, then $\ulin T:=(T_1,T_2)$ is   called a separable  stopping time  (w.r.t.\ $\F^1$ and $\F^2$).  \label{separable}
\end{Definition}

The following proposition is  \cite[Lemma 2.13]{Two-Green-interior}.

\begin{Proposition}
Suppose $\F$ is separable. Let $\ulin T$ and $\ulin S$ be two $\F$-stopping times, where $\ulin S$ is separable. Then $\{\ulin T\le \ulin S\}\in\F_{ \ulin T}$. \label{separable-lemma}
\end{Proposition}

\begin{Definition}
The right-continuous augmentation of $\F$ is another $\R_+^2$-indexed filtration $\lin \F$ defined by $\lin\F_{\ulin t}=\bigcap_{\ulin s>\ulin t}\F_{\ulin s}$, $\ulin t\in\R_+^2$.  We say that  $\F$ is right-continuous if $\lin\F =\F $. 
\end{Definition}

The following proposition follows from a standard argument.

\begin{Proposition}
The right-continuous augmentation $\lin\F$ of $\F$ is right-continuous.
An $[0,\infty]^2$-valued map $\ulin T$ on $\Omega$ is an   $\lin\F$-stopping time if and only if $\{\ulin T<\ulin t\}\in\F_{\ulin t}$ for any $\ulin t\in\R_+^2$; and for such $\ulin T$, $A\in \lin\F_{\ulin T}$ if and only if $A\cap\{\ulin T<\ulin t\}\in\F_{\ulin t}$ for any $\ulin t\in\R_+^2$. If $\F$ is right-continuous, and if $(\ulin T^n)_{n\in\N}$ is a decreasing sequence of $\F$-stopping times, then $\ulin T:=\inf_n \ulin T^n$ is also an $\F$-stopping time, and $\F_{\ulin T}=\bigcap_n \F_{\ulin T^n}$.
  \label{right-continuous-0}
\end{Proposition}

Now we fix a probability measure $\PP$ on $(\Omega,\F_{\ulin\infty})$,  and let $\EE[\cdot|\F_t]$ denote the corresponding conditional expectations.
	
\begin{Definition}
	An  $\F$-adapted process $(X_{\ulin t})_{\ulin t\in\R_+^2}$ is called an  $\F$-martingale (w.r.t.\ $ \PP$) if  for any $\ulin s\le \ulin t\in\R_+^2$, a.s.\ $\EE[X_{\ulin t}|\F_{\ulin s}]=X_{\ulin s}$. If there is $\zeta\in L^1(\Omega,\F_{\ulin\infty},\PP)$ such that a.s.\ $X_{\ulin t}=\EE[\zeta|\F_{ \ulin t}]$ for all $\ulin t\in\R_+^2$, then we say that $X$ is closed by $\zeta$.
\end{Definition}

The following is \cite[Lemma 2.11]{Two-Green-interior}.



		
	\begin{Proposition} [Optional Stopping Theorem]
		Suppose $X$ is a continuous $\F$-martingale.
Then the following are true. (i) If $X$ is closed by $\zeta$, then for any finite $\F$-stopping time $\ulin T$, $X_T=\EE[\zeta|\F_{ \ulin T}]$. (ii) If $\ulin T\le \ulin S$ are two bounded $\F$-stopping times, then $\EE[X_{\ulin S}|\F_{\ulin T}]=X_{\ulin T}$.\label{OST}
	\end{Proposition}

 \section{Ensemble of Two Radial Loewner Chains}\label{section-Ensemble-radial}
\subsection{Deterministic ensemble} \label{section-det}
Let $m\in\N\cup\{0\}$. Let $w_1,w_2,v_1,\dots,v_m\in\R$ be such that $e^{iw_1},e^{iw_2},e^{iv_1},\dots,e^{iv_m}$ are {pairwise} distinct. For $j=1,2$, let $\ha w_j\in C([0,\ha T_j),\R)$ be a radial Loewner driving function with $\ha w_j{(0)}=w_j $.
Suppose $\ha w_j$ generates radial Loewner hulls $K_j(t)$,  radial Loewner maps $g_j(t,\cdot)$, covering radial Loewner hulls $\til K_j(t)$ and covering radial Loewner maps $\til g_j(t,\cdot)$, $0\le t<\ha T_j$.
 Let $\cal D$ denote the set of $(t_1,t_2)\in[0,\ha T_1)\times [0,\ha T_2)$ such that $\lin{K_1(t_1)}\cap \lin{K_2(t_2)}=\emptyset$ and $e^{iv_1 },\dots,e^{iv_m}\not\in \lin{K_1(t_1)}\cup \lin{K_2(t_2)}$. Then $\cal D$ is an HC region as in Definition \ref{Def-HC}, and we may define functions $T^{\cal D}_1$ and $T^{\cal D}_2$.
For $(t_1,t_2)\in\cal D$, let $K(t_1,t_2)=K_1(t_1)\cup K_2(t_2)$. Then $K(t_1,t_2)$ is also an $\D$-hull. Let
 $g((t_1,t_2),\cdot)=g_{K(t_1,t_2)}$, and $\mA(t_1,t_2)=\dcap(K(t_1,t_2))$. For $(t_1,t_2)\in\cal D$ and $j\ne k\in\{1,2\}$, let
$K_{j,t_k}(t_j)=g_k(t_k,K_j(t_j))$, and $g_{j,t_k}(t_j,\cdot)=g_{K_{j,t_k}(t_j)}$.
Let $\til K(t_1,t_2),\til K_{j,t_k}(t_j)\subset\HH$ be the pre-images of $K(t_1,t_2),K_{j,t_k}(t_j)$, respectively, under the map $e^i$.
Let $\til g(\ulin t,\cdot)$, $\ulin t\in\cal D$, be the unique family of maps, such that $\til g(\ulin t,z)$ is jointly continuous in $\ulin t,z$, $\til g(\ulin 0,\cdot)=\id$, and for each $\ulin t\in\cal D$, $\til g(\ulin t,\cdot):\HH\sem  \til K(\ulin t)\conf\HH$, and $e^i\circ \til g(\ulin t,\cdot)=  g(\ulin t,\cdot)\circ e^i$.  Define $\til g_{1,t_2}(t_1,\cdot)$ and $\til g_{2,t_1}(t_2,\cdot)$, $(t_1,t_2)\in\cal D$, similarly. %
Fix $j\ne k\in\{1,2\}$ and $s\in\{1,\dots,m\}$. Let $(t_1,t_2)\in{\cal D}$. We define the following real valued functions on $\cal D$:
$$ V_s(t_1,t_2)=\til g((t_1,t_2),v_s ),\quad V_{s,1}(t_1,t_2)=\til g'((t_1,t_2),v_s );$$ 
$$ W_j(t_1,t_2)=\til g_{k,t_j}(t_{k},\ha w_j(t_j)),\quad W_{j,h}(t_1,t_2)=\til g_{k,t_j}^{(h)}(t_{k},\ha w_j(t_j));$$ 
$$ W_{j,S}(t_1,t_2)=\frac{W_{j,3}(t_1,t_2)}{W_{j,1}(t_1,t_2)}-\frac 32 \Big(\frac{W_{j,2}(t_1,t_2)}{W_{j,1}(t_1,t_2)}\Big)^2. $$ 
Here  the prime and the superscript $(h)$ are respectively the partial derivative and $h$-th partial derivative w.r.t.\ the space variable: $v_s$ or $\ha w_j(t_j)$. 

Let $j\ne k\in\{1,2\}$ and $s\in\{1,\dots,m\}$. By \cite[Section 3.1]{Two-Green-interior}, for any $t_k\ge 0$, $K_{j,t_k}(t_j)$ and $g_{j,t_k}(t_j,\cdot)$, $0\le t_j<T^{\cal D}_j(t_k)$, are radial Loewner hulls and maps, respectively, driven by $W_j|^{k}_{t_k}$ with speed $(W_{j,1}|^k_{t_k})^2$. Moreover, we have the following formulas. 
\BGE \pa_{t_j} \mA =W_{j,1} ^2\pa t_j;\label{pamA}\EDE
\BGE\pa_{t_j} W_k=W_{j,1}^2 \cot_2(W_k-W_j)\pa t_j,\quad \pa_{t_j} V_s=W_{j,1}^2\cot_2(V_s-W_j)\pa t_j;\label{dWkVs}\EDE
\BGE \frac{\pa_{t_j} W_{k,1}}{W_{k,1}}=W_{j,1}^2\cot_2'(W_k-W_j)\pa t_j,\quad \frac{\pa_{t_j} V_{s,1}}{V_{s,1}}=W_{j,1}^2\cot_2'(V_s-W_j)\pa t_j;\label{dWj1Vs1}\EDE
\BGE \pa_{t_j} W_{k,S}=W_{j,1}^2W_{k,1}^2 \cot_2'''(W_k-W_j)\pa t_j;\label{pajWKS}\EDE
\BGE \pa_{t_j}\til g_{k,t_j}(t_k,\ha z)|_{\ha z=\ha w_j(t_j)}=-3\til g_{k,t_j}''(t_k,\ha w_j(t_j))=-3W_{j,2};\label{-3}\EDE
\BGE \frac{\pa_{t_j} \til g_{k,t_j}'(t_k,\ha z)|_{\ha z=\ha w_j(t_j)}}{\til g_{k,t_j}'(t_k,\ha z)|_{\ha z=\ha w_j(t_j)}}=\frac 12 \Big(\frac{W_{j,2}}{W_{j,1}}\Big)^2-\frac 43 \frac{W_{j,3}}{W_{j,1}}-\frac 16 (W_{j,1}^2-1). \label{1/2-4/3}\EDE

Finally, suppose that $\ha w_1$ and $\ha w_2$ generate radial Loewner curves $\eta_1$ and $\eta_2$, respectively. For each $j\ne k\in(1,2)$ and $\ulin t=(t_1,t_2)\in\cal D$, we define $\eta_j^{t_k}(t_j)=g_k(t_k,\eta_j(t_j))$. Then we have $K_{j,t_k}(t_j)=\Hull(\eta_j^{t_k}[0,t_j])$. So for any $t_k\ge 0$, $\eta_{j}^{t_k}(t_j)$, $0\le t_j<T^{\cal D}_j(t_k)$, is a radial Loewner curve driven by $W_j|^{k}_{t_k}$ with speed $(W_{j,1}|^k_{t_k})^2$.


\subsection{Commutation couplings}\label{section-Probability-measures}
\subsubsection{The statement}
We use the setup in the previous subsection. We view $(\ha w_1(t))_{0\le t<\ha T_1}$ and $(\ha w_2(t))_{0\le t<\ha T_2}$ as elements in $\Sigma:=\bigcup_{0<T\le\infty} C([0,T),\R)$. Now suppose $\ha w_1$ and $\ha w_2$ are random. Then their laws are probability measures on $\Sigma$.
Let  $\F^1$ and $\F^2$ be the $\R_+$-indexed filtrations respectively generated by $\ha w_1$ and $\ha w_2$.
More specifically, for $j=1,2$ and $t\in\R_+$, $\F^j_t$ is the $\sigma$-algebra generated by $\{ s<\ha T_j\}\cap \{\ha w(s)\in U\}$, $0\le s\le t$,  $U\in\cal{B}(\R)$.
We are going to prove the following.

\begin{Proposition}
  Let $\kappa\in(0,\infty)$  and $\ulin\rho=(\rho_1,\dots,\rho_m)\in\R^m$. We write  $e^{i\ulin v}$ for $(e^{iv_1},\dots,e^{iv_m})$. There is a coupling of two random radial Loewner curves $\eta_j(t)$, $0\le t<\ha T_j$, $j=1,2$, with radial Loewner driving functions $\ha w_j$ and  maps $g_j(t,\cdot)$ such that $\ha w_j(0)=w_j$ and the following holds. For $j=1,2$, $\eta_j$ is a radial SLE$_\kappa(2,\ulin\rho)$ curve  in $\D$ started from $e^{iw_1}$ aimed at $0$ with force points $(e^{iw_2},e^{i\ulin v})$. For any $j\ne k\in\{1,2\}$ and any $\F^k$-stopping time $\tau_k$ with $\tau_k<\ha T_k$, conditionally on $\F^k_{\tau_k}$, the $\eta_j^{\tau_k}(t_j)=g_k(\tau_k,\eta_j(t_j))$, $0\le t_j<T^{\cal D}_j(\tau_k)$, has the law of a time-change of a radial SLE$_\kappa(2,\ulin\rho)$ curve in $\D$ started from $g_k(\tau_k,e^{iw_j})$ aimed at $0$ with force points $(e^{i\ha w_k(\tau_k)},g_k(\tau_k,e^{iv_1}),\dots,g_k(\tau_k,e^{iv_m}))$ stopped at some stopping time. \label{prop-commu}
\end{Proposition}

The $\eta_1$ and $\eta_2$ in Proposition \ref{prop-commu} are said to  commute with each other.
The idea of the commutation relation between SLE curves is generated in \cite{Julien}.
The chordal counterpart of this proposition, with  $e^{iw_1},e^{iw_2},e^{i\ulin v},0$ replaced by $w_1,w_2,\ulin v,\infty$, follows easily from the imaginary geometry developed in \cite{MS1}.
 {In that case, one may construct the  two chordal SLE$_\kappa(2,\ulin\rho)$ curves as two flow lines of a GFF in $\HH$ with piecewise constant boundary data changing values at $w_1,w_2,v_1,\dots,v_m$, where the constants are determined by $\kappa,\ulin\rho$.}
Similarly,  Proposition \ref{prop-commu} could be probably proved by developing a radial version of imaginary geometry theory. We will not develop such theory here. Instead, we give another proof, which provides us some information {of the coupling measure, e.g., Lemma \ref{RN-Lemma},} that will be needed later. The proof in the special case that $m=2$ and $\rho_1=\rho_2=2$ were given in
\cite[Section 3]{Two-Green-interior}. Here we are interested in  the case that $m=2$ and $\rho_1=\rho_2=\kappa-4$.
For the reader's convenience  {and future reference}, we provide a complete proof of Proposition \ref{prop-commu} here.

For $j=1,2$, let $\Xi_j$ denote the space of simple crosscuts of $\D$ that separate $e^{iw_j} $ from $e^{iw_{3-j}},e^{iv_1} ,\dots,e^{iv_m},0$. For $j=1,2$ and $\xi_j\in\Xi_j$, let $\tau^j_{\xi_j}$ be the first time that $\eta_j$ hits the closure of $\xi_j$. 
Let $\Xi=\{(\xi_1,\xi_2)\in\Xi_1\times\Xi_2,\dist(\xi_1,\xi_2)>0\}$. For $\ulin\xi=(\xi_1,\xi_2)\in\Xi$, let $\tau_{\ulin\xi}=(\tau^1_{\xi_1},\tau^2_{\xi_2})$. We may pick a countable subset $\Xi^*$ of $\Xi$ such that for every $\ulin\xi=(\xi_1,\xi_2)\in\Xi$, there is $\ulin\xi^*=(\xi_1^*,\xi_2^*)$ such that $\xi_j^*$ disconnects $\xi_j$ from $0$, $j=1,2$. The significance of $\Xi$ and $\Xi^*$ is that $[\ulin 0, \tau_{\ulin\xi}]\subset \cal D$ for any $\ulin\xi\in\Xi$, and ${\cal D}=\bigcup_{\xi\in \Xi^*} [\ulin 0, \tau_{\ulin\xi}]$.

We call the joint law of $\ha w_1$ and $\ha w_2$ in Proposition \ref{prop-commu} a global commutation coupling. Such a measure is unique, if it exists, because the stated property determines the marginal law of $\eta_2$ (taking $\tau_1=0$) and the conditional law of $\eta_1$ given the part of $\eta_2$ up to any $\F^2$-stopping time $\tau_2$ that happens before $\eta_2$ ends.  Let $\ulin\xi=(\xi_1,\xi_2)\in\Xi$. If $\eta_1,\eta_2$ satisfy the properties in Proposition \ref{prop-commu} with the following modifications: (i) $\tau_k$ is required to be $\le \tau_{\xi_k}$, and (ii)  the time that $\eta_j$ hits $\eta_k[0,\tau_k]$ is replaced by $\tau_{\xi_j}$, then we call the joint law of the driving functions for $\eta_1,\eta_2$ a local commutation coupling within $\ulin\xi$. The global commutation coupling is automatically a local commutation coupling within $\ulin\xi$ for every $\ulin\xi \in\Xi$.  A local commutation coupling within $\ulin\xi$, when restricted to $\F^1_{\tau^1_{\xi_1}}\vee\F^2_{\tau^2_{\xi_2}}$, is unique.

\subsubsection{Two-variable local martingale}
For $j=1,2$, let $\PP^j_*$ denote the law of the driving function of a radial SLE$_\kappa(2,\ulin\rho)$ curve in $\D$ started from $e^{iw_j}$ aimed at $0$ with force points $(e^{iw_{3-j}},e^{i\ulin v})$, which is a probability measure on $\Sigma$. Let $\PP^{\ii}_*=\PP^1_*\times \PP^2_*$. The superscript $\ii$ means ``independence''. We may prove the existence of the global commutation coupling using the stochastic coupling technique developed in \cite{reversibility,duality}. It suffices to prove that there is a positive continuous process $\ha M_*$ defined on $\cal D$ with the following properties.
\begin{itemize}
\item $\ha M_*(\cdot,0)=\ha M_*(0,\cdot)\equiv 1$.
\item For any $\ulin\xi\in\Xi$, $\log(\ha M_*)$ is uniformly bounded on $[\ulin 0,\tau_{\ulin\xi}]$.
\item For any $\ulin\xi\in\Xi$, $\R_+^2\ni \ulin t\mapsto \ha M_*(\ulin t\wedge \tau_{\ulin\xi})$ is an $\R_+^2$-indexed martingale under $\PP^{\ii}_*$.
\item For any $\ulin\xi\in\Xi$, the probability measure $\PP^{\ulin\xi}_*$ on $\Sigma^2$ defined by $d\PP^{\ulin\xi}_*=\ha M_*(\tau_{\ulin\xi}) d\PP^{\ii}_*$ is a local commutation coupling within $\ulin\xi$.
\end{itemize}
We have the following proposition (cf.\ \cite[Sections 6 and 7]{reversibility}).

\begin{Proposition}
Whenever the above $\ha M_*$ exists, there exists a global commutation coupling, denoted by $\PP^c_*$, which agrees with $\PP^{\ulin\xi}_*$ on $\F^1_{\tau^1_{\xi_1}}\vee \F^2_{\tau^2_{\xi_2}}$ for any $\ulin\xi=(\xi_1,\xi_2)\in\Xi$. In particular, for any $\ulin\xi=(\xi_1,\xi_2)\in\Xi$, $\PP^c_*$ is absolutely continuous w.r.t.\ $\PP^{\ii}_*$ on $\F^1_{\tau^1_{\xi_1}}\vee \F^2_{\tau^2_{\xi_2}}$, and the Radon-Nikodym derivative is $\ha M_*(\tau_{\ulin\xi})$. \label{prop-M-coupling}
\end{Proposition}

For $j=1,2$, let $\PP^j_B$ denote the law of the process $w_j+\sqrt\kappa B(t)$, $0\le t<\infty$, where $B$ is a standard linear Brownian motion. Let $\PP^{\ii}_B=\PP^1_B\times \PP^2_B$. In order to construct $\ha M_*$, it suffices to find another positive continuous process $M_*$ on $\cal D$ with the following properties.
\begin{itemize}
  \item For any $\ulin\xi\in\Xi$, $\log(M_*)$ is uniformly bounded on $[\ulin 0,\tau_{\ulin\xi}]$.
  \item For any $\ulin\xi\in\Xi$, $\R_+^2\ni \ulin t\mapsto  M_*(\ulin t\wedge \tau_{\ulin\xi})$ is an $\R_+^2$-indexed martingale under $\PP^{\ii}_B$.
  \item For any $\ulin\xi\in\Xi$, the probability measure $\PP^{\ulin\xi}_*$ on $\Sigma^2$ defined by $d\PP^{\ulin\xi}_*= (M_*(\tau_{\ulin\xi})/M_*(\ulin 0)) d\PP^{\ii}_B$ is a local commutation coupling within $\ulin\xi$.
\end{itemize}
If such $M_*$ exists, then for any $j\in\{1,2\}$ and $\xi_j\in\Xi_j$, $\PP^j_*$ is absolutely continuous w.r.t.\ $\PP^j_B$ on $\F^j_{\tau^j_{\xi_j}}$, and the Radon-Nikodym derivative is $M_*|^{3-j}_0(\tau^j_{\xi_j})/M_*(\ulin 0)$. So for any $\ulin\xi=(\xi_1,\xi_2)\in\Xi$, $\PP^{\ii}_*$  is absolutely continuous w.r.t.\ $\PP^{\ii}_B$ on $\F^1_{\tau^1_{\xi_1}}\vee \F^2_{\tau^2_{\xi_2}}$, and the Radon-Nikodym derivative is $\frac{M_*(\tau^1_{\xi_1},0)M_*(0,\tau^2_{\xi_2})}{M_*(0,0) M_*(0,0)}$.  Thus, the $\ha M_*$ defined by $\ha M_*(t_1,t_2)=\frac{M_*(t_1,t_2)M_*(0,0)}{M_*(t_1,0)M_*(0,t_2)}$ satisfies the properties stated in the previous paragraph. Moreover, the global commutation coupling $\PP^c_*$ is absolutely continuous w.r.t.\ $\PP^{\ii}_B$ on $\F^1_{\tau^1_{\xi_1}}\vee \F^2_{\tau^2_{\xi_2}}$ for any $\ulin\xi=(\xi_1,\xi_2)\in\Xi$, and the Radon-Nikodym derivative is $M_*(\tau_{\ulin\xi})/M_*(\ulin 0)$. From now on, we are going to construct such $M_*$.

First suppose $(\ha w_1,\ha w_2)$ follows the law $\PP^{\ii}_{B}$. Then for two independent Brownian motions $B_1$ and $B_2$, we have $\ha w_j(t)=w_j+\sqrt\kappa B_j(t)$, $t\ge 0$, $j=1,2$. 


Recall the boundary scaling exponent $\bb$ and central charge $\cc$  in the literature: 
$$ \bb=\frac{6-\kappa}{2\kappa},\quad \cc=\frac{(3\kappa-8)(6-\kappa)}{2\kappa}.$$ 
Fix $j\ne k\in\{1,2\}$. Let $\tau_k$ be an $\F^k$-stopping time with $\tau_k<\ha T_k$. Let $\F^{(j,\infty)}$ denote usual augmentation of the $\R_+$-indexed filtration $(\F^j_t\vee \F^k_{\infty})_{t\ge 0}$. By independence, $B_j$ is an $\F^{(j,\infty)}$-Brownian motion. From now on, we will repeatedly apply It\^o's formula (cf.\ \cite{RY}), where all SDE are $\F^{(j,\infty)}$-adapted. By (\ref{-3}), we have
$$dW_j|^k_{\tau_k}(t)=W_{j,1}|^k_{\tau_k}(t)\sqrt\kappa dB_j(t)-\kappa\bb W_{j,2}|^k_{\tau_k}(t) d t_j. $$
To make the symbols less heavy, we write the SDE as
\BGE \pa_j W_j=W_{j,1}\sqrt\kappa \pa B_j-\kappa \bb W_{j,2}\pa t_j.\label{WB}\EDE
The symbol ``$\pa_j$'' in the SDE could be understood as the partial derivative w.r.t.\ the $j$-th variable.
We keep in mind that the $k$-th variable is fixed to be $\tau_k$.
Using (\ref{1/2-4/3}), we get
\BGE \frac{\pa_jW_{j,1}^{\bb}}{W_{j,1}^{\bb} }=\bb\frac{W_{j,2}}{W_{j,1}}\sqrt\kappa\pa B_j+\frac{\cc}6 W_{j,S}\pa t_j-\frac {\bb}6(W_{j,1}^2-1)\pa t_j.\label{paW1bc}\EDE
Using (\ref{dWkVs},\ref{dWj1Vs1}) we get for $r\ne s\in\{1,\dots,m\}$,
\BGE \frac{\pa_j \sin_2(W_k-V_r)^{\frac {\rho_r}\kappa}}{\sin_2(W_k-V_r)^{\frac {\rho_r}\kappa}}
 =-\frac{\rho_r}{2\kappa} [1+\cot_2(W_j-W_k)\cot_2(W_j-V_r)]W_{j,1}^2\pa t_j;\label{WkVr}\EDE
 \BGE \frac{\pa_j \sin_2(V_r-V_s)^{\frac {\rho_r\rho_s}{2\kappa}}}{\sin_2(V_r-V_s)^{\frac {\rho_r\rho_s}{2\kappa}}}
 =-\frac{\rho_r\rho_s}{4\kappa} [1+\cot_2(W_j-V_r)\cot_2(W_j-V_s)]W_{j,1}^2\pa t_j;\label{VrVs}\EDE
 \BGE \frac{\pa_j W_{k,1}^{\bb} }{W_{k,1}^{\bb}}
 =-\frac{\bb}{2} [1+\cot_2(W_j-W_k)\cot_2(W_j-W_k)]W_{j,1}^2\pa t_j;\label{WkWk}\EDE
 \BGE \frac{\pa_j V_{r,1}^{\frac {\rho_r(\rho_r+4-\kappa)}{4\kappa}}}{V_{r,1}^{\frac {\rho_r(\rho_r+4-\kappa)}{4\kappa}}}
 =-\frac{\rho_r(\rho_r+4-\kappa)}{8\kappa} [1+\cot_2(W_j-V_r)\cot_2(W_j-V_r)]W_{j,1}^2\pa t_j;\label{VrVr}\EDE
$$\frac{\pa_j \sin_2(W_j-W_k)}{\sin_2(W_j-W_k)}=\frac { \sqrt\kappa}2\cot_2(W_j-W_k)W_{j,1} \pa  B^k_{\tau_k}-\frac \kappa 2 \bb\cot_2 (W_j-W_k) W_{j,2} \pa t_j $$
$$+\frac 12 \cot_2 (W_j-W_k)^2 W_{j,1}^2 \pa t_j- \frac \kappa 8 W_{j,1}^2 \pa t_j;$$
$$\frac{\pa_j \sin_2(W_j-V_r)}{\sin_2(W_j-V_r)}=\frac { \sqrt\kappa}2\cot_2(W_j-V_r)W_{j,1} \pa  B_j -\frac \kappa 2 \bb\cot_2 (W_j-V_r) W_{j,2}\pa t_j$$
$$ +\frac 12 \cot_2 (W_j-V_r)^2 W_{j,1}^2 \pa t_j-  \frac \kappa 8 W_{j,1}^2 \pa t_j.$$
The last two formulas further imply that
$$\frac{\pa_j \sin_2(W_j-W_k)^{\frac {\rho_0}\kappa }}{\sin_2(W_j-W_k)^{\frac {\rho_0}\kappa }}=\frac{\rho_0}2\cot_2(W_j-W_k) W_{j,1}\frac{\pa B_j}{\sqrt\kappa}-\frac{\rho_0}2\bb \cot_2(W_j-W_k) W_{j,2}\pa t_j$$
\BGE +\frac{\rho_0(\rho_0+4-\kappa)}{8\kappa}\cot_2(W_j-W_k)^2 W_{j,1}^2 \pa t_j-\frac {\rho_0}8 W_{j,1}^2 \pa t_j,\quad \rho_0\in\R;\label{Wjk}\EDE
$$\frac{\pa_j \sin_2(W_j-V_r)^{\frac{\rho_r}\kappa}}{\sin_2(W_j-V_r)^{\frac{\rho_r}\kappa}}=  \frac {\rho_r}2\cot_2(W_j-V_r)W_{j,1}\frac{ \pa  B_j}{\sqrt\kappa} -\frac {\rho_r} 2\bb \cot_2 (W_j-V_r)  W_{j,2} \pa t_j$$
   \BGE  +\frac{\rho_r(\rho_r+4-\kappa)}{8\kappa}\cot_2(W_j-V_r)^2W_{j,1}^2\pa t_j-  \frac {\rho_r} 8 W_{j,1}^2 \pa t_j.\label{WjVr}\EDE
Define a function $I$ on $\cal D$ by
$$I(t_1,t_2)=\exp\Big(\int_0^{t_1} \int_0^{t_2} W_{1,1}(s_1,s_2)^2 W_{2,1}(s_1,s_2)^2 \cot_2'''(W_1(s_1,s_2)-W_2(s_1,s_2))ds_2ds_1\Big).$$
By (\ref{pajWKS}) and the fact that $W_{j,S}|^k_0\equiv 0$, we get
\BGE \frac{\pa_j I^{-\frac{\cc}6}}{I^{-\frac{\cc}6}}=-\frac{\cc}6 W_{j,S}\pa t_j.\label{paF}\EDE
Define $\ha \mA$ on $\cal D$ by $\ha \mA(t_1,t_2)=\mA(t_1,t_2)-t_1-t_2$. By (\ref{pamA}),
\BGE \pa_j \ha \mA=(W_{j,1}^2-1)\pa t_j.\label{pahamA}\EDE
Let $\rho_\Sigma=\sum_s \rho_s$. Define a positive continuous process $M_*$ on $\cal D$ by
$$M_*=e^{\frac{(\rho_\Sigma+2)(\rho_\Sigma+6)}{8\kappa}\mA}e^{\frac {\bb}6 \ha \mA} W_{1,1}^{\bb}W_{2,1}^{\bb} I^{-\frac{\cc}6}|\sin_2(W_1-W_2)|^{\frac 2\kappa } \prod_{1\le r<s\le m} |\sin_2(V_r-V_s)|^{\frac{\rho_r\rho_s}{2\kappa}}$$ \BGE \times \prod_{s\in\{1,\dots,m\}}\Big(V_{s,1}^{\frac{\rho_s(\rho_s+4-\kappa)}{4\kappa}} \prod_{j\in\{1,2\}}|\sin_2(W_j-V_s)|^{\frac{\rho_s}\kappa} \Big) . \label{M}\EDE
Combining (\ref{pamA},\ref{paW1bc}-\ref{pahamA}), where $\rho_0$ is set to be $2$ in (\ref{Wjk}),  we get
\BGE \frac{\pa_j M_*}{M_*}=\bb\frac{W_{j,2}}{W_{j,1}}\sqrt\kappa\pa B_j+\cot_2(W_j-W_k)W_{j,1} \frac{\pa  B_j}{\sqrt\kappa}+\sum_r \frac {\rho_r}2\cot_2(W_j-V_r)W_{j,1}\frac{ \pa  B_j}{\sqrt\kappa}  .\label{pajM}\EDE
This means that $M_*|^k_{\tau_k}(t)$ is a continuous $\F^{(j,\infty)}$-local martingale.

\begin{Lemma}
	For any $\ulin \xi\in\Xi$, $|\log(M_*)|$ is uniformly bounded on $[\ulin 0,\tau_{\ulin\xi}]$ by a  constant depending only on $\kappa,\ulin\xi,e^{iw_1},e^{iw_2},e^{iv_1},\dots,e^{iv_m}$. \label{uniform}
\end{Lemma}
\begin{proof}
  It suffices to show that $\mA$, $\log|\sin_2(W_1-W_2)|$, $\log|\sin_2(W_j-V_s)|$,  $\log|\sin_2(V_r-V_s)|$,  $\log V_{s,1}$, $\log|W_{j,1}|$, $j=1,2$, $r\ne s\in\{1,\dots,m\}$,	are all bounded in absolute value on $[\ulin 0,\tau_{\ulin\xi}]$ by constants depending only on $\ulin\xi,e^{iw_1},e^{iw_2},e^{iv_1},\dots,e^{iv_m}$. These statements were all proved in the proof of  Lemma 3.1 of \cite{Two-Green-interior}.
\end{proof}

Let $\ulin \xi=(\xi_1,\xi_2)\in\Xi$. Let $j\ne k\in\{1,2\}$. Let $\tau_k\le \tau_k'$ be two $\F^k$-stopping times such that $\tau_k\le \tau^k_{\xi_k}$. By Lemma \ref{uniform} and (\ref{pajM}),  $M|^k_{\tau_k}(\cdot\wedge \tau_{\xi_j})$ under $\PP^{\ii}_B$ is an $\F^{(j,\infty)}$-martingale closed by $M|^k_{\tau_k}(\tau^j_{\xi_j})$. The filtration $\F^{(j,\infty)}$ in the statement can be replaced by $(\F^j_t\vee \F^k_{\tau_k'})_{t\ge 0}$  since $M|^k_{\tau_k}(\cdot\wedge \tau_{\xi_j})$ is adapted to it. Let $\F$ denote the $\R_+^2$-indexed filtration generated by $\F^1$ and $\F^2$. Applying the above result twice: first to $k=1$ and $\tau_1=\tau_1'=\tau^1_{\xi_1}$, and then to $k=2$, $\tau_2=t_2\wedge \tau^2_{\xi_2}$ and $\tau_2'=t_2$,   we get the following lemma (cf.\ \cite[Corollary 3.2]{Two-Green-interior}).


\begin{Lemma}
  For any $\ulin\xi=(\xi_1,\xi_2)\in\Xi$, $\R_+^2\ni \ulin t\mapsto M_*(\ulin t\wedge  \tau_{\ulin\xi})$ is an $\F$-martingale under $\PP^{\ii}_B$ closed by $M_*(\tau_{\ulin\xi})$. In particular, we have $\EE^{\ii}_B[M_*(\tau_{\ulin\xi})]=M_*(\ulin 0)$. \label{M-mtgl}
\end{Lemma}

\subsubsection{Construction of the coupling measure}\label{subsubsection-construction}

Fix $\ulin\xi\in\Xi$. We may now define a new probability measure $\til\PP_*$ by $d\til\PP_*/d\PP^{\ii}_B= M_*(\tau_{\ulin\xi})/M_*(\ulin 0)$. Fix $j\ne k\in\{1,2\}$ and an $\F^k$-stopping time $\tau_k$ with $\tau_k\le \tau^k_{\xi_k}$. Define
$$\til B_j(t)=B_j(t)-\int_0^{t\wedge \tau^j_{\xi_j}}\frac{2}{2\sqrt\kappa} \cot_2(W_j|^k_{\tau_k}(s)-W_k|^k_{\tau_k}(s))W_{j,1}|^k_{\tau_k}(s)^2ds$$
$$-\sum_{l=1}^m \int_0^{t\wedge \tau^j_{\xi_j}}\frac{\rho_s}{2\sqrt\kappa} \cot_2(W_j|^k_{\tau_k}(s)-V_l|^k_{\tau_k}(s))W_{j,1}|^k_{\tau_k}(s)^2 ds,\quad t\ge 0.$$
By (\ref{pajM}) and the Girsanov Theorem,  $\til B_j$ is an $\F^{(j,\infty)}$-Brownian motion under $\til\PP_*$. By (\ref{WB}) and the definition of $\til B$, we find that the $W_j$ (with $t_k$ fixed to be $\tau_k$) satisfies the SDE
$$\pa_j W_j=W_{j,1}\sqrt\kappa \pa \til B+  \cot_2({W_j-W_k})W_{j,1}^2\pa t_j +\sum_{s=1}^m \frac{\rho_s}2  \cot_2({W_j-V_s})W_{j,1}^2\pa t_j,\quad 0\le t_j\le \tau^j_{\xi_j}.$$
Since $\eta_j^{\tau_k}(t_j)$, $0\le t_j<T^{\cal D}_j(\tau_k)$, is a radial Loewner curve with speed $(W_{j,1}|^k_{\tau_k})^2$ driven by $W_j|^k_{\tau_k}$, the SDE implies that, under $\til\PP_*$, conditionally on $\F^k_{\tau_k}$, the $\eta_j^{\tau_k}$ up to $\tau^j_{\xi_j}$ has the law of a time-change of a radial SLE$_\kappa(2,\ulin\rho)$ curve in $\D$ started from $g_k(\tau_k, e^{iw_j})$ aimed at $0$ with force points $e^{iw_k(\tau_k)},g_k(\tau_k, e^{i\ulin v})$. This means that $\til\PP_*$ is a local commutation coupling within $\ulin\xi$. Thus, the $M_*$ satisfies all required properties. By Proposition \ref{prop-M-coupling}, we then complete the proof of Proposition \ref{prop-commu}. Moreover, we get the following proposition.

\begin{Proposition}
  For any $\ulin\xi=(\xi_1,\xi_2)\in\Xi$, the global coupling measure in Proposition \ref{prop-commu} is absolutely continuous w.r.t.\ $\PP^{\ii}_B$ on $\F^1_{\tau^1_{\xi_1}}\vee \F^2_{\tau^2_{\xi_2}}$, and the Radon-Nikodym derivative is $M_*(\tau_{\ulin\xi})/M_*(\ulin 0)$, where $M_*$ is defined by (\ref{M}). \label{prop-P*}
\end{Proposition}

{
\begin{Lemma}
  Suppose $\kappa>4$, $m=2$, $\rho_1=\rho_2=\kappa-4$, and $w_1>v_1>w_2>v_2>w_1-2\pi$. Let $\eta_1$ and $\eta_2$ be given by Proposition \ref{prop-commu}. Let ${\cal D}=\{(t_1,t_2)\in\R_+^2:\eta_1[0,t_1]\cap \eta_2[0,t_2]=\emptyset\}$. Then a.s.\ ${\cal D}=\R_+^2$.  \label{D=R^2}
\end{Lemma}}
\begin{proof}
Let $\ha w_j$ be the driving function of $\eta_j$, $j=1,2$. 
 Since the force values $\rho_1=\rho_2=\kappa-4>\frac \kappa 2-2$, and $2>0$, by Proposition \ref{transience}, a.s.\ $\eta_j$  does not intersect $A_{3-j}$, and so has lifetime $\infty$. Fix a deterministic time $t_2\in\N$. Conditionally on $\F^2_{t_2}$, the $g_2(t_2,\cdot)$-image of the part of $\eta_1$ before hitting $\eta_2[0,t_2]$ has the law of a time-change of a radial SLE$_\kappa(2,\kappa-4,\kappa-4)$ curve in $\D$ started from $g_2(t_2,e^{iw_1})$ aimed at $0$ with force points $e^{i\ha w_2(t_2)},g_2(t_2,e^{iv_1}),g_2(t_2,e^{iv_2})$. If $\eta_1$ does hit $\eta_2[0,t_2]$, then the (time-changed) radial SLE$_\kappa(2,\kappa-4,\kappa-4)$ curve $\eta_1^{t_2}$ hits the arc on $\pa\D$ with endpoints $g_2(t_2,e^{iv_1}),g_2(t_2,e^{iv_2})$ that contains $e^{i\ha w_2(t_2)}$. By Proposition \ref{transience} again, the latter event has probability zero. Thus,  $\PP^c_*$-a.s.\ $\eta_1$ does not intersect $\eta_2[0,t_2]$, which implies that $\R_+\times [0,t_2]\subset \cal D$. Since this holds for every $t_2\in\N$, we get the conclusion.
\end{proof}

{
We use  $\PP^c_*$ to denote the joint law of the driving functions $\ha w_1$ and $\ha w_2$ of the $\eta_1$ and $\eta_2$ in Lemma \ref{D=R^2}. The superscript $c$ stands for coupling, and the subscript $*$ means that both curves end at $0$. }

For the $w_j$ and $v_j$ as above, we will need a different kind of commutation coupling  as follows. Define $M_c$ on $\cal D$ by
\BGE M_c=e^{\frac{(\kappa-6)(\kappa-2)}{8\kappa}\mA}e^{\frac {\bb}6 \ha \mA} W_{1,1}^{\bb}W_{2,1}^{\bb}I^{-\frac{\cc}6}|\sin_2(W_1-W_2)|^{\frac{\kappa-6}\kappa} .\label{Mc}\EDE
Let $\Xi$ be as usual for $e^{iw_j},e^{iv_j}$, $j=1,2$. Then Lemma \ref{uniform} holds for $M_c$.
Fix $j\ne k\in\{1,2\}$. Let $\tau_k$ be an $\F^k$-stopping time.
Using (\ref{pamA},\ref{paW1bc},\ref{WkWk},\ref{Wjk} (for $\rho_0=\kappa-6$),\ref{paF},\ref{pahamA}), we see that $M_c$, with $t_k$ fixed to be $\tau_k$, is a continuous local martingale under $\PP^{\ii}_B$ satisfying the following SDE:
$$ \frac{\pa_j M_c}{M_c}=\bb\frac{W_{j,2}}{W_{j,1}}\sqrt\kappa\pa B_j+\frac{\kappa-6}2\cot_2(W_j-W_k)W_{j,1} \frac{\pa  B_j}{\sqrt\kappa}.$$ 
So Lemma \ref{M-mtgl} also holds for $M_c$.

Fix $\ulin\xi=(\xi_1,\xi_2)\in\Xi$. One may define a probability measure $\til \PP_c$ by $d\til \PP_c/d\PP^{\ii}_B=M_c(\tau_{\ulin\xi})/M_c(\ulin 0)$. Fix $j\ne k\in\{1,2\}$ and an $\F^k$-stopping time $\tau_k$ with $\tau_k\le \tau^k_{\xi_k}$. Using the Girsanov Theorem, one can show that, under $\til \PP_c$, conditionally on $\F^k_{\tau_k}$, the $g_k(\tau_k,\cdot)$ image of the part of $\eta_j$ up to $\tau^j_{\xi_j}$ has the law of a time-change of a radial SLE$_\kappa(\kappa-6,0,0)$ curve in $\D$ started from $g_k(\tau_k, e^{iw_j})$ aimed at $0$ with force points $e^{iw_k(\tau_k)},g_k(\tau_k,e^{iv_1}),g_k(\tau_k,e^{iv_2})$. So we get the following proposition.

\begin{Proposition}
  Let $\kappa\in(0,\infty)$. There is a coupling of two random radial Loewner curves $\eta_j(t)$, $0\le t<\ha T_j$, $j=1,2$, with radial Loewner driving functions $\ha w_j$ and maps $g_j(t,\cdot)$ such that $\ha w_j(0)=w_j$ and the following holds.
  For $j=1,2$, $\eta_j$ is   a radial SLE$_\kappa(\kappa-6,0,0)$ curve  in $\D$ started from $e^{iw_j}$ aimed at $0$ with force points $e^{iw_{3-j}},e^{iv_1},e^{iv_2}$. For any $j\ne k\in\{1,2\}$ and any $\F^k$-stopping time $\tau_k$ with $\tau_k<\ha T_k$, conditionally on $\F^k_{\tau_k}$, the $\eta_j^{\tau_k}(t_j)$, $0\le t_j<T^{\cal D}_j(\tau_k)$, has the law of a time-change of a radial SLE$_\kappa(\kappa-6,0,0)$ curve in $\D$ started from $g_k(\tau_k,e^{iw_j})$ aimed at $0$ with force points $e^{i\ha w_k(\tau_k)},g_k(\tau_k,e^{iv_1}),g_k(\tau_k,e^{iv_2})$ stopped at the first time that it separates $0$ from any of the three force points in $\D$. Moreover,  for any $\ulin\xi=(\xi_1,\xi_2)\in\Xi$, the joint law of $\ha w_1$ and $\ha w_2$ is absolutely continuous w.r.t.\ $\PP^{\ii}_B$ on  $\F^1_{\tau^1_{\xi_1}}\vee \F^2_{\tau^2_{\xi_2}}$, and the Radon-Nikodym derivative is  $M_c(\tau_{\ulin\xi})/M_c(\ulin 0)$, where $M_c$ is defined by (\ref{Mc}).
  \label{prop-commu-c}
\end{Proposition}

By \cite{SW}, for $j=1,2$, the $\eta_j$ in Proposition \ref{prop-commu-c} is a time-change of a chordal SLE$_\kappa$ in $\D$ from $e^{iw_j}$ to $e^{iw_{3-j}}$ stopped at the first time that it separates $0$ from any of $e^{iw_{3-j}},e^{iv_1},e^{iv_2}$ in $\D$.  {When $\kappa\in(0,8]$, the coupling of $\eta_1$ and $\eta_2$ in Proposition \ref{prop-commu-c} can be easily constructed using the DMP and reversibility of chordal SLE$_\kappa$. For the construction, we start with a chordal SLE$_\kappa$ in $\D$ from $e^{iw_1}$ to $e^{iw_2}$, let $\eta_1$ and $\eta_2$ be respectively the part of $\eta$ and the reversal of $\eta$ up to the first time that it separates $0$ from any of the three force points in $\D$. But such a construction does not provide us the Radon-Nikodym derivative process $M_c$.
}


Let {$\PP^c_c$} denote the joint law of {the driving functions} $\ha w_1$ and $\ha w_2$  {of the $\eta_1$ and $\eta_2$} obtained from Proposition \ref{prop-commu-c}. {The superscript and subscript respectively stand  for coupling and chordal.
}

{We are mostly interested in the case $\kappa\in(4,8)$. In that case, under $\PP^c_c$ the $\eta_1$ and $\eta_2$  are parts of the same chordal SLE$_\kappa$ curve in $\D$ from $e^{iw_1}$ to $e^{iw_2}$. Under $\PP^c_*$ the $\eta_1$ and $\eta_2$  both end at $0$, and $\eta_1\cup A_1$ is disjoint from $\eta_2\cup A_2$, where for $j=1,2$, $A_j$ is the connected component of $\D\sem \{e^{iv_1},e^{iv_2}\}$ that contains $e^{iw_j}$. Thus, under $\PP^c_*$, $0$ is a cut point of $\eta_1\cup \eta_2\cup A_1\cup A_2\cup\{0\}$. Intuitively, we may understand $\PP^c_*$ as the chordal SLE$_\kappa$ curve in $\D$ from $e^{iw_1}$ to $e^{iw_2}$ conditioned on the singular event that $0$ is a cut point of $\eta\cup A_1\cup A_2$. See Figure \ref{figure1}.
}

\begin{figure}[t]
\centering
\includegraphics[scale=0.80]{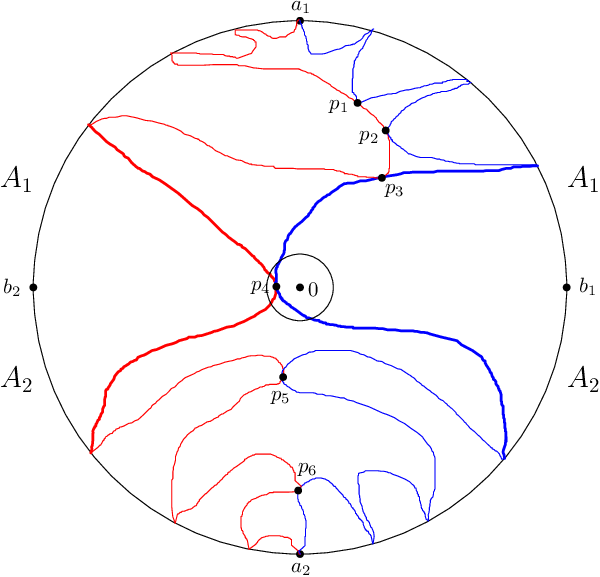}
{\caption{The above figure is a sketch of the left boundary (red) and right boundary (blue) of a chordal SLE$_\kappa$ curve $\gamma$ ($\kappa\in(4,8)$) in $\D$ connecting $a_1$ with $a_2$, where each boundary is a boundary-touching simple curve composed of countably many crosscuts of $\D$, exactly one of which connects the upper semicircle $A_1$ with the lower semicircle $A_2$. The $p_1,\dots,p_6$ illustrated are intersections of the two boundaries, and are cut points of $\gamma$, but only the intersection $p_4$ of the two crosscuts connecting $A_1$ and $A_2$  is a cut point of $\gamma\cup A_1\cup A_2$. The small circle has radius $r$. If one conditions $\gamma$ on the event that $\gamma\cup A_1\cup A_2$ has a cut point in $\{|z|<r\}$, and lets $r\to 0$, then the result is a curve from $a_1$ to $a_2$ passing through its cut point $0$, and the two sub-curves (respectively from $a_1,a_2$ to $0$) jointly follow the law $\PP^c_*$. }}
\label{figure1}
\end{figure}

Let $M_{*\to c}=M_c/M_*$, which can be expressed in terms of the $\alpha_0$ and $\til G$  defined in (\ref{alpha0},\ref{G}):
\BGE M_{*\to c}=e^{-\alpha_0 \mA }\til G(W_1,W_2,V_1,V_2)^{-1}. \label{M*c}
\EDE



\begin{Lemma} Suppose $\kappa\in(4,8)$.
  For any $\F$-stopping time $\ulin T$, $\PP^c_c$ restricted to $\F_{\ulin T}\cap \{\ulin T\in \cal D\}$ is absolutely continuous w.r.t.\ $\PP^c_*$, and the RN derivative is $M_{*\to c}(\ulin T)/M_{*\to c}(\ulin 0)$. In other words, if $A\in \F_{\ulin T}$ and $A\subset \{\ulin T\in\cal D\}$, then $\PP^c_c[A]=\EE^c_*[{\bf 1}_A M_{*\to c}(\ulin T)/M_{*\to c}(\ulin 0)]$.\label{RN-Lemma}
\end{Lemma}
\begin{proof}
Let $\ulin\xi=(\xi_1,\xi_2) \in\Xi$ and $\ulin t=(t_1,t_2) \in\R_+^2$. By Lemma \ref{M-mtgl}, Proposition \ref{prop-P*},  and the facts $\F^1_{t_1\wedge \tau^1_{\xi_1}}\vee \F^2_{t_2\wedge \tau^2_{\xi_2}}\subset \F_{\ulin t }$ and that $M_*(\ulin t\wedge \tau_{\ulin\xi})$ is $\F^1_{t_1\wedge \tau^1_{\xi_1}}\vee \F^2_{t_2\wedge \tau^2_{\xi_2}}$-measurable, we see that
$$\frac{d\PP^c_*|{\F^1_{t_1\wedge \tau^1_{\xi_1}}\vee \F^2_{t_2\wedge \tau^2_{\xi_2}}}}{d\PP^{\ii}_B|\F^1_{t_1\wedge \tau^1_{\xi_1}}\vee \F^2_{t_2\wedge \tau^2_{\xi_2}}}=\frac{\EE\Big[M_{*}(\tau_{\ulin\xi})\Big|\F^1_{t_1\wedge \tau^1_{\xi_1}}\vee \F^2_{t_2\wedge \tau^2_{\xi_2}}\Big]}{M_{*}(\ulin 0)}=\frac{M_{*}(\ulin t\wedge \tau_{\ulin\xi})}{M_{*}(\ulin 0)}.$$
A similar formula holds with $\PP^c_c$ and $M_c$ respectively in place of $\PP^c_*$ and $M_*$. So we get
\BGE \frac{d\PP^c_c|{\F^1_{t_1\wedge \tau^1_{\xi_1}}\vee \F^2_{t_2\wedge \tau^2_{\xi_2}}}}{d\PP^c_*|{\F^1_{t_1\wedge \tau^1_{\xi_1}}\vee \F^2_{t_2\wedge \tau^2_{\xi_2}}}}=\frac{M_{*\to c}(\ulin t\wedge \tau_{\ulin\xi})}{M_{*\to c}(\ulin 0)}.\label{RNc*xi}\EDE

Let $\ulin T=(T_1,T_2)$ be an $\F$-stopping time. Let $\ulin\xi \in\Xi^*$. Fix $A\in \F_{\ulin T}$ with $A\subset\{\ulin T < \tau_{\ulin \xi}\}$. 
  For $n\in\N$, let $\ulin T^n=(T_1^n, T_2^n)$, where $T_j^{n}=\frac{\lceil 2^n T\rceil}{2^n}$, $j=1,2$. It is easy to see that each $\ulin T^{n}$ is an  $\F$-stopping time, and $\ulin T^n\downarrow \ulin T$.  Let $A_n=A\cap \{\ulin T^{n}<\tau_{\ulin \xi}\}$, $n\in\N$. Then $A_n\uparrow A$.  Fix $n\in\N$. By Proposition \ref{separable-lemma}, $A_{n}\in \F_{\ulin T^{n}}$. We know that $\ulin T^{n}$ takes values in $\frac 1{2^n}\Z_+^2$. Let $\ulin s=(s_1,s_2) \in \frac 1{2^n}\Z_+^2$ and $A_{n,\ulin s}=A_{n}\cap \{\ulin T^{n}=\ulin s\}$. By Proposition \ref{T<S}, $A_{n,\ulin s}\in \F_{\ulin s}$. Since $\F^j_{s_j}\cap \{s_j\le s_j\wedge \tau^j_{\xi_j}\}\subset \F^j_{s_j\wedge \tau^j_{\xi_j}}$, $j=1,2$, by a monotone class argument we get $\F_{\ulin s}\cap \{\ulin s\le \ulin s\wedge \tau_{\ulin\xi}\}\subset \F^1_{t_1\wedge \tau^1_{\xi_1}}\vee \F^2_{t_2\wedge \tau^2_{\xi_2}}$.
  Since $A_{n,\ulin s}\in \F_{\ulin s}\cap\{\ulin s<\tau_{\ulin\xi}\}$, we see that $A_{n,\ulin s}\in \F^1_{t_1\wedge \tau^1_{\xi_1}}\vee \F^2_{t_2\wedge \tau^2_{\xi_2}}$.
  By (\ref{RNc*xi}),
  $$\PP^c_c[A_{n,\ulin s}]=\EE^c_*[{\bf 1}_{A_{n,\ulin s}} M_{*\to c}(\ulin s\wedge \tau_{\ulin\xi})/M_{*\to c}(\ulin 0)]=\EE^c_*[{\bf 1}_{A_{n,\ulin s}} M_{*\to c}(\ulin T^{n})/M_{*\to c}(\ulin 0)],$$
  where the last equality holds because $\ulin T^{n}=\ulin s\le \tau_{\ulin\xi}$ on $A_{n,\ulin s}$. Summing up over $\ulin s \in \frac 1{2^n}\Z_+^2$, we get
  $\PP^c_c[A_n]= \EE^c_*[{\bf 1}_{A_{n}} M_{*\to c}(\ulin T^{n})/M_{*\to c}(\ulin 0)]$. Sending $n\to\infty$ and using dominated convergence theorem, we get $\PP^c_c[A]= \EE^c_*[{\bf 1}_{A } M_{*\to c}(\ulin T)/M_{*\to c}(\ulin 0)]$. Here we use the facts that $\ulin T^n<\tau_{\ulin\xi}$ on $A_n$ and that $\log M_{*\to c}$ is uniformly bounded on $[0,\tau_{\ulin\xi}]$. This means that, for any $\xi\in\Xi^*$, $\PP^c_c$ restricted to $\F_{\ulin T}\cap \{\ulin T<\tau_{\ulin\xi}\}$ is absolutely continuous w.r.t.\ $\PP^c_*$, and the RN derivative is $M_{*\to c}(\ulin T)/M_{*\to c}(\ulin 0)$. The conclusion of the lemma then follows since $\bigcup_{\ulin\xi\in\Xi^*}   \{\ulin T<\tau_{\ulin\xi}\}=\{\ulin T\in\cal D\}$.
\end{proof}

\subsection{A time curve  in the time region}\label{section-u}
Suppose $\eta_1$ and $\eta_2$ are random radial Loewner curves driven by $\ha w_1$ and $\ha w_2$, which jointly follow one of the three laws $\PP^{\ii}_B,\PP^c_*,\PP^c_c$.
Let $\theta=V_1-V_2$ and $Z_j=W_j-V_j$, $j=1,2$. By \cite[(4.1)]{Two-Green-interior},
$$\pa_j \theta =\frac{-W_{j,1}^2 \sin_2(\theta)}{\sin_2(W_j-V_1)\sin_2(W_j-V_2)}\pa t_j,\quad j=1,2.$$
When $j=2$, using  $0>\sin_2(W_j-V_1)\sin_2(W_j-V_2)\ge -\sin(\theta/4)^2$ and (\ref{pamA}), we get
\BGE \pa_2 \theta \ge 2W_{2,1}^2 \cot(\theta/4)\pa t_j=2 \cot(\theta/4) \pa_2 \mA.\label{pathetamA}\EDE

Recall that most of the processes we have encountered are defined on the time region ${\cal D}\subset \R_+^{{2}}$.
From now on, suppose $v_1-v_2=\pi$. Then $\theta(\ulin 0)=\pi$. By  \cite[Section 4]{Two-Green-interior}, there exists a continuous increasing curve $\ulin u=(u_1,u_2):\R_+\to \lin{\cal D}$ with $\ulin u(0)=\ulin 0$, and {an extended number} $T^u\in(0,\infty]$, such that $\ulin u$ is strictly increasing and takes values  in $\cal D$ on $[0,T^u)$, takes constant values in $\pa D$ on $[T^u,\infty)$, and for any $t\in[0,T^u)$, $\mA(\ulin u(t))=t$ and $\theta(\ulin u(t))=\pi$.
We use such curve $\ulin u$ to obtain a one-time-parameter process $X^u:=X\circ \ulin u$ from any two-time-parameter process $X$ on $\cal D$. We have the following facts.
\begin{itemize}
  \item For $0\le t<T^u$, $\mA^u(t)=t$, $\theta^u(t)=\pi$, and $Z_j^u\in(0,\pi)$, $j=1,2$.
  \item For any $j\in\{1,2\}$ and $t\ge 0$, $u_j(t)\le t$.
  \item $\ulin u$ is differentiable with positive derivatives on $[0,T^u)$ that satisfy
  \BGE (W_{j,1}^u)^2 u_j'=\frac{\sin(Z_j^u)}{\sin(Z_1^u)+\sin(Z^u_2)},\quad \mbox{on }[0,T^u),\quad j=1,2.\label{uj'}\EDE
\item For any deterministic time $t\in\R_+$, $\ulin u(t)$ is an $\F$-stopping time.
\end{itemize}

Using the last property, we define an $\R_+$-indexed filtration $\F^u$ by $\F^u_t=\F_{\ulin u(t)}$, $t\ge 0$. For $\xi=(\xi_1,\xi_2)\in\Xi$, let $\tau^u_{\ulin\xi}$ denote the first $t\ge 0$ such that $u_1(t)=\tau^1_{\xi_1}$ or $u_2(t)=\tau^2_{\xi_2}$, whichever comes first. Then for any $\xi\in\Xi$, $\tau^u_{\ulin\xi}$ is an $\F^u$-stopping time; $\ulin u(\tau^u_{\ulin\xi})$ is an $\F$-stopping time; and for every deterministic $t\in\R_+$, $\ulin u(t\wedge \tau^u_{\ulin\xi})$ is an $\F$-stopping time.

First, suppose $(\ha w_1,\ha w_2)$ follows the law $\PP^{\ii}_B$. Then there are independent Brownian motions $B_1$ and $B_2$ such that $\ha w_j(t_j)=w_j+\sqrt\kappa B_j(t_j)$, $t_j\ge 0$, $j=1,2$. So we get five continuous $\F$-martingales: $\ha w_j(t_j)$, $\ha w_j(t_j)^2-\kappa t_j$, $j=1,2$, and $\ha w_1(t_1)\ha w_2(t_2)$. Using Proposition \ref{OST} and the fact that for any $t\in\R_+$,  $\ulin u(t)$ is an $\F$-stopping time bounded by $(t,t)$, we see that $\ha w^u_j(t)$, $\ha w^u_j(t)^2-\kappa u_j(t)$, $j=1,2$, and $\ha w^u_1(t)\ha w^u(t)$ are all $\F^u$-martingales. Note that $\ha w^u_j(t)=\ha w_j(u_j(t))$. So we get quadratic variation and covariation for $\ha w_j^u$, $j=1,2$:
\BGE \langle \ha w_j^u\rangle_t=\kappa u_j(t), \quad j=1,2,\quad \langle \ha w_1^u,\ha w_2^u\rangle \equiv 0.\label{quadratic-w}\EDE

Fix $\ulin\xi=(\xi_1,\xi_2)\in\Xi$. Let $M\in\{M_*,M_c\}$. We have known that $M(\cdot\wedge \tau_{\ulin\xi})$ is a bounded $\F$-martingale (under $\PP^{\ii}_B$). From Proposition \ref{OST}, $M(\ulin u(\cdot)\wedge \tau_{\ulin\xi})$ is an $\F^u$-martingale. Since $\tau^u_{\ulin\xi}$ is an $\F^u$-stopping time, we see that $M(\ulin u(\cdot\wedge \tau^u_{\ulin\xi})\wedge \tau_{\ulin\xi})=M(\ulin u(\cdot\wedge \tau^u_{\ulin\xi})) =M^u(\cdot\wedge \tau^u_{\ulin\xi})$
is an $\F^u$-martingale. Here the first equality holds because $\ulin u(t\wedge \tau^u_{\ulin\xi})\le \tau_{\ulin\xi}$. Since $[0,T^u)=\bigcup_{\xi\in\Xi^*} [0,\tau^u_{\ulin\xi}]$ and $\Xi^*$ is countable, we conclude that $M_*^u$ and $M_c^u$ are $\F^u$-local martingales with lifetime $T^u$ under the probability measure $\PP^{\ii}_B$.

Next, we compute the SDE for $M_*^u$ in terms of $\ha w_1^u$ and $\ha w_2^u$. By (\ref{M}) we may express $M_*^u$ as a product of several factors. Among these factors, $\sin_2(W_1^u-W_2^u)^{\frac 2\kappa}$, $(W_{j,1}^u)^b$ and $|\sin_2(W_j^u-V_s^u)|^{1-\frac 4\kappa}$, $j,s\in\{1,2\}$, contribute the diffusion part of $M_*^u$, and other factors are differentiable in $t$. For $j\ne k\in\{1,2\}$ and $s\in\{1,2\}$, using (\ref{dWkVs},\ref{-3},\ref{1/2-4/3}), we get the following SDEs:
\BGE dW_j^u=W_{j,1}^u d\ha w_j^u-\kappa\bb W_{j,2} u_j' dt+\cot_2(W_j^u-W_k^u)(W_{k,1}^u)^2u_k' dt, \label{dWju}\EDE
$$\frac{d(W_{j,1}^u)^{\bb}}{(W_{j,1}^u)^{\bb}}=\bb\frac{W_{j,2}^u}{W_{j,1}^u}\,d\ha w_j^u+\mbox{drift terms},$$
$$\frac{d\sin_2(W_j^u-V_s^u)^{1-\frac 4\kappa}}{\sin_2(W_j^u-V_s^u)^{1-\frac 4\kappa}}=\frac{\kappa-4}{2\kappa} \cot_2(W_j^u-V_s^u)W_{j,1}^u d\ha w_j^u +\mbox{drift terms},$$
$$\frac{d\sin_2(W_1^u-W_2^u)^{ \frac 2\kappa}}{\sin_2(W_1^u-W_2^u)^{\frac 2\kappa}}=\frac{1}{\kappa} \cot_2(W_j^u-W_k^u)[W_{j,1}^u d\ha w_j^u-W_{k,1}^u d\ha w_k^u] +\mbox{drift terms}.$$
Since we already know that $M_*^u$, $\ha w_1^u$ and $\ha w_2^u$ are $\F^u$-local martingales, we get the SDE:
\BGE \frac{dM_*^u}{M_*^u}=\sum_{j=1}^2\Big[\bb\frac{W_{j,2}^u}{W_{j,1}^u}+\frac 1\kappa \cot_2(W_j^u-W_{3-j}^u)W_{j,1}^u+\sum_{s=1}^2 \frac{\kappa-4}{2\kappa} \cot_2(W_j^u-V_s^u)W_{j,1}^u \Big] d\ha w_j^u.\label{dMu}\EDE

\begin{Lemma}
  Under $\PP^c_*$,  for $j=1,2$, $\ha w_j^u$ satisfies the SDE
  \BGE d\ha w_j^u=\sqrt{\kappa u_j'}dB_j^u+\Big[\kappa\bb \frac{W_{j,2}^u}{W_{j,1}^u}+\cot_2(W_j^u-W_{3-j}^u)W_{j,1}^u+\sum_{l=1}^2 \frac{\kappa-4}{2} \cot_2(W_j^u-V_l^u) W_{j,1}^u\Big]u_j'dt,\label{SDE-wu}\EDE
  \label{P*BM}
where $B_j^u(t)$, $j=1,2$, are two independent Brownian motions.
\end{Lemma}
\begin{proof}
This lemma is similar to \cite[Lemma 4.3]{Two-Green-interior}, and the proof uses a Girsanov-type argument. Here is a sketch. We first define $\til w_j^u$, $j=1,2$, such that
\begin{align}
\til w_j^u(t)=\ha w_j^u(t)-\int_0^t &\Big[\kappa\bb \frac{W_{j,2}^u(s)}{W_{j,1}^u(s)}+\cot_2(W_j^u(s)-W_{3-j}^u(s))W_{j,1}^u(s) \nonumber \\
  & +\sum_{l=1}^2 \frac{\kappa-4}{2} \cot_2(W_j^u(s)-V_l^u(s)) W_{j,1}^u(s)\Big]u_j'(s)ds.\label{til-w-j}
\end{align}
By L\'evy's characterization of Brownian motion, it suffices to show that $\til w_1^u$ and $\til w_2^u$ are local martingals under $\PP^c_*$, $d\langle \til w_j^u\rangle= d u_j$, $j=1,2$, and $\langle \til w_1^u,\til w_2^u\rangle\equiv 0$.

Using (\ref{dMu},\ref{til-w-j}) and It\^o's formula, it is straightforward to calculate that $\til w_j^uM_*^u$, $j=1,2$, are $\F^u$-local martingales under $\PP^{\ii}_B$. Fix $j\in\{1,2\}$. Now we prove that $\til w_j^u$ is a local martingale under $\PP^c_*$. It suffices to prove that, for any $\ulin\xi\in\Xi$, $\til w_j^u(\cdot\wedge \tau^u_{\ulin \xi})$ is an $\F^u$-martingale under $\PP^c_*$. Fix $\ulin\xi\in\Xi$. Using an argument similar to the proof of Lemma \ref{uniform}, we can show that $|\til w_j^u|$ is bounded on $[0,\tau^u_{\ulin\xi}]$ by a constant depending only on $\kappa,\ulin\xi,w_1,w_2,v_1,v_2$. So we see that $\til w_j^u(\cdot\wedge \tau^u_{\ulin\xi})M_*^u(\cdot\wedge \tau^u_{\ulin\xi})$, $j=1,2$, are bounded $\F^u$-martingales under $\PP^{\ii}_B$.

To prove that $\til w_j^u(\cdot\wedge \tau^u_{\ulin \xi})$ is a martingale under $\PP^c_*$, we fix $t>s\ge 0$ and $A\in\F^u_s$, and need to prove that
\BGE \EE^c_*[{\bf 1}_A \til w_j^u(t\wedge \tau^u_{\ulin \xi})]= \EE^c_*[{\bf 1}_A \til w_j^u(s\wedge \tau^u_{\ulin \xi})].
\label{A}
\EDE
Let $A_1=A\cap \{\tau^u_{\ulin \xi}\le s\}$ and $A_2=A\cap \{\tau^u_{\ulin \xi}> s\}$. Then (\ref{A}) would follow from
\BGE \EE^c_*[{\bf 1}_{A_1} \til w_j^u(t\wedge \tau^u_{\ulin \xi})]= \EE^c_*[{\bf 1}_{A_1} \til w_j^u(s\wedge \tau^u_{\ulin \xi})];
\label{A1}
\EDE
\BGE \EE^c_*[{\bf 1}_{A_2} \til w_j^u(t\wedge \tau^u_{\ulin \xi})]= \EE^c_*[{\bf 1}_{A_2} \til w_j^u(s\wedge \tau^u_{\ulin \xi})].
\label{A2}
\EDE
Equality (\ref{A1}) is trivially true because $\til w_j^u(t\wedge \tau^u_{\ulin \xi})=\til w^j_u(\tau^u_{\ulin\xi})=\til w_j^u(s\wedge \tau^u_{\ulin \xi})$ on $A_1$. 


Since $\til w_j^u(\cdot\wedge \tau^u_{\ulin\xi})M_*^u(\cdot\wedge \tau^u_{\ulin\xi})$ is an $\F^u$-martingale under $\PP^{\ii}_B$ and $A_2\in \F^u_s$, we get
$$\EE^{\ii}_B[{\bf 1}_{A_2} \til w_j^u(t\wedge \tau^u_{\ulin\xi})M_*^u(t\wedge \tau^u_{\ulin\xi})]
=\EE^{\ii}_B[{\bf 1}_{A_2} \til w_j^u(s\wedge \tau^u_{\ulin\xi})M_*^u(s\wedge \tau^u_{\ulin\xi})].$$
Then (\ref{A2}) would follow from the above equality and the equalities
\BGE \EE^{\ii}_B[{\bf 1}_{A_2} \til w_j^u(r\wedge \tau^u_{\ulin\xi})M_*^u(r\wedge \tau^u_{\ulin\xi})] = M_*(\ulin 0) \EE^c_*[{\bf 1}_{A_2} \til w_j^u(r\wedge \tau^u_{\ulin \xi})],\quad r\in\{s,t\} .\label{rst}\EDE

We now prove (\ref{rst}) in the case $r=t$. The case $r=s$ is similar.
By Proposition \ref{prop-P*}, $\PP^c_*|\F_{\tau_{\ulin\xi}}$ is absolutely continuous w.r.t.\  $\PP^{\ii}_B|\F_{\tau_{\ulin\xi}}$, and the RN derivative is $M_*(\tau_{\ulin\xi})/M_*(\ulin 0)$.
Since $M_*(\cdot\wedge \tau_{\ulin\xi})$ is a martingale, for any two-dimensional stopping time $\ulin T\le \tau_{\ulin\xi}$, $\PP^c_*|\F_{\ulin T}$ is absolutely continuous w.r.t.\  $\PP^{\ii}_B|\F_{\ulin T}$, and the RN derivative is $M_*(\ulin T)/M_*(\ulin 0)$. Let $\ulin T=\ulin u(t\wedge \tau_{\ulin \xi})$. Then $\ulin T$ is a stopping time $\le \tau_{\ulin\xi}$. Since $\{\tau^u_{\ulin \xi}> s\}\subset \{\ulin u(s)\le \ulin u(s\wedge \tau^u_{\ulin\xi})\}$ and $A\in \F^u_s=\F_{\ulin u(s)}$, we get $A_2\in \F_{\ulin u(s\wedge \tau^u_{\ulin\xi})}\subset \F_{\ulin u(t\wedge \tau_{\ulin \xi})}=\F_{\ulin T}$. Thus, ${\bf 1}_{A_2} \til w_j^u(t\wedge \tau^u_{\ulin \xi})$ is $\F_{\ulin T}$-measurable. So we get (\ref{rst}) in the case $r=t$ using the RN derivative between   $\PP^c_*|\F_{\ulin T}$ and $\PP^{\ii}_B|\F_{\ulin T}$.

Now we have proved that $\til w_j^u$, $j=1,2$, are $\F^u$-martingales under $\PP^c_*$. 
Since (\ref{quadratic-w}) holds for $\ha w^u_j$ under $\PP^{\ii}_B$, it also holds for $\til w^u_j$ under $\PP^{\ii}_B$. For any $\xi\in\Xi$, using the local absolute continuity between $\PP^c_*|\F_{\tau_{\ulin\xi}}$ and $\PP^{\ii}_B|\F_{\tau_{\ulin\xi}}$, we can then conclude that (\ref{quadratic-w}) also holds for $\til w^u_j$ under $\PP^{c}_*$ up to $\tau^u_{\ulin\xi}$. Since this holds for any $\xi\in\Xi$, we get $d\langle \til w_j^u\rangle= d u_j$, $j=1,2$, and $\langle \til w_1^u,\til w_2^u\rangle\equiv 0$ throughout. The proof is then complete.
\end{proof}

Now we suppose $(\ha w_1,\ha w_2)$ follows the law $\PP^c_*$, and let $B^u_1$ and $B^u_2$ be the independent Brownian motions from Lemma \ref{P*BM}. Combining (\ref{uj'},\ref{dWju},\ref{SDE-wu}), we get, for $j=1,2$,
$$dW_j^u=W_{j,1}^u \sqrt{\kappa u_j'} dB_j^u+ \cot_2(W_j^u-W_{3-j}^u)dt + \sum_{s=1}^2 \frac{\kappa -4}{2} \cot_2(W_j^u-V_s^u) (W_{j,1}^u)^2u_j' dt .$$
Using (\ref{dWkVs}) we get, for $j=1,2$,
$$dV_j^u=-\cot_2(W_j^u-V_j^u)(W_{j,1}^u)^2 u_j'dt-\cot_2(W_{3-j}^u-V_j^u)(W_{3-j,1}^u)^2 u_{3-j}'dt.$$
Since $Z_j^u=W_j^u-V_j^u$, $W_{3-j}^u-V_j^u=W_{3-j}^u-V_{3-j}^u\pm \pi=Z_{3-j}^u\pm \pi$, $W_j^u-W_{3-j}^u=Z_j^u-Z_{3-j}^u \mp \pi$, and $W_j^u-V_{3-j}^u=Z_j^u\mp\pi$,
combining the above two displayed formulas with (\ref{uj'}), we get, for $j=1,2$,
\BGE dZ_j^u= \sqrt{\frac{\kappa \sin(Z_j^u)}{\sin(Z_1^u)+\sin(Z_2^u)}}dB^u_j +\frac{(\kappa-2)\cos (Z_j^u)}{\sin(Z_1^u)+\sin(Z_2^u)} \,dt .\label{dZju}\EDE

\begin{Remark}
  If the two force values $\kappa-4$ and $\kappa-4$ at $e^{iv_1}$ and $e^{iv_2}$ are respectively replaced by $\rho_1,\rho_2\in\R$, then $Z_j^u$, $j=1,2$, satisfy the SDE:
  $$dZ_j^u= \sqrt{\frac{\kappa \sin(Z_j^u)}{\sin(Z_1^u)+\sin(Z_2^u)}}dB^u_j +\frac{(2+(\rho_1+\rho_2)/2)\cos (Z_j^u)}{\sin(Z_1^u)+\sin(Z_2^u)} \,dt +\frac{(\rho_j-\rho_{3-j})/2}{\sin(Z_1^u)+\sin(Z_2^u)} \,dt.$$
\end{Remark}

\subsection{Transition density}
We are going to derive the transition density of $(Z^u_1,Z^u_2)$ under $\PP^c_*$ or $\PP^c_c$ following the approach of \cite[Section 5]{Two-Green-interior}. The idea, originated in \cite[Appendix B]{tip}, is to solve a PDE eigenvalue problem using orthogonal {polynomials}. First suppose that the underlying probability measure is $\PP^c_*$. Recall that $(Z^u_1,Z^u_2)$ satisfy (\ref{dZju}), where $B^u_1$ and $B^u_2$ are independent Brownian motions  by Lemma \ref{P*BM}.
Define $B_+^u$ and $B_-^u$ by
$$B^u_\pm(t)=\int_0^t  \sqrt{\frac{  \sin(Z_1^u(s))}{\sin(Z_1^u(s))+\sin(Z_2^u(s))}}dB^u_1(s)\pm \int_0^t  \sqrt{\frac{  \sin(Z_2^u(s))}{\sin(Z_1^u(s))+\sin(Z_2^u(s))}}dB^u_1(s).$$
Then $B_+^u$ and $B_-^u$ are standard linear Brownian motions with quadratic covariation
\BGE d\langle B_+^u,B_-^u\rangle_t=\cot_2(Z_1^u+Z_2^u)\tan_2(Z_1^u-Z_2^u)dt.\label{covB+-}\EDE
Define $Z_\pm^u=(Z_1^u\pm Z_2^u)/2$. Then $Z_+^u\in(0,\pi)$, $Z_-^u\in(-\pi/2,\pi/2)$, and they satisfy the SDEs:
$$dZ_+^u=\frac{\sqrt\kappa}2 dB_+^u+\frac{\kappa-2}2\cot(Z^u_+)dt;$$
$$dZ_-^u=\frac{\sqrt\kappa}2 dB_-^u-\frac{\kappa-2}2\tan(Z^u_-)dt.$$
Let $X=\cos(Z^u_+)$ and $Y=\sin(Z^u_-)$. Then $X,Y\in(-1,1)$, and satisfy the SDEs
$$ dX=-\frac{\sqrt\kappa}2 \sqrt{1-X^2}dB^u_+-\Big(\frac{\kappa-2}2+\frac \kappa 8\Big)Xdt; $$ 
$$ dY=+\frac{\sqrt\kappa}2 \sqrt{1-Y^2}dB^u_--\Big(\frac{\kappa-2}2+\frac \kappa 8\Big)Ydt.$$ 
From (\ref{covB+-}) we have
$$ d\langle X,Y\rangle_t=-\frac \kappa 4 XYdt.$$ 
We have $(X,Y)\in\D$ because $X^2+Y^2=1-\sin(Z_1^u)\sin(Z_2^u)<1$.

Define a second order differential operator $\cal L$ by
$${\cal L}:=\frac\kappa 8(1-x^2)\pa_x^2+\frac\kappa 8(1-y^2)\pa_y^2-\frac \kappa 4 xy \pa_x\pa_y-(\frac{\kappa-2}2+\frac \kappa 8 )x\pa_x-(\frac{\kappa-2}2+\frac \kappa 8 )y\pa_y.$$
Using orthogonal polynomials, we may find eigenvectors and eigenvalues  of $\cal L$. Define
$$ \lambda_s=-\frac \kappa 8 s(s+4-\frac 8\kappa  ),\quad s\in\Z_+=\{0,1,2,3,\dots\}.$$ 
A straightforward calculation shows that, for any $n,m\in\Z_+$, ${\cal L}(x^ny^m)$ equals $\lambda_{n+m} x^ny^m$ plus a polynomial in $x,y$ of degree less than $n+m$. Hence, for each $n,m\in\Z_+$, there is a polynomial $P_{(n,m)}(x,y)$ of degree $n+m$, which equals $x^ny^m$ plus a polynomial in $x,y$ of degree less than $n+m$, such that ${\cal L} P_{(n,m)}=\lambda_{n+m} P_{(n,m)}$.

Let $\Psi(x,y)=(1-x^2-y^2)^{1-\frac4\kappa}$, and define  $\langle f,g\rangle_\Psi:=\int\!\!\int_{\D} f(x,y)g(x,y)\Psi(x,y)dxdy$. Since $\Psi\equiv 0$ on $\TT$, using integration by parts, we can show that for smooth functions $f$ and $g$ on $\lin{\D}$, $\langle \L f,g\rangle_\Psi=\langle f,\L g\rangle_\Psi$. In fact, if we let $a_{xx}=\frac\kappa 8(1-x^2)$, $a_{yy}=\frac \kappa 8(1-y^2)$, $a_{xy}=a_{yx}=-\frac \kappa 8xy$, $b_x=-(\frac{\kappa-2}2+\frac \kappa 8 )x$, and $b_y=-(\frac{\kappa-2}2+\frac \kappa 8 )y$, then both $\langle \L f,g\rangle_\Psi$ and $\langle f,\L g\rangle_\Psi$ equal
$$-\int\!\!\int_{\D} [(\pa_x f) a_{xx}\Psi (\pa_x g)+(\pa_x f)a_{xy}\Psi(\pa_y g)+(\pa_y f) a_{yx}\Psi (\pa_x g) +(\pa_y f) a_{yy}\Psi (\pa_y g) ]dxdy.$$
Here we use $\pa_x(a_{xx}\Psi)+\pa_y(a_{xy}\Psi)=b_x\Psi$ and $\pa_x(a_{yx}\Psi)+\pa_y(a_{yy}\Psi)=b_y\Psi$.

Thus, the eigenvectors of $\cal L$ with different eigenvalues are orthogonal w.r.t.\ $\langle \cdot\rangle_\Psi$, and we may use $P_{(n,m)}$, $n,m\in\Z_+$, to construct an $\langle \cdot\rangle_\Psi$-orthonormal family of polynomials $v_{(n,s)}$, $n\in\Z_+$, $s\in\{0,1,\dots,n\}$, such that $v_{(n,s)}$ is of degree $n$ and ${\cal L} v_{(n,s)}=\lambda_n v_{(n,s)}$. From \cite[Section 1.2.2]{orthogonal-2}, we may choose $v_{(n,s)}$ such that for each $n\ge 0$, $v_{(n,0)},v_{(n,1)},\dots,v_{(n,n)}$ are given by
$$v_{n,j,1}=h_{n,j,1}P_j^{(1-\frac4\kappa,n-2j)}(2r^2-1)r^{n-2j}\cos((n-2j)\theta),\quad 0\le j\le \lfloor n/2\rfloor,$$
$$v_{n,j,2}=h_{n,j,2}P_j^{(1-\frac4\kappa,n-2j)}(2r^2-1)r^{n-2j}\sin((n-2j)\theta),\quad 0\le j\le \lfloor (n-1)/2\rfloor,$$
where $P_j^{(1-\frac4\kappa,n-2j)}$ are Jacobi polynomials of index $(1-\frac4\kappa,n-2j)$, $(r,\theta)$ is the polar coordinate of $(x,y)$: $x=r\cos\theta$ and $y=r\sin\theta$, and $h_{n,j,i}>0$ are normalization constants. Using the polar integration and Formula \cite[Table 18.3.1]{NIST:DLMF}:
$$ \int_{-1}^1 P_j^{(\alpha,\beta)}(x)^2(1-x)^\alpha(1+x)^\beta dx=\frac{2^{\alpha+\beta+1}\Gamma(j+\alpha+1)\Gamma(j+\beta+1)} {j!(2j+\alpha+\beta+1)\Gamma(j+\alpha+\beta+1)}$$ 
with $\alpha=1-\frac 4\kappa$ and $\beta=n-2j$, we compute
$$ h_{n,j}:=h_{n,j,1}=h_{n,j,2}=\sqrt{\frac {1+{\bf 1}_{n\ne 2j}}\pi\cdot\frac{j!(n+2-\frac 4\kappa)\Gamma(n-j+2-\frac 4\kappa)}{\Gamma(j+2-\frac 4\kappa)\Gamma(n-j+1)}}.$$ 
Using the supremum  norm (over $[-1,1]$) of $P_j^{(\alpha,\beta)}$ (\cite[18.14.1,18.14.2]{NIST:DLMF}):
$$ \|P_j^{(\alpha,\beta)}\|_\infty=\frac{\Gamma(\max\{\alpha,\beta\}+j+1)} {j!\Gamma(\max\{\alpha,\beta\}+1)},\quad\mbox{if } \max\{\alpha,\beta\}\ge -1/2 \mbox{ and }\min\{\alpha,\beta\}> -1,$$ 
we get
\BGE \|v_{n,j,1}\|_\infty =\|v_{n,j,2}\|_\infty=h_{n,j}\max\Big\{\frac{\Gamma(2-\frac 4\kappa+j)}{j!\Gamma(2-\frac 4\kappa)},\frac{\Gamma(n-j+1)}{j!\Gamma(n-2j+1)}\Big\}.\label{v}\EDE

For $t>0$, $(x,y),(x^*,y^*)\in\D$, we define
\BGE p_t((x,y),(x^*,y^*))=\sum_{n=0}^\infty\sum_{s=0}^n \Psi(x^*,y^*)v_{(n,s)}(x,y)v_{(n,s)}(x^*,y^*)e^{\lambda_nt}, \label{pt}\EDE
and $p_\infty(x^*,y^*)=\frac 1\pi(2-\frac 4\kappa) \Psi(x^*,y^*)$, which is the  term in the summation  for $n=s=0$. 
The following propositions are similar to Lemma 5.1, Lemma 5.2 and Corollary 5.3 of \cite{Two-Green-interior}. The proofs use the estimate (\ref{v}) and the orthogonality between $v_{n,s}$ w.r.t.\ $\langle \cdot\rangle_\Psi$. The exponent $1-\frac 58\kappa$ in Lemma \ref{asym-lemma} is the $\lambda_1$ here.

\begin{Lemma}
  For any $t_0>0$, the series in (\ref{pt}) converges uniformly on $[t_0,\infty)\times \lin\D\times \lin\D$, and there is $C_{t_0}\in(0,\infty)$ depending only on $\kappa$ and $t_0$ such that
$$ |p_t((x,y),(x^*,y^*))-p_\infty(x^*,y^*)|\le  C_{t_0} e^{(1-\frac58 \kappa)t}p_\infty(x^*,y^*),\quad t\ge t_0,\quad (x,y),(x^*,y^*)\in\lin\D.$$ 
Moreover, for any $t>0$ and $(x^*,y^*)\in \lin\D$,
$$ p_\infty(x^*,y^*)=\int\!\!\int_{\D} p_\infty(x,y)p_t((x,y),(x^*,y^*))dxdy.$$ 
\label{asym-lemma}
\end{Lemma}

\begin{Lemma}
 $(X,Y)$ (under $\PP^c_*$) has transition density $p_t $ and invariant density $p_\infty $.
\end{Lemma}

\begin{Corollary}
 $(Z^u_1,Z^u_2)$ under $\PP^c_*$ has   transition density
$$ p^Z_t(\ulin z,\ulin z^*) :=p_t((\cos_2(z_1+z_2),\sin_2(z_1-z_2)),(\cos_2(z_1^*+z_2^*),\sin_2(z_1^*-z_2^*)))\frac{\sin z_1^* +\sin z_2^*}{4}$$
and  invariant density
$p^Z_\infty ( \ulin z^*):= p_\infty( \cos_2(z_1^*+z_2^*),\sin_2(z_1^*-z_2^*))\frac{\sin z_1^* +\sin z_2^*}{4}$.
\label{transition-P*}
\end{Corollary}

We may then use Lemma \ref{RN-Lemma} to derive the transition density for $(Z^u_1,Z^u_2)$ under the law $\PP^c_c$. Let $t\in\R_+$. Since $\ulin u(t)$ is an $\F$-stopping time, and $\ulin  u(t)\in\cal D$ iff $T^u>t$, by Lemma \ref{RN-Lemma},
\BGE \frac{d\PP^c_c|\F^u_t\cap \{T^u>t\}}{d\PP^c_*|\F^u_t\cap \{T^u>t\}}=\frac{M_{*\to c}^u(t)}{M_{*\to c}^u(0)}.\label{d/d}\EDE
Let $\til G^u(z_1,z_2):=\til G(\pi+z_1,\pi,z_2,0)$. By (\ref{M*c}), we get $M_{*\to c}^u(t)= e^{-\alpha_0 t} \til G^u(Z_1^u(t),Z_2^u(t))^{-1}$. 
Combining this with (\ref{d/d}) and Corollary \ref{transition-P*}, we get the transition density for $(Z^u_1,Z^u_2)$ under the law $\PP^c_c$ in the lemma below, which resembles  \cite[Lemma 5.4]{Two-Green-interior}.

\begin{Lemma}
	Under $\PP^c_c$, $\ulin Z^u(t):=(Z^u_1(t),Z^u_2(t))$, $0\le t<T^u$, has transition density
$$\til p^Z_t(\ulin z,\ulin z^*):=  e^{-\alpha_0t} p^Z_t(\ulin z,\ulin z^*) \frac{\til G^u(\ulin z)}{\til G^u(\ulin z^*)}.$$
Here  the meaning of the transition density is: if $\ulin Z^u$ starts with $\ulin z\in(0,\pi)^2$, then for any $t>0$ and any bounded measurable function $f$ on $(0,\pi)^2$,
$$\EE^c_c[{\bf 1}_{\{T^u>t\}} f(\ulin Z^u(t))]=\int_{(0,\pi)^2} \til p^Z_t(\ulin z,\ulin z^*) f(\ulin z^*) d\ulin z^*.$$
\label{Lemma-til-p}
\end{Lemma}

We compute that $\frac{p_\infty^z(z_1,z_2)}{\til G^u(z_1,z_2)}=C\cos_2(z_1-z_2)^{2-\frac 8\kappa}\sin_2(z_1+z_2)$ for some constant $C\in(0,\infty)$ depending only on $\kappa$. Since $2-\frac 8\kappa>0$, we may define
\BGE {\cal Z}=\int_{(0,\pi)^2} \frac{p^Z_\infty(\ulin z)} {\til G^u(\ulin z)} d\ulin z \in(0,\infty),\label{calZ}\EDE
\BGE \til p^Z_\infty(\ulin z)=\frac 1{{\cal Z}}\frac{p^Z_\infty(\ulin z)} {\til G^u(\ulin z)},\quad \ulin z\in(0,\pi)^2.\label{tilpZinfty}\EDE
The following lemma resembles \cite[Lemma 5.5]{Two-Green-interior}.

\begin{Lemma}
\begin{enumerate}
  \item [(i)] For any $t>0$ and $\ulin z^*\in(0,\pi)^2$, \BGE \int_{(0,\pi)^2} \til p^Z_\infty(\ulin z)\til p^Z_t(\ulin z,\ulin z^*)d\ulin z=\til p^Z_\infty(\ulin z^*)e^{-\alpha_0t}.
	\label{invariant-psudo}\EDE
\item [(ii)] If $\ulin Z^u$ starts from $\ulin z\in(0,\pi)^2$, then as $t\to \infty$,
	$$\PP_c^c[T^u>t]={\cal Z} \til G^u(\ulin z) e^{-\alpha_0t}(1+O(e^{(1-\frac58 \kappa)t})),$$
	$$ \til p^Z_t(\ulin z,\ulin z^*)=\PP_c^c[T^u>t]\til p^Z_\infty(\ulin z^*)(1+O(e^{(1-\frac58 \kappa)t})),$$
where the implicit constant depends only on $\kappa$. The two formulas
 together imply that
\BGE \til p^Z_t(\ulin z,\ulin z^*)={\cal Z} \til G^u(\ulin z) \til p^Z_\infty(\ulin z^*)e^{-\alpha_0t} (1+O(e^{(1-\frac58 \kappa)t})),\quad \ulin z,\ulin z^*\in (0,\pi)^2. \label{tilptT>t}\EDE
\end{enumerate}
\end{Lemma}

\section{Proof of the Main Theorem}\label{section-cutpoint}
We now start the proof of Theorem \ref{Main-thm}. By conformal invariance, Koebe's distortion theorem and reflection symmetry, we may assume that $D=\D$, $z_0=0$, and $a_j=e^{iw_j}$ and $b_j=e^{iv_j}$, $j=1,2$, where  $w_1>v_1>w_2>v_2>w_1-2\pi$.  Then we have ${G_{D;a_1,b_1,a_2,b_2}(0)}=\til G(w_1,v_1,w_2,v_2)$.
Let $\gamma_1$ be as in the theorem, and let $\gamma_2$ be a time-reversal of $\gamma_1$. By reversibility, $\gamma_2$ is a chordal SLE$_\kappa$ in $\D$ from $e^{iw_2}$ to $e^{iw_1}$. 
We may assume that for $j=1,2$, $\gamma_j$ is parametrized by some capacity viewed from $a_{3-j}$. This means that there is a conformal map $f_j:\D\conf \HH$, which sends $a_{3-j}$ to $\infty$ such that $\hcap_2 f_j (\gamma_j[0,t])=t$ for $t\ge 0$.
For $r\in (0,1)$, let  $E(r)$ denote the event that  $\gamma\cup A_1\cup A_2$ has a cut point  that lies in $\{|z|<r\}$, which also depends on $w_1,v_1,w_2,v_2$, and let $P(w_1,v_2,w_2,v_2;r)=\PP[E(r)]$. We need to show that there is a constant $C_0>0$ depending only on $\kappa$ such that
\BGE P(w_1,v_1,w_2,v_2;r)=C_0  \til G^u(w_1,v_1,w_2,v_2)  r^{\alpha_0}(1+O(r^{\beta_0})),\quad \mbox{as }r\to 0.\label{final}\EDE

For $j=1,2$, let $\ha S_j$ be the first time that $\gamma_j$ separates $0$ from any of $e^{iw_{3-j}},e^{iv_1},e^{iv_2}$ in $\D$. Let $v_j(t)=\dcap(\gamma_j[0,t])$, $0\le t<\ha S_j$, $\ha T_j=v_j(\ha S_j)$,   $u_j=v_j^{-1}:[0,\ha T_j)\to [0,\ha S_j)$, and $\eta_j=\gamma_j\circ u_j$. By \cite{SW}, $\eta_j$ is a radial SLE$_\kappa(\kappa-6,0,0)$ curve in $\D$ started from $e^{iw_j}$ with force points $e^{iw_{3-j}},e^{iv_1},e^{iv_2}$. Let $\ha w_j$ be a radial Loewner driving function of $\eta_j$ with $\ha w_j(0)=w_j$. We will use the notation in Section \ref{section-det} for the radial Loewner curves $\eta_1,\eta_2$ in this section.


For $j=1,2$, let ${\cal G}^j$ and $\F^j$ be respectively the   filtrations generated by $\gamma_j$ and $\eta_j$. Then $\ha S_j$ is an ${\cal G}^j$-stopping time. Let $\lin\F^j$ and $\lin {\cal G}^j$ be respectively the right-continuous augmentation of $\F^j$ and ${\cal G}^j$. We extend $u_j$ to $[0,\infty]$ such that $u_j\equiv \infty$ on $[\ha T_j,\infty]$. The following is a well-known proposition about time-change.

\begin{Proposition}
  For $j\in\{1,2\}$, 
  if $\tau$ is an $\lin\F^j$-stopping time, then $u_j(\tau)$ is an $\lin{\cal G}^j$-stopping time,  and $\lin\F^j_{\tau} \subset \lin{\cal G}^j_{u_j(\tau)}$. \label{tau-sigma}
\end{Proposition}

Fix $j\ne k\in\{1,2\}$. Let $\tau_k$ be an $\F^k$-stopping time less than $\ha T_k$ and let $\sigma_k=u_k(\tau_k)$, which by Proposition \ref{tau-sigma} is an $\lin {\cal G}^k_{\sigma_k}$-stopping time.
By DMP of chordal SLE$_\kappa$, conditionally on $\lin {\cal G}^k_{\sigma_k}$, the part of $\gamma_k$ after $\sigma_k$ is a chordal SLE$_\kappa$ in $\D\sem K_k({\tau_k})$ from $\gamma_k(\sigma_k)=\eta_k(\tau_k)$ to $e^{iw_j}$.
Since $\gamma_j$ is a time-reversal of $\gamma_k$, by reversibility of chordal SLE$_\kappa$, conditionally on $\lin {\cal G}^k_{\sigma_k}$, the part of $\gamma_j$ up to the time that it hits $\eta_k(\tau_k)$ is a chordal SLE$_\kappa$ in $\D\sem K_k({\tau_k})$ from $e^{i w_j}$ to $\eta_k(\tau_k)$. Recall that $\eta_j$ is a time-change of the part of $\gamma_j$ up to the first time that it separates $0$ from any of $e^{iw_k},e^{iv_1},e^{iv_2}$ in $\D$. So the part of $\eta_j$ up to the first time that it hits $\eta_k[0,\tau_k]$ is a time-change of the part of $\gamma_j$ up to the first time that it separates $0$ from any of $e^{iw_k},e^{iv_1},e^{iv_2}$ in $\D\sem K_k({\tau_k})$.

Since $g_k({\tau_k},\cdot)$ maps $\D\sem K_k({\tau_k})$ conformally onto $\D$, sends $\eta_k(\tau_k)$ to $e^{i\ha w_k(\tau_k)}$, and is measurable w.r.t.\ $\F^k_{\tau_k}\subset \lin{\cal G}^k_{\sigma_k}$, by conformal invariance of chordal SLE, we see that, conditionally on $\lin{\cal G}^k_{\sigma_k}$, the $g_k({\tau_k},\cdot)$-image of the part of $\eta_j$ up to the first time that it hits $\eta_k[0,\tau_k]$ is a time-change of a chordal SLE$_\kappa$ in $\D$ from $g_k({\tau_k},e^{iw_j})$ to $e^{i\ha w_k(\tau_k)}$ up to the first time that it separates $0$ from any of $e^{i\ha w_k(\tau_k)},g_k({\tau_k},e^{iv_1}),g_k({\tau_k},e^{iv_2})$ in $\D$, which by \cite{SW} is a time-change of a radial SLE$_\kappa(\kappa-6,0,0)$ curve in $\D$ started from $g_k({\tau_k},e^{iw_j})$ with force points $e^{i\ha w_k(\tau_k)},g_k({\tau_k},e^{iv_1}),g_k({\tau_k},e^{iv_2})$. Since $g^k_{\tau_k}$ and $e^{i\ha w_k(\tau_k)}$ are $\F^k_{\tau_k}$-measurable, the above statement holds with $\F^k_{\tau_k}$ in place of $\lin{\cal G}^k_{\sigma_k}$. So we obtain the commutation coupling in Proposition \ref{prop-commu-c}, and the law of $(\ha w_1,\ha w_2)$ is $\PP^c_c$.

Let $\F$ and ${\cal G}$ be respectively the separable $\R_+^2$-indexed filtration generated by $\F^1,\F^2$ and   ${\cal G}^1,{\cal G}^2$, and let $\lin\F$ and $\lin{\cal G}$ be their right-continuous augmentation. Let $\ulin{\ha T}=(\ha T_1,\ha T_2)$ and $\ulin {\ha S}=(\ha S_1,\ha S_2)$. Define $u_{\otimes} $  on $[0,\infty]^2$ such that if $(t_1,t_2)\in [0,\ha T_1)\times [0,\ha T_2)$, $u_\otimes (t_1,t_2)=(u_1(t_1),u_2(t_2))$, and otherwise $u_\otimes (t_1,t_2)=(\infty,\infty)$.

\begin{Lemma}
For any $\lin\F$-stopping time $\ulin T$, $ u_{\otimes}(\ulin T)$ is an $\lin{\cal G}$-stopping time and $\lin\F_{\ulin T} \subset \lin{\cal G}_{u_\otimes(\ulin T)}$. \label{tau-sigma-2}
\end{Lemma}
\begin{proof}
  Let $A\in\lin\F_{\ulin T}$ and $\ulin s=(s_1,s_2)\in\R_+^2$. We have
  \BGE A\cap\{u_\otimes(\ulin T)<\ulin s\}=\bigcup_{\ulin p\in \Q_+^2} (A\cap \{\ulin T<\ulin p\}\cap \{ u_\otimes(\ulin p)<\ulin s\}).\label{A-inclusion}\EDE
  By Proposition \ref{right-continuous-0}, $A\cap \{\ulin T<\ulin p\}\in \F_{\ulin p}$. For any $j=1,2$ and $B_j\in \F^j_{p_j}$, by Proposition \ref{tau-sigma},  $B_j\cap \{u_j(p_j)<s_j\}\in {\cal G}^j_{s_j}$.
  By a monotone class argument, we get $B\cap  \{ u_\otimes(\ulin p)<\ulin s\}\in {\cal G}_{\ulin s}$ for any $B\in \F_{\ulin p}$. So $A\cap \{\ulin T<\ulin p\}\cap \{ u_\otimes(\ulin p)<\ulin s\}\in {\cal G}_{\ulin s}$. By (\ref{A-inclusion}) we get $A\cap\{u_\otimes (\ulin T)<\ulin s\}\in {\cal G}_{\ulin s}$. Since this holds for all $A\in\lin\F_{\ulin T}$ and $\ulin s\in\R_+^2$, by Proposition \ref{right-continuous-0}, we see that $u_\otimes (\ulin T)$ is an $\lin{\cal G}$-stopping time and $\lin\F_{\ulin T}\subset \lin{\cal G}_{u_\otimes (\ulin T)}$.
\end{proof}

\begin{Lemma}
  Let $0<r<1/4$.  Let $\ulin T=(T_1,T_2)$ be an $\lin\F$-stopping time such that $T_j<-\log(4r)$, $j=1,2$, and when $\ulin T\in\cal D$, $\mA(\ulin T)<-\log(4r)$. On the event $\{\ulin T\in\cal D\}$, let $R_+>R_-\in (0,1)$ be such that $\frac{R_\pm}{(1\pm R_\pm)^2}=e^{\mA(\ulin T)}r$, which exist uniquely because $e^{\mA(\ulin T)}r<1/4$.
Then
  \BGE \PP[E(r)|\lin\F_{\ulin T}]\lesseqgtr {\bf 1}_{\{\ulin T\in\cal D\}}\cdot P (W_1(\ulin T) ,V_1(\ulin T) ,W_2(\ulin T) ,V_2(\ulin T) ; R_\pm ).\label{gtreqless}\EDE
 Here and below the symbol $\lesseqgtr $ means that when we choose $+$ and $-$ in the $\pm$ on the RHS, the inequality holds with $\le$ and $\ge$, respectively.
  \label{conditional prob}
\end{Lemma}
\begin{proof}
For $j=1,2$, let $\tau^j_r$ and $\til \tau^j_r$ be respectively the first time that $\eta_j$ and $\gamma_j$ visits $\{|z|\le r\}$. First, suppose the event $E(r)$ happens. Then $\gamma_1$ and $\gamma_2$ do reach $\{|z|\le r\}$, i.e., $\til\tau^j_r<\infty$, $j=1,2$. By duality of SLE (cf.\ \cite[{Theorem 1.4}]{MS1} and {\cite[the discussion before Theorem 1.5]{MW}}), the {left boundary and right boundary} of $\gamma_1$ are two {(simple) flow lines, whose intersection is the cut-set of $\eta_1$}. In fact, the two simple flow lines are SLE$_{\kappa'}(\kappa'-4,\frac{\kappa'}2-2)$ and SLE$_{\kappa'}(\frac{\kappa'}2-2,\kappa'-4)$ curves, where $\kappa'=16/\kappa$, starting from $\eta_1(0)$ with the force points being $\eta_1(0)^+$ and $\eta_1(0)^-$. We do not need the laws of these curves. Thus, for $j=1,2$, $\gamma_j[0,\til\tau^j_r]$ does not intersect $A_{3-j}$ or disconnect $0$ from $\pa\D$; and $\gamma_1[0,\til\tau^1_r]\cap \gamma_2[0,\til \tau^2_r]=\emptyset$. The former condition implies that $\til\tau^j_r<\ha S_j$, and so $\tau^j_r=v_j(\til\tau^j_r)<\ha T_j$; and the latter implies that $\eta_1[0,\tau^1_r]\cap \eta_2[0,\tau^2_r]=\emptyset$. So we get $(\tau^1_r,\tau^2_r)\in\cal D$.
Since $\D\sem K_j(\tau^j_r)\supset \{|z|<r\}$,  by Koebe's $1/4$ theorem, $\tau^j_r\ge -\log(4r)\ge T_j$, $j=1,2$.   Since $(\tau^1_r,\tau^2_r)\in\cal D$ and $\cal D$ is an HC region, we get $\ulin T\in\cal D$. Thus, $E(r)\subset \{\ulin T\in\cal D\}$. 

Now we assume that $\{\ulin T\in\cal D\}$ (instead of $E(r)$) happens. Let $\til\gamma_1$ denote the part of $\gamma_1$ between $\eta_1(T_1)$ and $\eta_2(T_2)$. Let $\til A_j:=A_j\cup \eta[0,T_j]$, $j=1,2$. Then $\gamma_1\cup A_1\cup A_2=\til\gamma_1\cup\til A_1\cup \til A_2$.
By Lemma \ref{tau-sigma-2}, $u_\otimes (\ulin T)$ is an $\lin{\cal G}$-stopping time. By Theorem \ref{DMP-thm}, conditionally on $\lin{\cal G}_{u_\otimes (\ulin T)}$ and the event $\{\ulin T\in{\cal D}\}$, $\til\gamma$ is a time-change of a chordal SLE$_\kappa$ in $\D\sem K(\ulin T)$ from $\eta_1(T_1)=\gamma_1(u_1(T_1))$ to $\eta_2(T_2)=\gamma_2(u_2(T_2))$. Since the event $\{\ulin T\in{\cal D}\}$ and the random elements  $\D\sem K(\ulin T)$ and $\eta_j(T_j)$, $j=1,2$, are all $\lin\F_{\ulin T}$-measurable, the above statement holds with $\lin\F_{\ulin T}$ in place of $\lin{\cal G}_{u_\otimes (\ulin T)}$. Since $g(\ulin T,\cdot)$ is $\lin\F_{\ulin T}$-measurable, and sends $\eta_1(T_1)$ and $\eta_2(T_2)$ respectively to $e^{iW_1(\ulin T)}$ and $e^{iW_2(\ulin T)}$, we conclude that, conditionally on $\lin\F_{\ulin T}$ and the event  $\{\ulin T\in{\cal D}\}$, $g(\ulin T,\til\gamma)$ is a time-change of a chordal SLE$_\kappa$ in $\D$ from $e^{iW_1(\ulin T) }$ to $e^{iW_2(\ulin T) }$.

 For $j=1,2$, let $A_j(\ulin T)$ be the boundary arc of $\D$ connecting $e^{iV_1(\ulin T)}$ and $e^{iV_2(\ulin T)}$ which contains $e^{iW_j(\ulin T)}$. Since $g(\ulin T,\cdot)$ sends $e^{iv_1}$ and $e^{iv_2}$ respectively to $e^{iV_1(\ulin T)}$ and $e^{iV_2(\ulin T)}$, that $\til\gamma_1\cup\til A_1\cup \til A_2$ contains a cut point in $\{|z|<r\}$ is then equivalent to that $g(\ulin T,\til \gamma_1)\cup A_1(\ulin T)\cup A_2(\ulin T)$ has a cut point contained in $g (\ulin T,\{|z|<r\})$. Recall that $g(\ulin T,\cdot) $ maps $\D\sem K(\ulin T)$ conformally onto $\D$, fixes $0$, and has derivative $e^{\mA(\ulin T)}$ at $0$.  By Koebe's distortion theorem {applied to $e^{\mA(\ulin T)} g(\ulin T,\cdot)^{-1}$,
$$\Big\{|z|<\frac {e^{-\mA(\ulin T)} R}{(1+R)^2}\Big\}\subset g(\ulin T,\cdot)^{-1}(\{|z|<R\})\subset \Big\{|z|<\frac {e^{-\mA(\ulin T)} R}{(1-R)^2}\Big\},\quad \forall R\in(0,1). $$
Taking $R=R_\pm$, we get $\{|z|<R_-\}\subset g(\ulin T,\{|z|<r\})\subset \{|z|<R_+\}$.} Then (\ref{gtreqless}) follows immediately.
\end{proof}

For $z_1,z_2\in(0,\pi)$, we write $P^u(z_1,z_r;r)$ for $P(\pi+z_1,\pi,z_2,0;r)$. By rotation symmetry, if $v_1-v_2=\pi$, then $P(w_1,v_1,w_2,v_2;r)=P^u(w_1-v_1,w_2-v_2;r)$.

\begin{Lemma}
  Let $\til p_t^Z(\ulin z,\ulin z^*)$, $t\in(0,\infty)$, $\ulin z,\ulin z^*\in(0,\pi)^2$, be the transition density given by Lemma \ref{Lemma-til-p}. Let $r\in(0,1/4)$ and $t_0\in(0,-\log(4r))$. Let $R_+,R_-\in(0,1)$ be such that $\frac{R_\pm}{(1\pm R_\pm)^2}=e^{t_0} r$. Then for any $\ulin z\in(0,\pi)^2$,
  \BGE P^u(\ulin z;r)\lesseqgtr \int_{(0,\pi)^2} P^u (\ulin z^*,R_\pm ) \til p^Z_{t_0}(\ulin z,\ulin z^*) d\ulin z^*.\label{int-pZ}\EDE \label{tilptP(r)}
\end{Lemma}
\begin{proof}
  Fix $\ulin z=(z_1,z_2)\in(0,\pi)^2$. Let $w_1=\pi+z_1$, $v_1=\pi$, $w_2=z_2$, $v_2=0$.  Then $\PP[E(r)]=P^u(\ulin z,r)$. Since $v_1-v_2=\pi$, we may define the curve $\ulin u=(u_1,u_2)$  as in Section \ref{section-u}. If $\ulin u(t_0)\in \cal D$, then $u_j(t_0)\le \mA(\ulin u(t_0))=t_0$, $j=1,2$. If $\ulin u(t_0)\not\in\cal D$, then $t_0\ge T^u$ and $\ulin u(t_0)=\ulin u(T^u)=\lim_{t\uparrow T^u}\ulin u(t)\le T^u\le t_0$ because for $0\le t<T^u$, we have $\ulin u(t)\in\cal D$, and so $u_j(t)\le t\le T^u$, $j=1,2$. Thus, in any case we have $u_j(t_0) \le t_0$, $j=1,2$.

  Apply Lemma \ref{conditional prob} to the $\F$-stopping time $\ulin u(t_0)$. Since $\mA(\ulin u(t_0))=t_0$ when  $\ulin u(t_0)\in\cal D$, the $R_\pm$ here agree with the $R_\pm$ in Lemma \ref{conditional prob}. From $V^u_1(t_0)-V^u_2(t_0)=\pi$, we then get
  $$\PP[E(r)|\F^u_{t_0}]\lesseqgtr {\bf 1}_{\ulin u(t_0)\in\cal D} P^u (Z^u_1(t_0),Z^u_2(t_0);R_\pm).$$
By integration we get
\BGE P^u(\ulin z;r)=\PP[E(r)]\lesseqgtr \EE[{\bf 1}_{\ulin u(t_0)\in\cal D} P^u (\ulin Z^u(t_0);R_\pm)].\label{int-pZ2}\EDE
  Since $(\ha w_1,\ha w_2)$ follows the law $\PP^c_c$, $\til p^Z_t(\ulin z,\ulin z^*)$ is the transition density of the process $(\ulin Z^u(t))$. So the density of $\ulin Z^u(t_0)$ is $\til p^Z_{t_0}(\ulin z,\cdot)$, and the RHS of (\ref{int-pZ2}) agrees with the RHS of (\ref{int-pZ}).
\end{proof}

Let $\til p^Z_\infty$ be given by (\ref{tilpZinfty}{)}. Define
$\lin P(r)=\int_{(0,\pi)^2} P^u(\ulin z,r)\til p^Z_\infty(\ulin z)d\ulin z$, $r\in(0,1)$.

\begin{Lemma}
We have  $\lin P(r)=\ha c r^{\alpha_0}(1+O(r))$ as  $r\downarrow 0$, where $\ha c\ge 0$ is   a constant   depending only on $\kappa$, and the implicit constants in $O(r)$ depend only on $\kappa$.
\label{P(r)}
\end{Lemma}
\begin{proof}
Let $r,t_0,R_\pm$ be as in Lemma \ref{tilptP(r)}. By integrating both sides of (\ref{int-pZ}) against $\til p^Z_\infty(\ulin z)$ and using (\ref{invariant-psudo}), we get $\lin P(r)\lesseqgtr e^{-\alpha_0 t_0} \lin P(R_\pm)$.
Let $Q(r)=r^{-\alpha_0}\lin P(r)$. Since $\frac{R_\pm}{(1\pm R_\pm)^2}=e^{t_0}r$, we get
\BGE   Q(r)\lesseqgtr (1\pm R_\pm )^{2\alpha_0} Q(R_\pm ) .\label{t0r}\EDE
Let $R\in(0,\frac 16]$ and $r\in(0,\frac{R}{2}]$. Then $r<\frac{R}{(1+R)^2}<\frac{R}{(1-R)^2}<\frac 14$. So there are $t_+,t_-\in (0,-\log(4r))$ such that $\frac R{(1\pm R)^2}=e^{t_\pm} r$. By (\ref{t0r}) (which does not contain $t_0$),
\BGE  Q(r)\lesseqgtr (1\pm R)^{2\alpha_0} Q(R),\quad \mbox{if }0<r\le R/2<R\le1/6.\label{pre-limit}\EDE
Thus, for any $a\in(0,\frac 1{12}]$ and $r_1,r_2\in (0,a]$, we have  $(1+2a)^{-2\alpha_0} Q(r_1)\le (1-2a)^{-2\alpha_0} Q(r_2)$. This implies that $\lim_{r\downarrow 0} Q(r)$ converges. Let the limit be denoted by $\ha c$, which is nonnegative and depends only on $\kappa$. Letting $r\downarrow 0$ in (\ref{pre-limit}), we get
$ Q(R)\gtreqless \ha c (1\pm R)^{-2\alpha_0}$, if $R\le 1/6$. This immediately implies the conclusion.
\end{proof}

Recall the $\cal Z$ defined in (\ref{calZ}). Let $C_0=\ha c{\cal Z}$.   Then  $C_0$ is a nonnegative constant depending only on $\kappa$. We are now ready to prove that (\ref{final}) holds for such $C_0$.
First suppose $v_1-v_2=\pi$. Then $P(w_1,v_1,w_2,v_2;r)=P^u(z_1,z_2;r)$ and $\til G(w_1,v_1,w_2,v_2)=\til G^u(z_1,z_2)$, where $z_j=w_j-v_j$. Recall that $\beta_0=1-\frac 8{5\kappa}$. Let $R=r^{\beta_0}$. When $r$ is small enough (depending on $\beta_0$), we may choose $t_+,t_-\in (0,-\log(4r))$ such that $\frac{R}{(1\pm R)^2}=e^{t_\pm} r$. By (\ref{tilptT>t}) and Lemmas \ref{tilptP(r)} and \ref{P(r)},
  $$P^u(\ulin z;r)\lesseqgtr \int_{(0,\pi)^2} P^u (\ulin z^*,R ) \til p^Z_{t_\pm}(\ulin z,\ulin z^*) d\ulin z^*
 ={\cal Z} \til G^u(\ulin z) e^{-\alpha_0 t_\pm} \lin P(R) (1+O(e^{(1-\frac 58\kappa)t_\pm}))$$
  $$ =C_0 \til G^u(\ulin z) e^{-\alpha_0 t_\pm} R^{\alpha_0}(1+O(R)+O(e^{(1-\frac 58\kappa)t_\pm}))=C_0 \til G^u(\ulin z) r^{\alpha_0}(1+O(R)+O(e^{(1-\frac 58\kappa)t_\pm})).$$
Since $O(R)=O(e^{(1-\frac 58\kappa)t_\pm})=O(r^{\beta_0})$, the above formula implies that
  \BGE P^u(\ulin z;r)=C_0 \til G^u(\ulin z) r^{\alpha_0} (1+O(r^{\beta_0})),\quad \mbox{as }r\downarrow 0.\label{final-u}\EDE
 So we have finished the proof of (\ref{final}) in the case $v_1-v_2=\pi$.

Now suppose $\theta(\ulin 0)=v_1-v_2<\pi$.  By (\ref{pathetamA}), $\theta$ is increasing in $t_2$. Let $\tau_2$ be the first $t_2\in (0,\ha T_2)$ such that $\theta(0,\tau_2)=\pi$, if this time exists; otherwise let $\tau_2=\ha T_2$. From (\ref{pathetamA}) and that $\mA(0,t)=t$, we see that $\cos(\theta(0,t)/4)\le e^{-t/2} \cos(\theta(\ulin 0)/4)\le e^{-t/2}$, $0\le t<\ha T_2$. Then for $0\le t<\tau_2$, $t\le -2\log \cos(\theta(0,t)/4)\le \log(2)$, which implies that $\tau_2\le \log(2)$. Suppose $r\in(0,1/8)$. Then $\tau_2\le -\log(4r)$. Let $R_\pm\in(0,1)$ be such that $\frac{R_\pm}{(1\pm R_\pm)^2}=e^{\tau_2} r$. Then $R_\pm =e^{\tau_2} r(1+O(r))$ as $r\downarrow 0$. We now apply Lemma \ref{conditional prob} to $\ulin T=(0,\tau_2)$. Since $\theta(0,\tau_2)=\pi$ when $(0,\tau_2)\in\cal D$, using (\ref{final-u}), we get, as $r\downarrow 0$,
$$\PP[E(r)|\F_{(0,\tau_2)}]\lesseqgtr {\bf 1}_{\{(0,\tau_2)\in\cal D\}} P^u( \ulin Z(0,\tau_2);R_\pm)={\bf 1}_{\{(0,\tau_2)\in\cal D\}} P^u(\ulin Z(0,\tau_2);e^{\tau_2}r (1+O(r)))$$
$$=C_0 {\bf 1}_{\{(0,\tau_2)\in\cal D\}}  e^{\alpha_0\tau_2} \til G^u(\ulin Z(0,\tau_2)) r^{\alpha_0} (1+O(r^{1-\frac{8}{5\kappa}})), $$
which together with (\ref{M*c}) implies that
$$\PP[E(r)]=C_0 r^{\alpha_0} (1+O(r^{1-\frac{8}{5\kappa}})) \EE[{\bf 1}_{\{(0,\tau_2)\in\cal D\}}  e^{\alpha_0\mA(0,\tau_2)} \til G^u(\ulin Z(0,\tau_2))]$$
$$=C_0 r^{\alpha_0}  (1+O(r^{1-\frac{8}{5\kappa}}))\EE[{\bf 1}_{\{(0,\tau_2)\in\cal D\}}   M_{*\to c}(0,\tau_2)^{-1}] $$
Since $\EE=\EE^c_c$, by Lemmas \ref{D=R^2} and \ref{RN-Lemma},
$$\EE[{\bf 1}_{\{(0,\tau_2)\in\cal D\}}  M_{*\to c}(0,\tau_2)^{-1}] =M_{*\to c }(\ulin 0)^{-1} \PP^c_*[(0,\tau)\in{\cal D}]=M_{*\to c }(\ulin 0)^{-1}=\til G(w_1,v_1,w_2,v_2).$$
The last two displayed formulas together imply that (\ref{final}) holds in the case $v_1-v_2<\pi$.

The case that $v_1-v_2>\pi$ can be handled in a similar way. By (\ref{pathetamA}),  $\theta$ is decreasing in $t_1$. We may  apply Lemma \ref{conditional prob} to $\ulin T=(\tau_1,0)$, where $\tau_1$ is the first $t_1\in [0,\ha T_1)$ such that $\theta(t_1,0)=\pi$, if this time exists; and $\tau_1=\ha T_1$ otherwise. Thus, (\ref{final}) holds in all cases.

It remains to show that $C_0>0$. Suppose $C_0=0$. Then for any $z_0\in D$, we have $\PP[E_{z_0}(r)]=0$ when $r>0$ is small enough, which implies that  a.s.\ $\gamma_1\cup A_1\cup A_2$ does not have a cut point. However, since $\kappa\in(4,8)$, $\gamma_1$ does have a cut point, and when this happens, one can always find $b_1$ and $b_2$ in some deterministic countable dense subset of $\pa D\sem \{a_1,a_2\}$ such that $\gamma_1\cup A_1\cup A_2$ {contains a cut point, where for $j=1,2$, $A_j$ is the connected component of $\pa\D\sem \{b_1,b_2\}$ that contains $a_j$}. This implies that for some choice of $b_1,b_2$, the probability that  $\gamma_1\cup A_1\cup A_2$ has a cut point is positive.  The contradiction shows that $C_0>0$. The proof   is now complete.



\appendixpage
\begin{appendices}

\section{Domain Markov Property in Two Directions}\label{Section-DMP-bi}
Let $\kappa\in (0,8]$. In this appendix, we combine the reversibility of chordal SLE$_\kappa$ with the usual (one-directional) DMP  to derive a two-directional DMP.

Let $D$ be a simply connected domain with  locally connected boundary. Let $a_1,a_2$ be distinct prime ends of $D$. For $j=1,2$, let $\gamma_j$ be a chordal SLE$_\kappa$ curve in $D$ from $a_j$ to $a_{3-j}$, parametrized by the capacity viewed from $a_{3-j}$, such that $\gamma_1$ and $\gamma_2$ are time-reversal of each other. Such $\gamma_1,\gamma_2$ exist by reversibility of SLE$_\kappa$. Let $\phi_2$ be the decreasing auto-homeomorphism of $[0,\infty]$ such that $\gamma_1\circ \phi_2=\gamma_2$. For $j=1,2$ and $t\ge 0$, let $D^j_t$ be the connected component of $D\sem \gamma_j[0,t]$, which shares the prime end $a_{3-j}$ with $D$. We may view $\gamma_j(t)$ as a prime end of $D^j_t$. For $j=1,2$, let ${\cal D}_j=\{(t_1,t_2)\in\R_+^2:\gamma_j(t_j)\not\in \gamma_{3-j}[0,t_{3-j}]\}$ and ${\cal D}_{\cap}={\cal D}_1\cap {\cal D}_2$. When $t_1<\phi_2(t_2)$, let $D_{\ulin t}$ denote the connected component of $D\sem \bigcup_{j=1,2} \gamma_j[0,t_j]$ whose closure contains $\gamma_1[t_1,\phi_2(t_2)]$. If $\ulin t\in {\cal D}_\cap$, $D_{\ulin t}$ shares the prime end $\gamma_j(t_j)$ with $D^j_{t_j}$, $j=1,2$. For $j=1,2$, let ${\cal G}^j$ be the $\R_+$-indexed filtration generated by $\gamma_j$. Let ${\cal G}$ be the separable $\R_+^2$-indexed filtration generated by ${\cal G}^1$ and ${\cal G}^2$, and $\lin{\cal G}$ be the right-continuous   augmentation of $\cal G$.  We are going to prove the following theorem, which has been used in
Lemma \ref{conditional prob}.

\begin{Theorem}
If $\ulin T=(T_1,T_2)$ is an $\lin{\cal G}$-stopping time, then conditionally on $\lin{\cal G}_{\ulin T}$ and the event $\{\ulin T\in{\cal D}_{\cap}\}$, $\gamma_1|_{[T_1,\phi_2(T_2)]}$ is a time-change of a chordal SLE$_\kappa$ curve from the prime end $\gamma_1(T_1)$ to the prime end $\gamma_2(T_2)$ in $D_{\ulin T}$.
\label{DMP-thm}
\end{Theorem}

\begin{Remark}
A more specific statement of the theorem is the following. The event $\{\ulin T\in{\cal D}_{\cap}\}$ is $\lin {\cal G}_{\ulin T}$-measurable, and on this event, there are
\begin{itemize}
  \item an $\lin{\cal G}_{\ulin T}$-measurable conformal map $h_{\ulin T}$ from $\HH$ onto $D_{\ulin T}$, which has continuation to $\lin \HH^\#$, and sends $0$ and $\infty$ respectively to the prime ends $\gamma_1(T_1)$ and $\gamma_2(T_2)$,
\item a standard chordal SLE$_\kappa$ curve $\beta_1^{\ulin T} $ conditionally independent of $\lin{\cal G}_{\ulin T}$, and
    \item an increasing homeomorphism $u_1^{\ulin T}$ from $[0,\infty]$ onto $[T_1,\phi_2(T_2)]$,
\end{itemize}
such that   $\gamma_1\circ u_1^{\ulin T} =h_{\ulin T} \circ \beta_1^{\ulin T}$ a.s.\ on $\{\ulin T\in\cal D_\cap\}$.

The random curve $\beta_1^{\ulin T}$ is only defined on the event $\{\ulin T\in{\cal D}_{\cap}\}$. By saying that $\beta_1^{\ulin T} $ is a standard chordal SLE$_\kappa$ curve conditionally independent of $\lin{\cal G}_{\ulin T}$, we mean that  under the new probability measure $\PP[\cdot|\{\ulin T\in{\cal D}_{\cap}\}]$, where $\PP$ is the original underlying probability measure, $\beta_1^{\ulin T}$ has the law of a standard chordal SLE$_\kappa$ curve and is independent of $\lin{\cal G}_{\ulin T}$.
\end{Remark}

\begin{Remark}
  In the statement of the theorem, we condition on the $\sigma$-algebra $\lin{\cal G}_{\ulin T}$ and an event $\{\ulin T\in {\cal D}_{\cap}\}$, which is $\lin{\cal G}_{\ulin T}$-measurable. The bigger the event, the stronger the statement is. For the application of Theorem \ref{DMP-thm} in Lemma \ref{conditional prob}, we only need a weaker result, in which ${\cal D}_{\cap}$ is replaced by ${\cal D}_0=\{(t_1,t_2)\in\R_+^2: \gamma_1[0,t_1]\cap \gamma_2[0,t_2]=\emptyset\}$, which is a proper subset of ${\cal D}_{\cap}$ in the case $\kappa\in(4,8]$. 
  The strongest possible form of the theorem for $\kappa\in(4,8)$ is when ${\cal D}_{\cap}$ is replaced by the bigger set $\{(t_1,t_2)\in\R_+^2:t_1<\phi_2(t_2)\}$. In that case, the statement has to be modified since  $D_{\ulin T}$ may not share the prime end $\gamma_j(T_j)$ with $D^j_{T_j}$, and so we may not view $\gamma_j(T_j)$ as a prime end of $D_{\ulin T}$. 
  The proof of such a strong statement remains open to the author.
\end{Remark}

\begin{Remark}
  Theorem \ref{DMP-thm} is intuitively correct, but still requires a proof, which we could not find in the literature. We make a complete (nontrivial) proof of the theorem here because it is a solid step in the development of the multi-time-parameter framework. 

The essential property of the chordal SLE$_\kappa$ curve for $\kappa\in(0,8)$ used in the proof is the fact that such a curve commutes with its time reversal, which is also a chordal SLE$_\kappa$ curve. Thus, the argument can be easily extended to any commuting pair of SLE-type curves. Taking the $\eta_1$ and $\eta_2$ in Proposition \ref{prop-commu} for example, we have the following result. Let ${\cal D}$ be as defined in Section \ref{section-det}. Let $\F^j$ be the $\R_+$-indexed filtration generated by $\eta_j$, $j=1,2$, and let $\lin\F$ be the right-continuous augmentation of the separable $\R_+^2$-indexed filtration generated by $\F^1$ and $\F^2$. Then for any $\lin\F$-stopping time $\ulin T=(T_1,T_2)$, conditionally on $\lin\F_{\ulin T}$ and $\{\ulin T\in{\cal D}\}$, for $j=1,2$, the $g(\ulin T,\cdot)$-images of $\eta_j(T_j+\cdot)$ is a time-change of a radial SLE$_\kappa(2,\ulin\rho)$ curve in $\D$ started from $e^{iW_j(\ulin T)}$ aimed at $0$ with force points $e^{iW_{3-j}(\ulin T)},e^{iV_1(\ulin T)},\dots, e^{iV_m(\ulin T)}$. In addition, 
the two new curves $\ha\eta_1$ and $\ha\eta_2$ commute in the same way as in $\eta_1$ and $\eta_2$ do.  
A similar result has been used in \cite[Lemma 4.1]{Two-Green-boundary}.
\end{Remark}

We now give a sketch of the proof. Recall that the one-directional DMP follows from the strong Markov property of Brownian motion and the definition of chordal Loewner equation. The standard way to prove the strong Markov property of Brownian motion at an arbitrary stopping time $T$ is to approximate $T$ by a decreasing sequence of stopping times $(T_n)$ taking countably many values and apply the Markov property of Brownian motion at deterministic times to prove that the strong Markov property holds for each $T_n$.

  To prove Theorem \ref{DMP-thm}, we could also approximate $\ulin T$ using a decreasing sequence of stopping times $(\ulin T^n)$ such that each $\ulin T^n$ takes countably many values. Using the one-directional DMP and reversibility of chordal SLE$_\kappa$, it is easy to prove that the theorem holds for a deterministic  times $\ulin t=(t_1,t_2)$, which  means that, conditionally on $\lin{\cal G}_{\ulin t}$ and the event $\{\ulin t\in{\cal D}_{\cap}\}$, $\gamma_1|_{[t_1,\phi_2(t_2)]}$ is a time-change of a chordal SLE$_\kappa$ curve from $\gamma_1(t_1)$ to  $\gamma_2(t_2)$ in the domain $D_{\ulin t}$. Since each $\ulin T^n$ takes countably many values, the statement also holds for $\ulin T^n$.

Extending the results from $\ulin T^n$ to $\ulin T$ takes significant amount of work. We observe that, for each $\ulin t=(t_1,t_2)\in{\cal D}_{\cap}$, there exists at least one conformal map from $\HH$ onto $D_{\ulin t}$, which sends $0$ and $\infty$ respectively to the prime ends $\gamma_1(t_1)$ and $\gamma_2(t_2)$, but such conformal maps are not unique. We will prove Lemma \ref{f1t}, which says that we may choose the conformal maps $f^1_{\ulin t}$, $\ulin t\in {\cal D}_{\cap}$, such that $f^1_{\ulin t}(z)$ is jointly continuous in $\ulin t$ and $z$. The continuity fact implies that $f^1_{\ulin T}=\lim_{n\to\infty} f^1_{\ulin T^n}$ is $\lin{\cal G}_{\ulin T}$-measurable.

Fix $\ulin t\in{\cal D}_{\cap}$. By DMP and reversibility of chordal SLE$_\kappa$, there is a random Loewner curve $\eta_{\ulin t}$ in $\HH$, which is a time-change of a standard SLE$_\kappa$ such that $f^1_{\ulin t}\circ \eta_{\ulin t}=\eta_1|_{[t_1,\phi_2(t_2)]}$. Suppose $\eta_{\ulin t}$ is driven by $\ha w_{\ulin t}$ with speed $u_{\ulin t}$. Let $\ha v_{\ulin t}=\ha w_{\ulin t}\circ u_{\ulin t}^{-1}$. Then $(\ha v_{\ulin t})$ has the law of $(\sqrt\kappa B_t)$. Since $\ulin T^n$ takes countably many values, $(\ha v_{\ulin T^n}(t))$ also has the law of $(\sqrt\kappa B_t)$. We will prove Lemma \ref{version}, which will imply that $\ha w_{\ulin t}(t)$ and $u_{\ulin t}(t)$ both have versions that are jointly continuous in $\ulin t$ and $t$, and so the same is true for $\ha v_{\ulin t}(t)$, which then implies that $\ha v_{\ulin T}(t)=\lim_{n\to \infty} \ha v_{\ulin T^n}(t)$ has the law of $(\sqrt\kappa B_t)$. This fact shows that $\eta_{\ulin T}(t)$ is a time-change of a standard SLE$_\kappa$. Since $f^1_{\ulin t}\circ \eta_{\ulin t}=\eta_1|_{[t_1,\phi_2(t_2)]}$ and $f^1_{\ulin T}$ is $\lin{\cal G}_{\ulin T}$-measurable, we arrive at the conclusion.

The proofs of Lemmas \ref{version} and \ref{f1t} use the Carath\'eodory topology for the convergence of domains developed in \cite[Sections 5.1-5.2]{LERW}, whose definition and basic properties will be recalled in Definition \ref{Def-Cara} and Propositions \ref{domain-converge-inverse} and \ref{compact-hull}. We will further develop this topology to include one prime end and derive some of its properties  in Lemma \ref{compact3'}.

The above description of the proof gives the basic idea but is very sketchy and   not precise. Here are some examples. We will not go directly from $\ulin T^n$ to $\ulin T$. Instead, we will first approximate $\ulin T=(T_1,T_2)$ using $(T^n_1,T_2)$, where $T^n_1$ takes countably many values (but $(T^n_1,T_2)$ does not), and then approximate each $(T^n_1,T_2)$ using stopping times taking countably many values.  Lemma \ref{version} will appear before Lemma \ref{f1t} and be used in the proof of the latter lemma.  The precise statement about the law of $(\ha v_{\ulin T^n}(t))$ should be: conditionally on $\lin{\cal G}_{\ulin T^n}$ and the event $\{\ulin T^n\in {\cal D}_{\cap}\}$, $(\ha v_{\ulin T^n}(t))$ has the law of $(\sqrt\kappa B_t)$.

Two groups of intertwined factors contribute to the significant length of the proof of Theorem \ref{DMP-thm}. The factors on the Probability side is the treatment of two-time-variable filtration, stopping times and stochastic processes; the factors on the Complex Analysis side is the argument involving Loewner's equation, prime ends, Carath\'eodory topology and extremal length. In addition, we have to deal with the annoying random time-changes and the fact that many random elements are not defined on the whole probability space.

\vskip 3mm
Now we start the proof of Theorem \ref{DMP-thm}.
By conformal invariance, it suffices to work on standard chordal SLE$_\kappa$ curves. Let  $J(z)=-1/z$.  By reversibility of chordal SLE$_\kappa$, there are two standard chordal SLE$_\kappa$ curves $\eta_1$ and $\eta_2$, and a decreasing auto-homeomorphism $\phi_2$ of $[0,\infty]$ such that $\eta_1\circ \phi_2=J\circ \eta_2$. For $j=1,2$, let $\ha w_j$ be the chordal Loewner driving function for $\eta_j$,   $K^j_t$, $0\le t<\infty$, be the chordal Loewner hulls driven by $\ha w_j$, $\eta_j^J=J\circ \eta_j$, and $K^{j,J}_{t}=J(K^j_{t})$. 
We also view $\eta_j(t)$ and $\eta_j^J(t)$ as prime ends of $\HH\sem K^j_t$ and $\HH\sem K^{j,J}_t$, respectively.
For $j=1,2$, let $\F^j$ be the filtration generated by $\ha w_j$. Let ${\cal F}$ be the separable $\R_+^2$-indexed filtration generated by ${\cal F}^1$ and ${\cal F}^2$, and $\lin{\cal F}$ be the right-continuous  augmentation of $\cal F$. Let ${\cal D}_j=\{(t_1,t_2)\in\R_+^2:  \eta_{j}^J(t_{j})\not\in \eta_{3-j}[0,t_{3-j}]\}$, $j=1,2$, and ${\cal D}_{\cap} ={\cal D}_1\cap {\cal D}_2$. For $\ulin t\in{\cal D}_{\cap}$, there is a unique connected component of $\HH\sem (\eta_1[0,t_1]\cup \eta_2^J[0,t_2])$, which shares the prime end $\eta_1(t_1)$ with $\HH\sem K^1_{t_1}$ and the prime end $\eta_2^J(t_2)$ with $\HH\sem K^{2,J}_{t_2}$. Let this component be denoted by $H^1_{\ulin t}$.
Theorem \ref{DMP-thm} is equivalent to the following theorem.

\begin{Lemma}
  Let $\ulin T=(T_1,T_2)$ be an $\lin\F$-stopping time. Then the event $\{\ulin T\in{\cal D}_\cap\}\in\lin\F_{\ulin T}$, and on this event there are
  \begin{itemize}
\item  an $\F_{\ulin T}$-measurable random conformal map $f _{\ulin T}$ from $\HH$ onto $H^1_{\ulin T}$, which sends $0$ and $\infty$  to the prime ends $\eta_1(T_1)$ and $\eta_2^J(T_2)$, respectively,
    \item   a standard chordal SLE$_\kappa$ curve $\zeta_1^{\ulin T} $ conditionally independent of $\lin\F_{\ulin T}$, and
    \item   an increasing homeomorphism $u_1^{\ulin T}$ from $[0,\infty]$ onto $[T_1,\phi_2(T_2)]$,
  \end{itemize}
such that   $\eta_1\circ u_1^{\ulin T} =f_{\ulin T} \circ \zeta_1^{\ulin T}$ a.s.\ on $\{\ulin T\in\cal D_\cap\}$.
\label{DMP-lem}
\end{Lemma}

The rest of the appendix is devoted to the proof of Lemma \ref{DMP-lem}.
We first review some revised Carath\'eodory topology introduced in \cite{LERW}, which is similar to but different from the Carath\'eodory kernel convergence in \cite[Section 1.4]{Pom}.

\begin{Definition}
  Let $(D_n)_{n\in\N}$ and $D$ be domains in $\C$. We say that $(D_n)$ converges to $D$ in the Carath\'eodory topology, and write $D_n\dto D$, if
  \begin{enumerate}
    \item [(i)] for every compact set $K\subset D$, there exists $n_0\in\N$ such that $K\subset D_n$ if $n\ge n_0$;
    \item [(ii)] for every $z_0\in \pa D$ there exists $z_n\in\pa D_n$ for each $n$  such that $z_n\to z_0$.
  \end{enumerate}
\label{Def-Cara}
\end{Definition}

Let $(D_n)$ and $D$ be as in the definition.  Suppose $f_n:D_n\to \C$ and $f:D\to \C$.
By $f_n\luto f$ in $D$ we mean that $f_n$ converges to $f$ locally uniformly in $D$, i.e., uniformly on every compact subset of $D$. 
The following is \cite[Lemma 5.1]{LERW}.

\begin{Proposition}
  Let $D_n$, $n\in\N$, and $D$ be domains in $\C$ such that $D_n\dto D$. Let $f_n:D_n\conf E_n$ for some domain $E_n\subset \C$, $n\in\N$. Suppose $f_n\luto f$ in $D$ for some nonconstant $f:D\to \C$. Then $f$ is a conformal map, $E_n\dto E:=  f(D)$, and $f_n^{-1}\luto f^{-1}$ in $E$. \label{domain-converge-inverse}
\end{Proposition}

The following proposition is well known (cf.\ \cite[Proposition 2.3]{Two-Green-boundary}).

\begin{Proposition}
   If $K$ is an $\HH$-hull with $K\subset \{|z-x_0|\le r\}$ for some $x_0\in\R$ and $r>0$, then  $ |g_K(z)-z|\le 3r$ for any $z\in\C\sem K^{\doub}$. \label{g-z-sup}
\end{Proposition}

For a nonempty $\HH$-hull $K$, we write $a_K=\min(\lin K\cap\R)$, $b_K=\max(\lin K\cap \R)$, $B_K=[a_K,b_K]$, and $K^{\doub}=K\cup \{\lin z:z\in K\}\cup B_K$. For $K\subset L\in \cal H$ with $L\ne\emptyset$, by Schwarz reflection principle, $g_K$ extends to a conformal map from  $\C\sem (K^{\doub}\cup B_L)$ conformally onto $\C\sem S^L_K$ for some compact interval $S^L_K$. We write $S_K$ for $S_K^K$. Let
 $c^L_K<d^L_K $ and $c_K<d_K$ be such that $S^L_K=[c^L_K,d^L_K]$ and $S_K=[c_K,d_K]$.  Then $g_K$ maps $(-\infty,a_K)$ and $(b_K,\infty)$ respectively onto $(-\infty,c_K)$ and $(d_K,\infty)$, and satisfies that
\BGE g_K(x)<x,\quad \forall x\in(-\infty,a_K);\quad g_K(x)>x,\quad \forall x\in (b_K,\infty).\label{g><x}\EDE
Let $f_K=g_K^{-1}$. Then $f_K:\HH\conf \HH\sem K$ and $f_K:\C\sem S_K\conf \C\sem K^{\doub}$.



Let $\cal H$ denote the space of $\HH$-hulls. There is a metric $d_{\cal H}$ on $\cal H$ such that $K_n\to K$ w.r.t.\ $d_{\cal H}$ iff $f_{K_n}\luto f_K$ in $\HH$. By Proposition \ref{domain-converge-inverse}, this implies that $\HH\sem K_n\dto \HH\sem K$. But the converse is not true. For  $K\in{\cal H}$, let ${\cal H}(K)=\{L\in{\cal H}:L\subset K\}$. The following is \cite[Lemma 5.4 (i)]{LERW}.

\begin{Proposition}
   For any $K\in\cal H$ with $K\ne\emptyset$,  ${\cal H}(K)$ is compact w.r.t.\ $d_{\cal H}$. This means that every sequence $(K_n)$ in ${\cal H}(K)$ contains a convergent subsequence w.r.t.\ $d_{\cal H}$ with limit in ${\cal H}(K)$. Moreover, if $K_n\to K_0$ in ${\cal H}(K)$, then $f_{K_n}\luto f_{K_0}$ in $\C\sem S_K$, $\C\sem (K_n^{\doub}\cup B_K)\dto \C\sem (K_0^{\doub}\cup B_K)$, and $g_{K_n}\luto g_{K_0}$ in $\C\sem (K_0^{\doub}\cup B_K)$. \label{compact-hull}
\end{Proposition}

Let ${\cal P}$ denote the space of pairs $(K,p)$, where $K$ is an $\HH$-hull, and $p$ is a prime end of $\HH\sem K$ other than $\infty$. Note that $g_K(p)\in\R$ for $(K,p)\in\cal P$. Equip ${\cal P}$ with the metric
 $$d_{\cal P}((K_1,p_1),(K_2,p_2))=d_{\cal H}(K_1,K_2)+|g_{K_1}(p_1)-g_{K_2}(p_2)|.$$
 For $(K,p)\in {\cal P}$, let $g_{K,p}=g_K-g_K(p)$ and $f_{K,p}=g_{K,p}^{-1}$. If $(K_n,p_n)\to (K,p)$ w.r.t.\ $d_{\cal P}$, then $g_{K_n,p_n}\luto g_{K,p}$ in $\HH\sem K$ and $f_{K_n,p_n}\luto f_{K,p}$ in $\HH$.  For a nonempty $\HH$-hull $L$, let ${\cal P}(L)$ denote the set of  $(K,p)\in\cal P$ such that $K \in{\cal H}(L)$, and $p\not\in (-\infty,a_{L})\cup(b_{L} ,\infty)$. Then $g_K(p)\in S^L_K$. Let $S^L_{K,p}=S^L_K-g_K(p)$. Then $0\in S^L_{K,p}$ and $g_{K,p}:\C\sem (K^{\doub}\cup B_L)\conf \C\sem S^L_{K,p}$.

The following example is important. Suppose $K_t$, $0\le t<T$, are chordal Loewner   hulls driven by $\ha w$. For $0\le t<T$, let $p_t$ be the prime end $g_{K_t}^{-1}(\ha w(t))$  of $\HH\sem K_t$. Then  $(K_t,p_t)$, $0\le t<T$, is a continuous curve in $\cal P$. The corresponding maps $g_{K_t,p_t}$ and $f_{K_t,p_t}$ are called centered Loewner maps in the literature.

Let $\Xi$ denote the set of pairs $(L^1,L^2)\in{\cal H}^2$ such that $L^j$ contains a neighborhood of $0$ in $\HH$, $j=1,2$, and $\dist((L^1)^{\doub},J((L^2)^{\doub}))>0$. For each $\ulin L=(L^1,L^2)\in\Xi$, let ${\cal PP}(\ulin L)={\cal P}(L^1)\times {\cal P}(L^2)$. Let ${\cal PP}=\bigcup_{\ulin L\in \Xi} {\cal PP}(\ulin L)$. Suppose $(K_1,p_1;K_2,p_2)\in{\cal PP}$. Define $H^1_{K_1;K_2}=\HH\sem (K_1\cup J(K_2))$. Then $H^1_{K_1;K_2}$ is simply connected and shares the prime ends $p_1$ and $J(p_2)$ respectively with $\HH\sem K_1$ and $\HH\sem J(K_2)$. Let $g_{K_2,p_2}^J=J\circ g_{K_2,p_2}\circ J$ and $f_{K_2,p_2}^J=(g_{K_2,p_2}^J)^{-1}$.  Let $\til K_1,\til p_1$ be respectively the $ g_{K_2,p_2}^J$-images of $K_1,p_1$.  Since  $g_{K_2,p_2}^J:H^1_{K_1;K_2}\conf \HH\sem \til K_1$, we see that $\HH\sem \til K_1$ is also simply connected. Since $J(p_2)$ is a prime end of  $H^1_{K_1,K_2}$ bounded away from $K_1$, and is sent to $\infty$ by $g_{K_2,p_2}^J$, we see that $\HH\sem \til K_1$ shares the prime end $\infty$ with $\HH$. Thus, $\til K_1\in\cal H$. Since $p_1$ and $J(p_2)$ are distinct prime ends of  $\HH_{K_1,K_2}$, we have $\til p_1=g_{K_2,p_2}^J(p_1)\ne g_{K_2,p_2}^J(J(p_2))=\infty$ in terms of prime ends. So $(\til K_1,\til p_1)\in\cal P$. 


\begin{Lemma}
  Let $(L^1,L^2)\in \Xi$. For $j=1,2$, let $(K^j_n,p^j_n)$, $n\in\N$, be a convergent sequence in ${\cal P}(L^j)$ with limit $(K^j_\infty,p^j_\infty)\in {\cal P}(L_j)$. For $n\in\lin\N:=\N\cup\{\infty\}$, let $\til K^1_n,\til p^1_n$ be respectively the $ g_{K^2_n,p^2_n}^ J$-images of $K^1_n,p^1_n$. Then $(\til K^1_n,\til p^1_n)\to (\til K^1_\infty,\til p^1_\infty)$ w.r.t.\ $d_{\cal P}$, $g_{\til K^1_n}\luto g_{\til K^1_\infty}$ in $\C\sem ((\til K^1_\infty)^{\doub}\cup g_{K^2_\infty,p^2_\infty}^J(B_{L^1}))$, and $g_{\til K^1_n}^{-1}\luto g_{\til K^1_\infty}^{-1}$ in $\C\sem (S_{\til K^1_\infty}\cup g_{\til K^1_\infty}\circ g_{K^2_\infty,p^2_\infty}^J(B_{L^1}))$.
   \label{compact3'}
\end{Lemma}
\begin{proof}
  From $K^j_n\to K^j_\infty$, $j=1,2$, we know that $H^1_{K^1_n;K^2_n}\dto H^1_{K^1_\infty;K^2_\infty}$. Since $g_{K^2_n,p^2_n}^J\luto g_{K^2_\infty,p^2_\infty}^J$ in $H^1_{K^1_\infty,K^2_\infty}$, by Proposition \ref{domain-converge-inverse}, $\HH\sem \til K^1_n\dto \HH\sem \til K^1_\infty$. Let $\Lambda=\Lambda(L^1,L^2)\in (0,\infty)$ be the extremal distance (cf.\ \cite{Ahl}) between $(L^2)^{\doub}$ and $J((L^1)^{\doub})$. By comparison principle, for any $n\in\N$, the extremal distance between $(K^2_n)^{\doub}\cup B_{L^2}$ and $J((K^1_n)^{\doub}\cup B_{L^1})$ is at least $\Lambda$. By conformal invariance, the extremal distance between $S^{L^2}_{K^2_n,p^2_n}$ and $g_{K^2_n,p^2_n}\circ J((K^1_n)^{\doub})$ is at least $\Lambda$. Since $S^{L^2}_{K^2_n,p^2_n}$ contains $0$ and has length $\ge b_{L^2}-a_{L^2}>0$ by (\ref{g><x}), there is $r=r(L^1,L^2)>0$ such that $\dist(0,g_{K^2_n,p^2_n}\circ J(K^1_n))\ge r$, which implies that $\til K^1_n\subset \{|z|\le 1/r\}$. 
   By Proposition \ref{compact-hull}, $\{\til K^1_n\}$ is pre-compact. Since $\HH\sem \til K^1_n\dto \HH\sem \til K^1_\infty$, we get $\til K^1_n\to \til K^1_\infty$ w.r.t.\ $d_{\cal H}$.

For each $n\in\lin\N$,  $g_n:=g_{\til K^1_n}\circ g_{K^2_n,p^2_n}^J\circ f_{K^1_n}$ is a conformal map from $\C\sem (S_{K^1_n}\cup g_{K^1_n}\circ J((K^2_n)^{\doub}))$ onto $\C\sem( S_{\til K^1_n}\cup g_{\til K^1_n}\circ J(S_{K^2_n}))$, and maps a neighborhood of $S_{K^1_n}$ to a neighborhood of  $S_{\til K^1_n}$. By Schwarz reflection principle, it extends to a conformal map on $\C\sem g_{K^1_n}\circ J((K^2_n)^{\doub})$, and maps $S_{K^1_n}$ onto $S_{\til K^1_n}$. From $K^1_n\to K^1_\infty$, $\til K^1_n\to \til K^1_\infty$ and $(K^2_n,p^2_n)\to (K^2_\infty,p^2_\infty)$, we know that $g_n\luto g_0$ in $\C\sem ( S^{L_1}_{K^1_\infty} \cup g_{K^1_\infty}\circ J((K^2_\infty)^{\doub}))$. By maximum principle, we get $g_n\luto g_0$ in $\C\sem g_{K^1_\infty}\circ J((K^2_\infty)^{\doub})$. Since $g_{\til K^1_n}(\til p^1_n)=g_n\circ g_{K^1_n}(p^1_n)$ for $n\in\lin\N$,  and $g_{K^1_n}(p^1_n)\to g_{K^1_\infty}(p^1_\infty)\in \C\sem g_{K^1_\infty}\circ J((K^2_\infty)^{\doub})$, we get $g_{\til K^1_n}(\til p^1_n)\to g_{\til K^1_\infty}(\til p^1_\infty)$.
Thus, $(\til K^1_n,\til p^1_n)\to (\til K^1_\infty,\til p^1_\infty)$.

Since $K^1_n\to K^1_\infty$ in ${\cal H}(L^1)$, we get $\C\sem((K^1_n)^{\doub}\cup B_{L^1})\dto \C\sem ((K^1_\infty)^{\doub}\cup B_{L^1})$. By Proposition \ref{domain-converge-inverse}, we get
$\C\sem ((\til K^1_n)^{\doub} \cup g_{K^2_n,p^2_n}^J(B_{L^2}))\dto \C\sem ((\til K^1_\infty)^{\doub} \cup g_{K^2_\infty,p^2_\infty}^J(B_{L^1}))$.
From the first paragraph of the proof, there is $R\in(0,\infty)$ such that $\til K^1_n\subset \{|z|<R\}$ for all $n\in\lin\N$. By Proposition \ref{g-z-sup}, $|g_{\til K^1_n}(z)-z|\le 3R$ for all $n\in\lin\N$ and $z\in \C\sem (\til K^1_n)^{\doub}$. Thus, any subsequence of $\{g_{\til K^1_n}\}$ contains a further subsequence, which converges locally uniformly in $\C\sem ((\til K^1_\infty)^{\doub} \cup g_{K^2_\infty,p^2_\infty}^J(B_{L^1}))$. The limit must be $g_{\til K^1_\infty}$ since we already know that $g_{\til K^1_n}\luto g_{\til K^1_\infty}$ in $\HH\sem \til K^1_\infty$. So we get $ g_{\til K^1_n} \luto g_{\til K^1_\infty}$ in $\C\sem ((\til K^1_\infty)^{\doub} \cup g_{K^2_\infty,p^2_\infty}^J(B_{L^1}))$. By Proposition \ref{domain-converge-inverse},  $g_{\til K^1_n}^{-1}\luto g_{\til K^1_\infty}^{-1}$ in
$g_{\til K^1_\infty}(\C\sem ((\til K^1_\infty)^{\doub} \cup g_{K^2_\infty,p^2_\infty}^J(B_{L^1})))=\C\sem (S_{\til K^1_\infty}\cup g_{\til K^1_\infty}\circ g_{K^2_\infty,p^2_\infty}^J(B_{L^1})))$.
\end{proof}

 If $\eta(t)$, $0\le t<T$, is a continuous curve in $\lin\HH$, and there are a continuous function $\ha w$ on $[0,T)$ and a continuous and strictly increasing function $u$ on $[0,T)$ with $u(0)=0$, such that $\eta\circ u^{-1}(t)$, $0\le t<u(T)$, is the chordal Loewner curve driven by $\ha w\circ u^{-1}$, then we say that $\eta$ is a chordal Loewner curve with speed $du$ driven by $\ha w$. For each $0\le t<T$, we also understand  $\eta(t)$ as the prime end $g_{\eta[0,t]}^{-1}(\ha w(t))$ of $\HH\sem \Hull(\eta[0,t])$. The following  proposition is a well-known result in the Loewner theory, and so we omit its proof.

\begin{Proposition}
	Suppose $\eta(t)$, $0\le t<T$, is a chordal Loewner curve. Let $I$ be an open real interval, which contains $\eta\cap\R$.   Let $U$ be a subdomain of $\HH$, which is a neighborhood of $I$ in $\HH$, and contains  $\bigcup_{0\le t<T} \Hull(\eta[0,t])$. Let $\psi$ be a conformal map from $U$ into $\HH$, which extends continuously to $U \cup I$, and maps $I$ into $\R$.  Then $\psi\circ \eta(t)$, $0\le t<T$, is a chordal Loewner curve with some speed, and for each $t\in[0,T)$, $\psi$ sends the prime end $\eta(t)$ of $\HH\sem \Hull(\eta[0,t])$ to the prime end $\psi\circ \eta(t)$ of $\HH\sem \Hull(\psi\circ\eta[0,t])$. 
\label{Loewner-KL-chordal}
\end{Proposition}

Now we come back to the proof of Lemma \ref{DMP-lem}. For $j=1,2$ and $t\ge 0$, let $g^j_t$ denote the centered Loewner map for $\eta_j$ at time $t$, i.e.,   $g^j_t=g_{K^j_t}-\ha w_j(t)$. Let $f^j_t=(g^j_t)^{-1}$ and $f^{j,J}_t=J\circ f^j_t\circ J$. 
Fix $t_2\ge 0$. By DMP of SLE, conditionally on $\F^2_{t_2}$, $\eta_2(t_2+\cdot)$ is a chordal SLE$_\kappa$ in $\HH\sem K^2_{t_2}$ from $\eta_2(t_2)$ to $\infty$. Since $f^2_{t}$ maps $\HH$ conformally onto $\HH\sem K^2_t$, extends continuously to $\lin\HH^\#$, and sends $0$ and $\infty$ respectively to $\eta_2(t)$ and $\infty$, there is a standard chordal SLE$_\kappa$ curve $\eta_2^{t_2+}$ independent of $\F^2_{t_2}$ such that a.s.\ $\eta_2(t_2+\cdot)= f^2_{t_2}\circ \eta_2^{t_2+}$. The $\eta_2^{t_2+}$ is the chordal Loewner curve driven by $\ha w_2^{t_2+}:=\ha w_2(t_2+\cdot)-\ha w_2(t_2)$. Let $K^{2,t_2+}_{t}=\Hull(\eta_2^{t_2+}[0,t])$ and $g^{2,t_2+}_t$ be the centered Loewner map for $\eta_2^{t_2+}$, i.e., $g_{K^{2,t_2+}_t}-\ha w_2^{t_2+}(t)$. Then $f^2_{t_2}=f^2_{t_2+t}\circ g^{2,t_2+}_t$, $t\ge 0$.

Since $\eta_1(t)$, $0\le t\le \phi_2(t_2)$, is a time-reversal of $ \eta^J_2(t)$, $t_2\le t\le \infty$,
 by reversibility of SLE$_\kappa$, conditionally on $\F^2_{t_2}$, $\eta_1(t)$, $0\le t\le \phi_2(t_2)$, is a time-change of a chordal SLE$_\kappa$ in $\HH\sem K^{2,J}_{t_2}$ from $0$ to $\eta^J_2(t_2)$.  
Since $f^{2,J}_{t_2}:(\HH;0,\infty)\conf ( \HH\sem K^{2,J}_{t_2};0, \eta^J_2(t_2))$, there is a standard chordal SLE$_\kappa$ curve $\zeta_1^{t_2}$ independent of $\F^2_{t_2}$, and an increasing homeomorphism $u_1^{t_2}$ from $[0,\phi_2(t_2)]$ onto $[0,\infty]$, such that a.s.\ $\eta_1=  f^{2,J}_{t_2}\circ \zeta_1^{t_2}\circ u_1^{t_2}$ on $[0,\phi_2(t_2)]$. Let $\ha v^{t_2}_1$ be the driving function for $\zeta_1^{t_2}$. Let $\eta^{t_2}_1=\zeta^{t_2}_1\circ u^{t_2}_1$ and $\ha w^{t_2}_1=\ha v^{t_2}_1\circ u^{t_2}_1$. Then $\eta^{t_2}_1$ is the chordal Loewner curve with speed $d u^{t_2}_1$ driven by $\ha w^{t_2}_1$, and a.s.\  $\eta_1= f^{2,J}_{t_2}\circ \eta_1^{t_2}$ on $[0,\phi_2(t_2)]$.

\begin{Lemma}
The functions $\eta_1^{t_2}(t_1),u_1^{t_2}(t_1),\ha w_1^{t_2}(t_1)$ all have  continuous versions on ${\cal D}_2=\{(t_1,t_2):\eta_2^J(t_2)\not\in\eta_1[0,t_1]\}$. More specifically, there are continuous processes $\eta_1^*:{\cal D}_2\to\lin\HH$, $u_1^*:{\cal D}_2\to\R_+$ and $\ha w_1^*:{\cal D}_2\to \R$ such that  for any fixed $t_2\in\R_+$, $\eta_1^*(\cdot,t_2)$ is the chordal Loewner curve with speed $du_1^*(\cdot,t_2)$  driven by $\ha w_1^*(\cdot,t_2)$; and
a.s.\ $\eta_1^*(\cdot,t_2)$, $u_1^*(\cdot,t_2)$ and $\ha w_1^*(\cdot,t_2)$ are respectively equal to $\eta_1^{t_2}$, $u_1^{t_2}$ and $\ha w_1^{t_2}$, which further implies that a.s.\ $\eta_1|_{[0,\phi_2(t_2))}=f^{2,J}_{t_2}\circ \eta_1^*(\cdot,t_2)$. Here for two functions to be equal, they must have the same domain.
\label{version}
\end{Lemma}
\begin{proof}
Let ${\cal A}$ denote the set of triples $(t_1^0;t_2^0,t_2^1)$, where $t_1^0,t_2^0,t_2^1\in\Q_+$ and $t_2^0<t_2^1$. For each $\ulin v=(t_1^0;t_2^0,t_2^1)\in\cal A$, define $R_{\ulin v}=[0,t_1^0]\times [t_2^0,t_2^1]$ and $E_{\ulin v}=\{R_{\ulin v}\subset {\cal D}_2\}$.
We make the following claim: for any $\ulin v=(t_1^0;t_2^0,t_2^1)\in\cal A$, on the event $E_{\ulin v}$, there are continuous processes $\eta_1^*:R_{\ulin v}\to\lin\HH$, $u_1^*:R_{\ulin v}\to\R_+$ and $\ha w_1^*:R_{\ulin v}\to \R$ such that for any $t_2\in[t_2^0,t_2^1]$, $\eta_1^*(\cdot,t_2)$ is the chordal Loewner curve with speed $du_1^*(\cdot,t_2)$  driven by $\ha w_1^*(\cdot,t_2)$; and a.s.\ $\eta_1^*(\cdot,t_2)$, $u_1^*(\cdot,t_2)$ and $\ha w_1^*(\cdot,t_2)$ are respectively equal to the restrictions of $\eta_1^{t_2}$, $u_1^{t_2}$ and $\ha w_1^{t_2}$ to $[0,t_1^0]$.

Fix $\ulin v=(t_1^0;t_2^0,t_2^1)\in\cal A$. Let $\Delta t_2=t_2^1-t_2^0$. Suppose $E_{\ulin v}$ happens. Then $\eta^J_2[t_2^0,t_2^1]\cap \eta_1[0,t_1^0]=\emptyset$. Since $\eta^J_2[t_2^0,t_2^1]=f^{2,J}_{t_2^0}\circ J \circ \eta_2^{t_2^0+}[0,\Delta t_2]$ and $\eta^1[0,t_1^0]=f^{2,J}_{t_2^0}\circ \eta_1^{t_2^0}[0,t_1^0]$, we get
 $(J\circ \eta_2^{t_2^0+}[0,\Delta t_2])\cap \eta^{t_2^0}_1[0,t_1^0]=\emptyset$. Since both $\eta^{t_2}_1$ and $\eta_2^{t_2^0+}$ start from $0$, we can find a (random) pair $\ulin L=(L^1,L^2)\in\Xi$ such that  $K^{1,t_2^0}_{t_1^0}\subset L^1$ and $K^{2,t_2^0+}_{\Delta t_2}\subset L^2$. Then $(K^{1,t_2^0}_{t_1},\eta_1^{t_2^0}({t_1}))$, $0\le t_1\le t_1^0$, and $(K^{2,t_2^0+}_{t_2},\eta_2^{t_2^0+}(t_2))$, $0\le t_2\le \Delta t_2$, are respectively continuous curves in ${\cal P}(L^1)$ and ${\cal P}(L^2)$.

Let $g^{2,t_2^0+,J}_{t}=J\circ g^{2,t_2^0+}_{t}\circ J$. Let $t_2\in [t_2^0,t_2^1]$.  We have a.s., for all $t_1\in[0,t_1^0]$,
$$f^{2,J}_{t_2}\circ \eta_1^{t_2}(t_1)=\eta_1(t_1)=f^{2,J}_{t_2^0}\circ \eta_1^{t_2^0}(t_1) = f^{2,J}_{t_2}\circ g^{2,t_2^0+,J}_{t_2-t_2^0} \circ \eta_1^{t_2^0}(t_1),$$
which implies that a.s.\ for all $t_1\in S_1:= [0,t_1^0]\cap (\eta_1^{t_2^0})^{-1}(\HH)$,    $\eta_1^{t_2}(t_1)= g^{2,t_2^0+,J}_{t_2-t_2^0}\circ \eta_1^{t_2^0}(t_1)$. Since $S_1$ is dense in $[0,t_1^0]$, both $\eta_1^{t_2}$ and $\eta_1^{t_2^0}$ are continuous on $[0,t_1^0]$, and $ g^{2,t_2^0+,J}_{t_2-t_2^0} $ is continuous on $\lin\HH\sem J(\lin{K^{2,t_2^0+}_{t_2-t_2^0}})\supset \lin\HH\sem J(\lin {L^2})\supset \lin {L^1}\supset \eta_1^{t_2^0}[0,t_1^0]$, we get a.s.\ $ \eta^{t_2}_1 = g^{2,t_2^0+,J}_{t_2-t_2^0} \circ \eta^{t_2^0}_1 $ on $[0,t_1^0]$.

Define $\eta_1^*$ and $u_1^*$ by $\eta_1^*(t_1,t_2)= g^{2,t_2^0+,J}_{t_2-t_2^0} \circ \eta^{t_2^0}_1(t_1)$ and $u_1^*(t_1,t_2)=\hcap_2(\eta_1^*([0,t_1],t_2))$ for $(t_1,t_2)\in  R_{\ulin v}$. Then $\eta_1^*$ and $u_1^*$ are continuous on $R_{\ulin v}$ by Proposition \ref{compact-hull}. By the last paragraph, for any $t_2\in[t_2^0,t_2^1]$, a.s.\ $\eta_1^*(\cdot,t_2)=\eta_1^{t_2}$  and $\eta_1 =f^{2,J}_{t_2}\circ \eta_1^*(\cdot,t_2)$ on $[0,t_1^0]$. So $\eta_1^*$ and $u_1^*$ are continuous versions of $\eta_1^{t_2}(t_1)$ and $u_1^{t_2}(t_1)$ on $R_{\ulin v}$. By the continuity of $\eta_1$, $\eta_1^*$, and $f^{2,J}_{t_2}$, we then get a.s.\ $\eta_1(t_1)=f^{2,J}_{t_2}\circ \eta_1^*(t_1,t_2)$ for any $(t_1,t_2)\in R_{\ulin v}$.
By Proposition \ref{Loewner-KL-chordal}, $\eta_1^*(\cdot,t_2)$ is a chordal Loewner curve with speed $du_1^*(\cdot,t_2)$. Define $\ha w_1^*: R_{\ulin v}\to \R$ such that for every $t_2\in[t_2^0,t_2^1]$, $\ha w_1^*(\cdot,t_2)$ is the driving function for $\eta_1^*(\cdot,t_2)$. For any $t_2\in[t_2^0,t_2^1]$, since a.s.\ $\eta_1^*(\cdot,t_2)=\eta_1^{t_2}$ on $[0,t_1^0]$, and $\ha w_1^{t_2}$ is the driving function for $\eta_1^{t_2}$, we find that a.s.\ $\ha w_1^*(\cdot,t_2)=\ha w_1^{t_2}$ on $[0,t_1^0]$. To prove the claim, it remains to prove that $\ha w_1^*$ is a.s.\ continuous on $R_{\ulin v}$.

From $\eta_1=f^{2,J}_{t_2}(\eta_1^*(\cdot,t_2))=f^{2,J}_{t_2^0}(\eta_1^*(\cdot,t_2^0))$, we get $\eta_1^*(\cdot,t_2)=g^{2,t_2^0+,J}_{t_2-t_2^0}\circ \eta_1^*(\cdot,t_2^0)$, $t_2\ge t_2^0$. For $\ulin t=(t_1,t_2)\in R_{\ulin v}$, let $K^{1,t_2}_{t_1}=\Hull(\eta_1^*([0,t_1],t_2))$. Then we have $K^{1,t_2}_{t_1}=g^{2,t_2^0+,J}_{t_2-t_2^0}(K^{1,t_2^0}_{t_1})$ and $\ha w_1^*(\ulin t)=g_{K^{1,t_2}_{t_1}}(\eta_1^*(\ulin t))$. 
Define $G_{\ulin t}=g_{K^{1,t_2}_{t_1}}\circ  g^{2,t_2^0+,J}_{t_2-t_2^0} \circ g_{K^{1,t_2^0}_{t_1}}^{-1}$, which maps $\HH\sem (g_{K^{1,t_2^0}_{t_1}}\circ J(K^{2,t_2^0+}_{t_2-t_2^0}))$ conformally onto $\HH$, and extends to a conformal map from $\C\sem ( g_{K^{1,t_2^0}_{t_1}}\circ J((K^{2,t_2^0+}_{t_2-t_2^0})^{\doub}))$, which contains $\ha w_1^*(t_1^0,t_2)$, into $\C $.   From $\ha w_1^*(\ulin t)=g_{K^{1,t_2}_{t_1}}(\eta_1^*(\ulin t))$, $\ha w_1^*(t_1,t_2^0)=g_{K^{1,t_2^0}_{t_1}}(\eta_1^*(t_1,t_2^0))$, and $\eta_1^*(\cdot,t_2)=g^{2,t_2^0+,J}_{t_2-t_2^0}\circ \eta_1^*(\cdot,t_2^0)$, we get $w_1^*(\ulin t)=G_{\ulin t}(w_1^*(t_1,t_2^0))$. Since $w_1^*(\cdot,t_2^0)=w_1^{t_2^0}$, which is continuous in $t_1$, to show that $w_1^*$ is continuous in $\ulin t$, it suffices to show that $G_{\ulin t}(x)$ is jointly continuous in $\ulin t$ and $x$ on a domain that contains $\{(\ulin t,x): \ulin t\in R_{\ulin v}, x\in S_{K^{1,t_2}_{t_1}}\}$.

Since $G_{\ulin t}$ is a conformal map for each $\ulin t\in\ulin v$, to prove the joint continuity, by maximum principle, it suffices to show that, for any fixed $\ulin t^*=(t_1^*,t_2^*)\in R_{\ulin v}$, there is a Jordan curve $\xi$, whose interior domain contains $S_{K^{1,t_2^0}_{t_1^*}}$, and is contained in $\C\sem ( g_{K^{1,t_2^0}_{t_1^*}}\circ J((K^{2,t_2^0+}_{t_2-t_2^0})^{\doub}))$, such that as $R_{\ulin v}\ni \ulin t^n\to \ulin t^*$, $G_{\ulin t^n}\to G_{\ulin t^*}$ uniformly on $\xi$. To choose such $\xi$, we may first find a Jordan curve $\xi'$, which separates $(L^1)^{\doub}$ from $J((L^2)^{\doub})$, and then let $\xi=g_{K^{1,t_2^0}_{t_1^*}}(\xi')$. Then we use Lemma \ref{compact3'}  to conclude that  $G_{\ulin t^n}\to G_{\ulin t^*}$ uniformly on $\xi$ if $R_{\ulin v}\ni \ulin t^n\to \ulin t^*$.

Now we have proved the claim. Since ${\cal D}_2$ is the union of $R_{\ulin v}$ over those $\ulin v\in\cal A$ such that $R_{\ulin v}\subset {\cal D}_2$, from the claim we conclude the existence of continuous maps $\eta_1^*$, $u_1^* $ and $\ha w_1^*$ defined on ${\cal D}_2$ with all properties described in the statement   except that now we can only say that a.s.\ $\eta_1^*(\cdot,t_2)$, $u_1^*(\cdot,t_2)$ and $\ha w_1^*(\cdot,t_2)$ are respectively equal to the restrictions of $\eta_1^{t_2}$, $u_1^{t_2}$ and $\ha w_1^{t_2}$ to $[0,T_2(t_2))$, where  $T_2(t_2):=\inf\{t_1:\eta_1(t_1)=\eta_2^J(t_2)\}$. From $\eta_1|_{[0,\phi_2(t_2))}=f^{2,J}_{t_2}\circ \eta_1^{t_2}$ we know that a.s.\ $f^{2,J}_{t_2}\circ \eta_1^*(\cdot,t_2)$ equals $\eta_1|_{[0,T_2(t_2))}$. We know that $\eta_1^{t_2}$, $u_1^{t_2}$ and $\ha w_1^{t_2}$ are defined on $[0,\phi_2(t_2))$. To conclude the proof, it suffices to show that a.s.\ $T_2(t_2)= \phi_2(t_2)$.

Since $\eta_1|_{[0,\phi_2(t_2)]}$ is a time-change of a chordal SLE$_\kappa$ in $\HH\sem J(K^{2,J}_{t_2})$ from $0$ to the prime end $\eta_2^J(t_2)$, $\eta_1$ first  visits the prime end $\eta_2^J(t_2)$ at the time $\phi_2(t_2)$. We claim that a.s.\ $\eta_2^J(t_2)$ is the only prime end of $\HH\sem J(K^{2,J}_{t_2})$, which determines the boundary point $\eta_2^J(t_2)$. When $\kappa\le 4$ this is true since  $\eta_2^J$ is simple. When $\kappa>4$, this follows from \cite[Theorem 6.1]{duality2}. Thus, a.s.\ $\eta_1$ first visits the point $\eta_2^J(t_2)$ at the time $\phi_2(t_2)$.  So we get a.s.\ $T_2(t_2)= \phi_2(t_2)$, as desired.
\end{proof}

By discarding a null event we may assume that $\eta_1(t_1)=f^{2,J}_{t_2}\circ \eta_1^*(t_1,t_2)$ for every $(t_1,t_2)\in{\cal D}_2$.
For $\ulin t=(t_1,t_2)\in{\cal D}_2$, let $K^{1,t_2}_{t_1}=\Hull(\eta_1^*([0,t_1],t_2))$ and $f^{1,t_2}_{t_1}$ be the centered Loewner map $f_{K^{1,t_2}_{t_1},\eta^{*}_1(t_1,t_2)}$. Let $f^1_{\ulin t}=f^{2,J}_{t_2}\circ f^{1,t_2}_{t_1}$.  It is clear that $\{\ulin t\in{\cal D}_2\}=\{\eta_2^J(t_2)\not\in \eta_1[0,t_1]\}\in\F_{\ulin t}$, and $f^{2,J}_{t_2}$ is $\F^2_{t_2}$-measurable.
Since $\eta_1^*(t,t_2)$, $0\le t\le t_1$, are determined by $\eta_1|_{[0,t_1]}$ and $\eta_2|_{[0,t_2]}$, $f^{1,t_2}_{t_1}$  defined on the event $\{\ulin t\in{\cal D}_2\}$ is $\F_{\ulin t}$-measurable, and so is  $f^1_{\ulin t}=f^{2,J}_{t_2}\circ f^{1,t_2}_{t_1}$.
Since $\eta_1^*$ is continuous on ${\cal D}_2$, by Proposition \ref{compact-hull}, $(t_1,t_2;z)\mapsto f^{1,t_2}_{t_1}(z)$ is jointly continuous on ${\cal D}_2\times\HH$. Since the same is obviously true for $f^{2,J}_{t_2}(z)$,  $f^1_{\ulin t}(z)$ is also jointly continuous on ${\cal D}_2\times \HH$.



\begin{Lemma}
\begin{enumerate}
\item [(i)] For every $\ulin t=(t_1,t_2)\in{\cal D}_2$, $f^1_{\ulin t}$ is a conformal map on $\HH$,  has continuation on $\lin\HH^\#$, and maps $0$ and $\infty$ respectively to the points $\eta_1(t_1)$ and $\eta_2^J(t_2)$. Moreover, $f^1_{\ulin t}(\HH)$ is the connected component of $\HH\sem (\eta_1[0,t_1]\cup \eta_2^J[0,t_2])$ which  shares the prime end $\eta^J_2(t_2)$ with $\HH\sem K^{2,J}_{t_2}$; and $f^1_{\ulin t}$ sends the prime end $\infty$ to the prime end $\eta_2^J(t_2)$.
    \item [(ii)] If $\ulin t=(t_1,t_2)\in{\cal D}_{\cap}$, then besides the results in (i), $f^1_{\ulin t}(\HH)$ also shares the prime end $\eta_1(t_1)$ with $\HH\sem K^1_{t_1}$, and $f^1_{\ulin t}$ sends the prime end $0$  to the prime end  $\eta_1(t_1)$.
    \end{enumerate}
\label{f1t}
\end{Lemma}
\begin{proof}
(i)  Let $\ulin t=(t_1,t_2)\in{\cal D}_2$.  Since $f^{2,J}_{t_2}$ and $f^{1,t_2}_{t_1}$ are both conformal maps from $\HH$ into $\HH$, and have continuation on $\lin\HH^\#$, the same is true for $f^1_{\ulin t}=f^{2,J}_{t_2}\circ f^{1,t_2}_{t_1}$. Since $f^{1,t_2}_{t_1}(0)=\eta_1^*(\ulin t)$, we get $f^1_{\ulin t}(0)=f^{2,J}_{t_2}\circ \eta_1^*(\ulin t)=\eta_1(t_1)$. Since $f^{1,t_2}_{t_1}(\HH)=\HH\sem K^{1,t_2}_{t_1}$ is a connected component of $\HH\sem \eta_1^*([0,t_1],t_2)$, $f^1_{\ulin t}(\HH)$ is a connected component of $f^{2,J}_{t_2}(\HH)\sem f^{2,J}_{t_2}\circ \eta_1^*([0,t_1],t_2)=\HH\sem (K^{2,J}_{t_2}\cup \eta_1[0,t_1])$,
which is also a connected component of $\HH\sem (\eta_2^J[0,t_2]\cup \eta_1[0,t_1])$.

 Since $f^{1,t_2}_{t_1}(\HH)=\HH\sem K^{1,t_2}_{t_1}$ is the connected component of $\HH\sem \eta_1^*([0,t_1],t_2)$ that shares the prime end $\infty$ with $\HH$,  we see that $f^1_{\ulin t}(\HH)=f^{2,J}_{t_2}(\HH\sem K^{1,t_2}_{t_1})$ is contained in and shares the prime end $f^{2,J}_{t_2}(\infty)=\eta_2^J(t_2)$ with $f^{2,J}_{t_2}(\HH)=\HH\sem  K^{2,J}_{t_2}$. Moreover, in terms of prime ends, we have
 $$f^1_{\ulin t}(\infty)=J\circ f^2_{t_2}\circ J\circ f^{1,t_2}_{t_1}(\infty)=J\circ f^2_{t_2}\circ J(\infty)
 =J\circ f^2_{t_2}(0)= \eta^J(t_2).$$



(ii) Suppose now $\ulin t\in{\cal D}_{\cap}$. Then $\eta_1(t_1)\not\in \eta_2^J[0,t_2]$. So there is a connected component of $\HH\sem (\eta_1[0,t_1]\cup \eta_2^J[0,t_2])$, which is contained in and shares the prime end $\eta_1(t_1)$ with $\HH\sem K^1_{t_1}$. Let the domain be denoted by $D^1_{\ulin t}$.  There is $\delta>0$ such that $\eta_1[t_1,t_1+\delta]\cap \eta_2^J[0,t_2]=\emptyset$. Let $S=\{t\in[t_1,t_1+\delta]:\eta_1(t)\in \HH\sem K^1_{t_1}\}$. Then $\eta_1(S)\subset D^1_{\ulin t}$, and $S$ is dense in $[t_1,t_1+\delta]$. Since $\eta_1 =f^{2,J}_{t_2}\circ \eta_1^*(\cdot,t_2)$, we get $\eta_1^*(S,t_2)\subset \HH\sem \eta_1^*([0,t_1],t_2)$, which implies that $\eta_1^*(S,t_2)\subset \HH\sem K^{1,t_2}_{t_1}=f^{1,t_2}_{t_1}(\HH)$. So $\eta_1(S)\subset f^{2,J}_{t_2}\circ f^{1,t_2}_{t_1}(\HH)=f^1_{\ulin t}(\HH)$. Now $f^1_{\ulin t}(\HH)$ and $D^1_{\ulin t}$ are both some connected component of $\HH\sem (\eta_1[0,t_1]\cup \eta_2^J[0,t_2])$, and have nonempty intersection: $\eta_1(S)$. So they must agree. Thus, $f^1_{\ulin t}(\HH)$ is contained in and shares the prime end $\eta_1(t_1)$ with $\HH\sem K^1_{t_1}$.

As $t\downarrow t_1$ along $S$, $\eta_1^*(t,t_2)$ and $\eta_1(t)=f^{2,J}_{t_2}\circ \eta_1^*(t,t_2)$ respectively approach the prime ends $\eta_1^*(t_1,t_2)$ of $\HH\sem K^{1,t_2}_{t_1}$ and $\eta_1(t_1)$ of $D^1_{\ulin t}$. So $f^{2,J}_{t_2} $ sends the prime end $\eta_1^*(\ulin t)$ to the prime end $\eta_1(t_1)$. Thus, in terms of prime ends, $f^1_{\ulin t}(0)=f^{2,J}_{t_2}\circ f^{1,t_2}_{t_1}(0)=f^{2,J}_{t_2}(\eta_1^*(\ulin t))=\eta_1(t_1)$.
\end{proof}

\begin{Lemma}
  For any $\lin\F$-stopping time $\ulin T$, the event $\{\ulin T\in{\cal D}_2\}$ is $\lin{\F}_{\ulin T}$-measurable, and $f^1_{\ulin T}$ is an $\lin{\F}_{\ulin T}$-measurable random conformal map defined on the event $\{\ulin T\in{\cal D}_2\}$. \label{TinD}
\end{Lemma}
\begin{proof} Let $\ulin T=(T_1,T_2)$ be an $\lin\F$-stopping time. For each $n\in\N$ and $j=1,2$, define $T^{\lfloor n\rfloor}_j=2^{-n}\lfloor 2^n T_j\rfloor$ and $T^n_j=T^{\lfloor n\rfloor}_j+2^{-n}$. Let
$\ulin T^{\lfloor n\rfloor}=(T^{\lfloor n\rfloor}_1,T^{\lfloor n\rfloor}_2)$ and $\ulin T^n=(T^n_1,T^n_2)$. Then $\ulin T^{\lfloor n\rfloor}\le T\le \ulin T^{n}$, $\ulin T^n$ is an $\F$-stopping time, $\ulin T^{\lfloor n\rfloor}\uparrow \ulin T$, and $\ulin T^n\downarrow \ulin T$. For each $n\in\N$ and $\ulin t\in (2^{-n}\N)^2$, by Proposition \ref{T<S}
$$\{[\ulin T^{\lfloor n\rfloor},\ulin T^n]\subset {\cal D}_2\}\cap \{\ulin T^n=\ulin t\}
=\{[\ulin t-(2^{-n},2^{-n}),\ulin t]\subset {\cal D}_2\} \cap \{\ulin T^n=\ulin t\}$$
$$=\{\eta_2^J[t_2-2^{-n},t_2]\cap \eta_1 [0,t_1]=\emptyset\} \in \F_{\ulin t}\cap \{\ulin T^n=\ulin t\}\subset \F_{\ulin T^n}.$$
Thus, $\{[\ulin T^{\lfloor n\rfloor},\ulin T^n]\subset {\cal D}_2\}=\bigcup_{\ulin t\in (2^{-n}\N)^2} (\{[\ulin T^{\lfloor n\rfloor},\ulin T^n]\subset {\cal D}_2\}\cap \{\ulin T^n=\ulin t\})\in \F_{\ulin T^n}$.
Similarly, for each $\ulin t\in (2^{-n}\N)^2$, $f^1_{\ulin T^n}$ restricted to $\{\ulin T^n\in {\cal D}_2\}\cap \{\ulin T^n=\ulin t\}$ is measurable w.r.t.\ $\F_{\ulin t}\cap \{\ulin T^n=\ulin t\}\subset \F_{\ulin T^n}$, and so $f^1_{\ulin T^n}$ is $\F_{\ulin T^n}$-measurable on $\{\ulin T^n\in {\cal D}_2\}$.

Since ${\cal D}_2$ is open, by Proposition \ref{right-continuous-0},
$\{\ulin T\in{\cal D}_2\}=\bigcup_{m\ge l}\bigcap_{n\ge m} \{[\ulin T^{\lfloor n\rfloor},\ulin T^n]\subset {\cal D}_2\}\in \F_{\ulin T^l}$ for any $l\in\N$. Thus, $\{\ulin T\in{\cal D}_2\}\in \bigcap_{l} \F_{\ulin T^l}=\lin\F_{\ulin T}$.
Since $f^1_{\ulin t}(z)$ is jointly continuous on ${\cal D}_2\times \HH$,    $f^1_{\ulin T^n}\luto f^1_{\ulin T}$ in $\HH$ on $\{\ulin T\in{\cal D}_2\}$. By Proposition \ref{right-continuous-0}, $f^1_{\ulin T}$ is $\lin\F_{\ulin T}$-measurable.
\end{proof}


\begin{Lemma}
  Let   $\F_n$, $n\in\N$,  and $\F_0$ be sub-$\sigma$-algebras of a probability space $(\Omega,\F,\PP)$ such that $\F_0\subset\bigcap_n \F_n$. Let $E_n\in\F_n$, $n\in\N\cup\{0\}$,   be such that $E_0\subset\liminf E_n$. For $n\in\N\cup\{0\}$, let $X_n(t)$, $t\ge 0$, be a continuous process defined on the event $E_n$ such that a.s.\ on $E_0$, $X_n\to X_0$ pointwise on $ \R_+$. Let $\mu$ be a probability measure on $C(\R_+)$. Suppose for each $n\in\N$, the conditional law of $X_n$ given $\F_n$ and $E_n$ is $\mu$. Then the conditional law of $X_0$ given $\F_0$ and $E_0$ is also $\mu$. \label{Lemma-indep}
\end{Lemma}
\begin{proof} Let $S=\{t_1<\cdots<t_m\}$ be a nonempty finite subset of $\R_+$, and $\mu_S$ be the finite dimensional distribution for $\mu$ on $S$. Then for any $n\in\N$, the conditional law of $(X_n(t_1),\dots,X_n(t_m))$ given $\F_n$ and $E_n$ is $\mu_S$. It suffices to show that the same is true for $X_0$, $\F_0$ and $E_0$. Let $f\in C_b(\R^m,\R)$ 
and $A_0\in \F_0\cap E_0$. Then for every $n\in\N$, $A_0\cap E_n\in \F_n\cap E_n$, and so
$$\EE[{\bf 1}_{A_0\cap E_n} f(X_n(t_1),\dots,X_n(t_m))]=\PP[A_0\cap E_n] \int f d\mu_S.$$
Since $E_0\subset\liminf E_n$, we have ${\bf 1}_{A_0\cap E_n}\to {\bf 1}_{A_0}$. By the continuity of $f$ and that $X_n\to X$, we get $f(X_n(t_1),\dots,X_n(t_m))\to f(X_0(t_1),\dots,X_0(t_m))$.  Letting $n\to \infty$ and applying the dominated convergence theorem, we get
$\EE[{\bf 1}_{A_0 } f(X_0(t_1),\dots,X_0(t_m))]=\PP[A_0] \int f d\mu_S$.
Since this holds for any $f\in C_b(\R^m)$, the conditional law of  $(X_0(t_1),\dots,X_0(t_m))$ given $\F_0$ and $E_0$ is $\mu_S$.
\end{proof}

\begin{proof}[Proof of Lemma \ref{DMP-lem}]
By Lemma \ref{TinD}, $\{\ulin T\in{\cal D}_2\}\in \lin\F_{\ulin T}$ and $f^1_{\ulin T}$ is $\lin\F_{\ulin T}$-measurable. Symmetrically, $\{\ulin T\in{\cal D}_1\}\in \lin\F_{\ulin T}$. So $\{\ulin T\in{\cal D}_{\cap}\}\in \lin\F_{\ulin T}$.
We are going to show that Lemma \ref{DMP-lem} holds with $f_{\ulin T}:=f^1_{\ulin T}=f^{2,J}_{T_2}\circ f^{1,T_2}_{T_1}$. By Lemma \ref{f1t}, on the event $\{\ulin T\in{\cal D}_\cap\}$, $f^1_{\ulin T}$ maps $\HH$ conformally onto $H^1_{\ulin T}$, and sends $0$ and $\infty$ respectively to $\eta_1(t_1)$ and $\eta_2^J(t_2)$.

 We first consider four special cases.  In the first three cases we prove a stronger result, where $\{\ulin T\in{\cal D}_{\cap}\}$ is replaced by a bigger event $\{\ulin T\in{\cal D}_2\}$. 

   \ulin{Case 1.} The $\ulin T$ is a deterministic point $\ulin t=(t_1,t_2)\in\R_+^2$. This is not completely trivial because we have to handle the right-continuous augmentation. We prove Case 1 in two steps.

   \ulin{Step 1.} We prove that the statement holds with $\lin\F$ replaced by $\F$ (without right-continuous augmentation). We give a complete proof for this simpler statement, which also prepares us for the proof in Step 2.  We write $E$ for the event $\{\ulin t\in {\cal D}_2\}$. 

Recall that $\zeta_1^{t_2}$ is a standard chordal SLE$_\kappa$ curve independent of $\F^1_{t_2}$, $u_1^{t_2}$ is an increasing homeomorphism from $[0,\phi_2(t_2)]$ onto $[0,\infty]$, and $\eta_1=f^{2,J}_{t_2}\circ \zeta^1_{t_2}\circ u^1_{t_2}$ on $[0,\phi_2(t_2)]$.
Let $\F^{1,t_2}_t=\sigma(\sigma(\zeta_1^{t_2}(s):s\le t)\cup \F^2_{t_2})$, $t\ge 0$. We extend $u_1^{t_2}$ from  $[0, \phi_2(t_2))$ to $[0,\infty]$ such that it equals $\infty$ on $[\phi_2(t_2),\infty]$.   For any $a\ge 0$, $\{ u_1^{t_2}(t_1)\le a\}=\{\hcap_2( f^{2,J}_{t_2}\circ  \zeta_1^{t_2}[0,a])\ge t\}\in \F^{1,t_2}_a$. Thus, $u_1^{t_2}(t_1)$ is an $\F^{1,t_2}$-stopping time. 
Applying the usual DMP to the standard chordal SLE$_\kappa$ curve $\zeta_1^{t_2}$ at the time $ u_1^{t_2}(t_1)$, we get a standard chordal SLE$_\kappa$ curve $ \zeta^{\ulin t}_1$ defined on the event $E$, which is conditionally independent of $\F^{1,t_2}_{ u _1^{t_2}(t_1)}$, such that $\zeta_1^{t_2}(u_1^{t_2}(t_1)+\cdot)= f^{1,t_2}_{t_1} \circ \zeta^{\ulin t}_1$ a.s\ on $E$. 
Let $u_1^{\ulin t}=(u_1^{t_2})^{-1}(u_1^{t_2}(t_1)+\cdot)$. Then $u_1^{\ulin t}$ is an increasing homeomorphism from $[0,\infty]$ onto $[t_1,\phi_2(t_2)]$. Combining $f^1_{\ulin t}=f^{2,J}_{t_2}\circ  f^{1,t_2}_{t_1}$, $\eta_1= f^{2,J}_{t_2}\circ \zeta_1^{t_2}\circ u_1^{t_2}$ and $ \zeta_1^{t_2}(u_1^{t_2}(t_1)+\cdot)= f^{1,t_2}_{t_1} \circ \zeta^{\ulin t}_1$, we get  $\eta_1\circ u^{\ulin t}_1=f^1_{\ulin t}\circ \zeta^{\ulin t}_1$ a.s.\ on $E$.

Since $\F^2_{t_2}\subset \F^{1,t_2}_{u_1^{t_2}(t_1)}$, and on the event $E$, for $0\le t\le t_1$, $\eta_1(t)= f^{2,J}_{t_2}\circ  \zeta_1^{t_2}\circ u_1^{t_2}(t)$, we have  $\F_{\ulin t}\cap E\subset\F^{1,t_2}_{u_1^{t_2}(t_1)}$. Since given $E$, $\zeta^{\ulin t}_1$ is conditionally independent of $\F^{1,t_2}_{u_1^{t_2}(t_1)}$, it is also conditionally independent of $\F_{\ulin t}$. So we have finished Step 1 of Case 1.

\ulin{Step 2.} We prove that the $\zeta_1^{\ulin t}$ in Step 1 is conditionally independent of $\lin\F_{\ulin t}$. When this is done, the proof of Case 1 is complete. 


Let $\Omega_2=\{( u_1^*(s_1,s_2),s_2):(s_1,s_2)\in{\cal D}_2\}$.
Define  $u_1^{*-1},\zeta_1^*,\ha v_1^* $ on $\Omega_2$ such that for any $s_2\in\R_+$, $u_1^{*-1}(\cdot,s_2)=u_1^*(\cdot,s_2)^{-1}$, $\zeta_1^*(\cdot,s_2)=\eta_1^*(u_1^{*-1}(\cdot,s_2),s_2)$, and $\ha v_1^*(\cdot,s_2)=\ha w_1^*(u_1^{*-1}(\cdot,s_2),s_2)$. Then $\zeta_1^*$ and $\ha v_1^* $ are continuous on $\Omega_2$. For any $s_2\ge 0$,  $\zeta_1^*(\cdot,s_2)$ is the chordal Loewner curve driven by $\ha v_1^*(\cdot,s_2)$, and $\zeta_1^*(\cdot,s_2)$ and $\ha v_1^*(\cdot,s_2)$ are respectively equal to $\zeta_1^{s_2}$ and $\ha v_1^{s_2}$.

 From the construction of $\zeta_1^{\ulin t}(t)$, $t\ge 0$, its driving function, say $\ha v_1^{\ulin t}(t)$, is given by $\ha v_1^{\ulin t}(t):=\ha v_1^{t_2}(u_1^{t_2}(t_1)+t)-\ha v_1^{t_2}(u_1^{t_2}(t_1))$, $t\ge 0$. By discarding a null event, we may assume that $\ha v_1^{\ulin t}= \ha v_1^*(u_1^{*}(t_1,t_2)+\cdot,t_2)-\ha v_1^{*}(u_1^{*}(t_1,t_2),t_2)$. Similarly, we may assume that $u_1^{\ulin t}=u_1^{*-1}(u_1^{*}(t_1,t_2)+\cdot,t_2)$. Then $\ha v_1^{\ulin t}$ and $u_1^{\ulin t}$ are both continuous in $\ulin t$ by continuity of $\ha v_1^*,u_1^*,u_1^{*-1}$.

Choose $\ulin t^n\downarrow \ulin t$. Then $\lin{\F}_{\ulin t}\subset \bigcap_n \F_{\ulin t^n}$. Let $E_n= \{\ulin t^n\in{\cal D}_2\}\in \F_{\ulin t^n}$. Since ${\cal D}_2$ is open, we have $E\subset\liminf E_n$. For any $n\in\N$,   by the result of Step 1,  conditionally on $\F_{\ulin t^n}$ and $E_n$, $\ha v_1^{\ulin t^n}/\sqrt\kappa $ has the law of a standard Brownian motion.  By the continuity of $\ha v^{\ulin t}_1$ in $\ulin t$, $\ha v^{\ulin t^n}\to \ha v^{\ulin t}$ pointwise. By Lemma \ref{Lemma-indep}, $\ha v^{\ulin t}_1$ is conditionally independent of $\lin\F_{\ulin t}$ given $E$. The same is true for $\zeta_1^{\ulin t}$ since it is the chordal Loewner curve determined by $\ha v^{\ulin t}_1$.  This finishes Step 2.

\ulin{Case 2.}  $\ulin T$ takes  values in a countable set $S$. This follows easily from Case 1. 

\ulin{Case 3.} $T_2$ takes values in a countable set $S_2$. For $n\in\N$, define $\ulin T^n=(T^n_1,T_2)$ such that $T_1^n=2^{-n}\lceil 2^n T_1\rceil$. Then each $\ulin T^n$ is an $\lin\F$-stopping time taking values in a countable set. By the results of Cases  1 and 2, for each $n$,  on the event $E_n:=\{\ulin T^n\in{\cal D}_2\}$ there are a chordal Loewner curve $\zeta^{\ulin T^n}_1$, whose driving function is given by $\ha v_1^{\ulin T^n}:=\ha v_1^*(u_1^*(\ulin T^n)+\cdot,T_2)-\ha v_1^*(u_1^*(\ulin T^n),T_2)$, and an increasing homeomorphism $u_1^{\ulin T^n}:=u_1^{*-1}(u_1^*(\ulin T^n)+\cdot,T_2)$ from $[0,\infty]$ onto $[T_1^n,\phi_2(T_2)]$, such that a.s.\ $\eta_1\circ u^{\ulin T^n}_1=f^1_{\ulin T^n}\circ \zeta^{\ulin T^n}_1$ on $E_n$, and $\ha v^{\ulin T^n}_1/\sqrt\kappa $ is a standard Brownian motion conditionally independent of $\lin\F_{\ulin T^n}$ given the event $E_n$. 
Using the same argument as in Step 2 of Case 1, we know that $\ha v_1^{\ulin T}:=\ha v_1^*(u_1^*(\ulin T)+\cdot,T_2)-\ha v_1^*(u_1^*(\ulin T),T_2)$ is  conditionally independent of $\lin\F_{\ulin T}$ with the law of $\sqrt\kappa B$ given $E:=\{\ulin T\in {\cal D}_2\}=\bigcup_n E_n$. Thus, $\ha v_1^{\ulin T}$ a.s.\ generates a chordal Loewner curve $\zeta_1^{\ulin T}$, which has the law of a standard chordal SLE$_\kappa$, and is conditionally independent of $\lin\F_{\ulin T}$ given $E$. Since $\ha v_1^*(\cdot,T_2)$ is the driving function for $\zeta_1^*(\cdot, T_2)$, from the definitions of $\ha v_1^{\ulin T}$ and $f^{1,T_2}_{T_1}$ we know that $\zeta_1^*(u_1^*(\ulin T)+\cdot,T_2)=f^{1,T_2}_{ T_1}\circ \zeta_1^{\ulin T}$. Combining this with $\eta_1^*(\cdot,T_2)=\zeta_1^*(\cdot,T_2)\circ u_1^*(\cdot,T_2)$, $\eta_1=f^{2,J}_{T_2}\circ \eta_1^*(\cdot,T_2)$, $u^{\ulin T}_1=u_1^{*-1}(u_1^*(\ulin T)+t,T_2)$, and $f^1_{\ulin T}=f^{2,J}_{T_2}\circ f^{1,T_2}_{T_1}$, we get $\eta_1\circ u_1^{\ulin T}=f^1_{\ulin T}\circ \zeta_1^{\ulin T}$ a.s.\ on $\{\ulin T\in {\cal D}_2\}$.

\ulin{Case 4.} $T_1$ takes values in a countable set $S_1$. Swapping indices ``$1$'' and ``$2$'' and using the result of Case 3, we know that on the event $\{\ulin T\in{\cal D}_1\}$, there are
  \begin{itemize}
  \item an $\lin\F_{\ulin T}$-measurable conformal map $f^2_{\ulin T}$ on $\HH$   with continuation on $\lin \HH^\#$,
     \item  a standard chordal SLE$_\kappa$ curve $\zeta_2^{\ulin T} $ conditionally independent of $\lin\F_{\ulin T}$, and
    \item an increasing homeomorphism $u_2^{\ulin T}$ from $[0,\infty]$ onto $[T_2,\phi_1(T_1)]$,
\end{itemize}
such that  $\eta_2\circ u_2^{\ulin T} =f^2_{\ulin T} \circ \zeta_2^{\ulin T}$ a.s.\ on  $\{\ulin T\in{\cal D}_1\}$. Moreover, on the event $E_\cap:=\{\ulin T\in{\cal D}_{\cap}\}$, $f^2_{\ulin T}$ maps $\HH$ conformally onto the connected component $D^2_{\ulin T}$ of $\HH\sem (\eta_2[0,T_2]\cup J\circ \eta_1[0,T_1])$ which shares the prime ends $\eta_2(T_2)$ and $J\circ \eta_1(T_1)$ respectively with $\HH\sem K^2_{T_2}$ and $\HH\sem J(K^1_{T_1})$, and sends $0$ and $\infty$ respectively to the prime ends $\eta_2(T_2)$ and $J\circ \eta_1(T_1)$.

Suppose $E_\cap$ happens.
 We have $J(D^2_{\ulin T})=D^1_{\ulin T}$ since they are both a connected component of $\HH\sem (\eta_1[0,T_1]\cup \eta_2^J[0,T_2])$ that shares the prime ends $\eta_1(T_1)$ and $\eta_2^J(T_2)$ respectively with $\HH\sem K^1_{T_1}$ and $\HH\sem K^{2,J}_{T_2}$. Let  $f^{2,J}_{\ulin T}=J\circ f^2_{\ulin T}\circ J$. Then both $f^{2,J}_{\ulin T}$ and $f^1_{\ulin T}$ map $\HH$ conformally onto $J( D^2_{\ulin T})=D^1_{\ulin T}$,  send $0$ and $\infty$ respectively to the same prime ends $\eta_1(T_1)$ and $  \eta_2^J(T_2)$, and  are $\lin\F_{\ulin T}$-measurable. So there is an $\lin\F_{\ulin T}$-measurable random variable $a>0$ such that $f^{2,J}_{\ulin T}=f^1_{\ulin T}(\cdot/a)$.

Applying the reversibility of SLE to $\zeta^{\ulin T}_2$, we get another standard chordal SLE$_\kappa$ curve $\til \zeta^{\ulin T}_1$, which is also conditionally independent of $\lin\F_{\ulin T}$, and a decreasing auto-homeomorphism $\ha \phi_2$ of $[0,\infty]$, such that a.s.\ $J\circ \zeta^{\ulin T}_2=\til \zeta^{\ulin T}_1\circ \ha \phi_2$. Let $u_1^{\ulin T}(t)= \phi_2\circ u^{\ulin T}_2\circ \ha \phi_2^{-1}(a^2 t)$. Then  $u_1^{\ulin T}$ is an increasing homeomorphism from $[0,\infty]$ onto $[T_1,\phi_2(T_2)]$.
 Let $\zeta_1^{\ulin T}(t)=\til \zeta_1^{\ulin T}(a^2 t)/a$. Since $a$ is $\lin\F_{\ulin T}$-measurable, by the scaling property of chordal SLE$_\kappa$,  $\zeta_1^{\ulin T}$ is also a standard chordal SLE$_\kappa$ curve conditionally independent of $\lin\F_{\ulin T}$. Moreover, on $\{\ulin T\in{\cal D}_{\cap}\}$, a.s.\
 $$J\circ \eta_1\circ u_1^{\ulin T}=\eta_2\circ \phi_2^{-1}\circ u_1^{\ulin T}= \eta_2\circ   u^{\ulin T}_2\circ \ha \phi_2^{-1}(a^2 \cdot)= f^2_{\ulin T} \circ \zeta_2^{\ulin T}\circ \ha \phi_2^{-1}(a^2 \cdot)$$
 $$=f^{2}_{\ulin T} \circ J  \circ \til\zeta^{\ulin T}_1(a^2 \cdot) =f^{2}_{\ulin T} \circ J (a \zeta^{\ulin T}_1)=J\circ f^{2,J}_{\ulin T} (a \zeta^{\ulin T}_1)=J\circ f^1_{\ulin T}\circ \zeta_1^{\ulin T},$$  which implies that $\eta_1\circ u_1^{\ulin T}= f^1_{\ulin T}\circ \zeta_1^{\ulin T}$.   This completes Case 4.



\ulin{General Case.}  For $n\in\N$, define $\ulin T^n=(T^n_1,T_2)$ such that $T_1^n=2^{-n}\lceil 2^n T_1\rceil$.
Then each $\ulin T^n$ is an $\lin\F$-stopping time with the first variable taking values in a countable set. Let $E_0=\{\ulin T\in{\cal D}_{\cap}\}$ and $E_n=\{\ulin T^n\in{\cal D}_{\cap}\}$, $n\in\N$. By Case 4, for each $n$, on the event $E_n$, there are a standard chordal SLE$_\kappa$ curve $\zeta_1^{\ulin T^n}$, which is conditionally independent of $\lin\F_{\ulin T^n}$ given $E_n$, and an increasing homeomorphism $u_1^{\ulin T^n}$ from $[0,\infty]$ onto $[T_1^n,\phi_2(T_2)]$, such that $\eta_1\circ  u_1^{\ulin T^n}= f^1_{\ulin T^n}\circ  \zeta_1^{\ulin T^n}$ a.s.\ on $E_n$.

Suppose  $E_n$ happens. Then $T_1^n<\phi_2(T_2)$. Let $u_1^n=(u_1^{\ulin T^n})^{-1}(\cdot+T_1^n)$, which is an increasing homeomorphism from $[0,\phi_2(T_2)-T_1^n]$ onto $[0,\infty]$. Let $\ha v^{\ulin T^n}_1$ be the driving function for $\zeta^{\ulin T^n}_1$.
Let $\eta_1^{\ulin T^n}= \zeta_1^{\ulin T^n}\circ u_1^{n}$ and $\ha w_1^{\ulin T^n}= \ha v_1^{\ulin T^n}\circ u_1^{n}$. Then $\eta_1^{\ulin T^n}$ is the chordal Loewner curve with speed $du_1^n$ driven by $\ha w_1^{\ulin T^n}$. From a.s.\ $\eta_1\circ  u_1^{\ulin T^n}= f^1_{\ulin T^n}\circ  \zeta_1^{\ulin T^n}$  we get a.s.\ $\eta_1(T_1^n+\cdot)=f^1_{\ulin T^n}\circ \eta_1^{\ulin T^n}$.

Define $\eta_1^{T_2,n},u_1^{T_2,n},\ha w_1^{T_2,n}$ on $[0,\phi_2(T_2))$ such that, for $0\le t_1< T_1^n $, $\eta_1^{T_2,n}(t_1)=\eta_1^*(t_1,T_2)$, $u_1^{T_2,n}(t_1)=u_1^*(t_1,T_2)$, and $\ha w_1^{T_2,n}(t_1)=\ha w_1^*(t_1,T_2)$; and for $T_1^n \le t_1<\phi_2(T_2)$, $\eta_1^{T_2,n}(t_1)=f^{1,T_2}_{T_1^n}(\eta_1^{\ulin T^n}(t_1-T_1^n))$, $u_1^{T_2,n}(t_1)=u_1^*(\ulin T^n)+u_1^n(t_1-T_1^n)$, and $\ha w_1^{T_2,n}(t_1)=\ha w_1^*(\ulin T^n)+\ha w_1^{\ulin T^n}(t_1-T_1^n)$. Since $\eta_1^*(\cdot,t_2)$ is the chordal Loewner curve with speed $du_1^*(\cdot,t_2)$ driven by $\ha w_1^*(\cdot,t_2)$, $f^{1,T_2}_{T_1^n}$ is the centered Loewner map for $\eta_1^*(\cdot,t_2)$ at the time $T_1^n$, and $\eta_1^{\ulin T^n}$ is the chordal Loewner curve with speed $d u_1^n$ driven by $\ha w_1^{\ulin T^n}$, we see that  $\eta_1^{T_2,n}$ is the chordal Loewner curve with speed $du_1^{T_2,n}$ driven by $\ha w_1^{T_2,n}$.
From a.s.\ $\eta_1 =f^{2,J}_{T_2}\circ \eta_1^*(\cdot,T_2)$, $\eta_1(T_1^n+\cdot)= f^1_{\ulin T^n}\circ  \eta_1^{\ulin T^n}$, and $f^1_{\ulin T^n}=f^{2,J}_{T_2}\circ f^{1,T_2}_{T_1^n}$, we get $\eta_1(t)=f^{2,J}_{T_2}\circ \eta_1^{T_2,n}(t)$, $0\le t<\phi_2(T_2)$, a.s.\ on $E_n$.

Let $\zeta_1^{T_2,n}=\eta_1^{T_2,n}\circ (u_1^{T_2,n})^{-1}$ and $\ha v_1^{T_2,n}=\ha w_1^{T_2,n}\circ (u_1^{T_2,n})^{-1}$. Then $\zeta_1^{T_2,n}$ is the chordal Loewner curve driven by $\ha v_1^{T_2,n}$, and $\eta_1|_{[0,\phi_2(T_2))}=f^{2,J}_{T_2}\circ \zeta_1^{T_2,n}\circ u_1^{T_2,n}$ a.s.\ on $E_n$. The last equality implies that a.s.\ $\zeta_1^{T_2,n}=\zeta_1^{T_2,m}$ and $u_1^{T_2,n}=u_1^{T_2,m}$ on $E_n\cap E_m$ for any $n,m\in\N$.
Since $E_0\subset \liminf E_n$, we may define a Loewner curve $\zeta_1^{T_2}$ and a time-change function $u_1^{T_2}$ on $E_0$ such that for every $n\in\N$,  $\zeta_1^{T_2}=\zeta_1^{T_2,n}$ and $u_1^{T_2}=u_1^{T_2,n}$ a.s.\ on $ E_0\cap E_n$.  Then we have $\eta_1|_{[0,\phi_2(T_2))}=f^{2,J}_{T_2}\circ \zeta_1^{T_2}\circ u_1^{T_2}$ a.s.\ on $E_0$.  Let $\ha v_1^{T_2}$ be the driving function for $\zeta_1^{T_2}$. Then $v_1^{T_2}=v_1^{T_2,n}$ a.s.\ on $ E_0\cap E_n$.

By the definition of $u_1^{T_2,n}$, $(u_1^{T_2,n})^{-1}(u_1^*(\ulin T^n)+t_1)=T_1^n+(u_1^n)^{-1}(t_1)$. By the definition of $\ha w_1^{T_2,n}$, $\ha w_1^{T_2,n}(T_1^n+(u_1^n)^{-1}(t_1))=\ha w_1^*(\ulin T^n)+\ha w_1^{\lin T^n}\circ (u_1^n)^{-1}(t_1)$. By these two equalities and the definition of $\ha v_1^{T_2,n}$ and that $\ha w_1^{\lin T^n}=\ha v_1^{\lin T^n}\circ u_1^n$ we get $\ha v_1^{\ulin T^n}= \ha v_1^{T_2,n}(u_1^*(\ulin T^n)+\cdot)-v_1^{T_2,n}(u_1^*(\ulin T^n) ) $.  From the independence property of $\zeta_1^{\ulin T^n}$, we see that, conditionally on $\lin\F_{\ulin T^n}$ and the event $E_n$, $\ha v^{\ulin T^n}_1$ has the law of $(\sqrt\kappa B_t)$. It is clear that, a.s.\ on $E_0$, $\ha v^{\ulin T^n}_1\to \ha v^{\ulin T}_1:=\ha v_1^{T_2}( u_1^*(\ulin T)+\cdot)-\ha v_1^{T_2}(u_1^*(\ulin T))$ pointwise on $ \R_+$ since $\ha v_1^{T_2,n}=\ha v_1^{T_2}$ and $ u_1^{T_2,n}=\ha u_1^{T_2}$ for $n$ big enough. By Lemma \ref{Lemma-indep},   conditionally on $\lin\F_{\ulin T}$ and the event $E_0$, $\ha v^{\ulin T}_1 $ also has the law of $(\sqrt\kappa B_t)$. Thus, $\ha v^{\ulin T}_1$ generates a chordal Loewner curve $\zeta_1^{\ulin T}$ defined on $E_0$, which has the law of a standard chordal SLE$_\kappa$ conditionally on $\lin\F_{\ulin T}$ and $E_0$. Since $(f^{1,T_2}_t)_{t\ge 0}$ are the centered Loewner maps for $\eta_1^{T_2}=\zeta_1^{T_2}\circ u_1^*(\cdot,T_2)$, we have $\zeta_1^{T_2}(u_1^*(\ulin T)+\cdot)=f^{1,T_2}_{T_1}\circ \zeta_1^{\ulin T}$. Let $u_1^{\ulin T}(s)=(u_1^{T_2})^{-1}(s+u_1^*(\ulin T))$, $s\ge 0$. Then $u_1^{\ulin T}$ is an increasing homeomorphism from $[0,\infty]$ onto $[T_1,\phi_2(T_2)]$. Since $\eta_1=f^{2,J}_{T_2}\circ \zeta_1^{T_2}\circ u_1^{T_2}$ a.s.\ on $E_0$ and $f^1_{\ulin T}=f^{2,J}_{T_2}\circ f^{1,T_2}_{T_1}$, we then get $\eta_1\circ u_1^{\ulin T}=f^1_{\ulin T}\circ \zeta_1^{\ulin T}$ a.s.\ on $E_0$, as desired.
\end{proof}

\end{appendices}


\begin{thebibliography}{00}
\bibitem{Ahl} Lars V.\ Ahlfors. {\it Conformal invariants: topics in geometric function theory}. McGraw-Hill Book Co., New York, 1973.
\bibitem{Bf} V.\ Beffara.  The dimension of SLE curves, {\it Ann.\ Probab.}, {\bf 36}:1421-1452, 2008.
\bibitem{Julien} Julien Dub\'edat. Commutation relations for SLE, {\it Comm.\ Pure Applied Math.}, {\bf 60}(12):1792-1847, 2007.
\bibitem{Brownian-cut} Nina Holden, Gregory F. Lawler, Xinyi Li and Xin Sun. Minkowski content of Brownian cut points. In preprint, arXiv:1803.10613.
	\bibitem{Law-SLE} Gregory Lawler. {\em Conformally Invariant Processes in the Plane}, Amer. Math. Soc, 2005.
\bibitem{LR} Gregory F.\ Lawler and Mohammad A.\ Rezaei. Minkowski content and natural parametrization for the Schramm-Loewner evolution. {\it Ann.\ Probab.}, {\bf 43}(3):1082-1120, 2015.
\bibitem{LSW} G.\ Lawler. O.\ Schramm, and W.\ Werner.  Conformal invariance of planar loop-erased random walks and uniform spanning trees, {\it Ann.\  Probab.}, {\bf 32}:939--995, 2004.
\bibitem{LW} G.\ Lawler and B.\ Werness. Multi-point Green's function for SLE and an estimate of Beffara, {\it Annals of Prob.} {\bf 41}:1513-1555, 2013.
	\bibitem{MS3} Jason Miller and Scott Sheffield. Imaginary Geometry III: reversibility of SLE$_\kappa$ for $\kappa\in (4, 8)$. {\it Ann.\ Math.}, {\bf 184}(2):455-486, 2016.
	\bibitem{MS1} Jason Miller and Scott Sheffield. Imaginary Geometry I: intersecting SLEs. {\it Probab.\ Theory Relat.\ Fields}, {\bf 164}(3):553-705, 2016.
\bibitem{MW} Jason Miller and Hao Wu. Intersections of SLE Paths: the double and cut point dimension of SLE. {\it Prob.\ Theory Relat.\ Fields}, {\bf 167}(1-2):45-105, 2017.
	\bibitem{NIST:DLMF} NIST Digital Library of Mathematical Functions. \url{http://dlmf.nist.gov/18}, Release 1.0.6 of 2013-05-06.
\bibitem{Pom} Pommerenke Ch. Boundary behaviour of conformal maps. Springer-Verlag, Berlin, Heidelberg, 1992
\bibitem{RS} Steffen Rohde and Oded Schramm. Basic properties of SLE. {\it Ann.\  Math.}, {\bf 161}:879-920, 2005.
\bibitem{RY} Daniel Revuz and Marc Yor. {\it Continuous Martingales and Brownian Motion}. Springer, Berlin, 1991.
\bibitem{S-SLE} O.\ Schramm. Scaling limits of loop-erased random walks and uniform spanning trees. {\it Israel J.\ Math.}, {\bf 118}:221-288, 2000.
	\bibitem{SW} Oded Schramm and David B.\ Wilson. SLE coordinate changes. {\it New York J.\ Math.}, {\bf 11}:659--669, 2005.
\bibitem{orthogonal-2} Yuan Xu. Lecture notes on orthogonal polynomials of several variables. Inzell Lectures on Orthogonal Polynomials.
	W. zu Castell, F. Filbir, B. Forster (eds.). Advances in the Theory of Special Functions and Orthogonal Polynomials.  Nova Science Publishers
	Volume 2, 2004, Pages 135-188.
\bibitem{Two-Green-boundary} Dapeng Zhan. Two-curve Green's function for $2$-SLE: the boundary case, {\it Electron.\ J.\ Probab.}, {\bf 26}, article no.\ 32:1-58, 2021.
\bibitem{Two-Green-interior} Dapeng Zhan. Two-curve Green's function for $2$-SLE: the interior case.  {\it Comm.\ Math.\ Phys.}, {\bf 375}:1-40, 2020.
	\bibitem{tip} Dapeng Zhan. Ergodicity of the tip of an SLE curve. {\it Prob.\ Theory Relat.\ Fields}, {\bf 164}(1):333-360, 2016.
\bibitem{duality2} Dapeng Zhan. Duality of chordal SLE, II. {\it Ann.\ I.\ H.\ Poincare-Pr.}, {\bf 46}(3):740-759, 2010.
	\bibitem{duality} Dapeng Zhan. Duality of chordal SLE.  {\it Invent.\ Math.}, {\bf 174}(2):309-353, 2008.
	\bibitem{reversibility} Dapeng Zhan. Reversibility of chordal SLE. {\it Ann.\ Probab.}, {\bf 36}(4):1472-1494, 2008.
\bibitem{LERW} Dapeng Zhan. The Scaling Limits of Planar LERW in Finitely Connected Domains. {\it Ann.\ Probab.} {\bf 36}:467-529, 2008.
\end{thebibliography}
\end{document}